\documentclass{conm-p-l}
\usepackage{txfonts}
\usepackage{amscd,amssymb,mathabx,subfigure}
\usepackage{graphicx,calrsfs}
\usepackage[colorlinks,plainpages,backref,urlcolor=blue,breaklinks]{hyperref}
\usepackage{tikz}
\usetikzlibrary{cd}
\usetikzlibrary{calc}
\usepackage{mdwlist}
\usepackage{enumitem} 
\usepackage{xurl}

\numberwithin{equation}{section}

\newtheorem{theorem}{Theorem}[section]
\newtheorem{corollary}[theorem]{Corollary}
\newtheorem{lemma}[theorem]{Lemma}
\newtheorem{prop}[theorem]{Proposition}
\newtheorem{proposition}[theorem]{Proposition}

\theoremstyle{definition}
\newtheorem{definition}[theorem]{Definition}
\newtheorem{example}[theorem]{Example}
\newtheorem{remark}[theorem]{Remark}
\newtheorem{question}[theorem]{Question}

\AtBeginEnvironment{example}{%
  \pushQED{\qed}%
}
\AtEndEnvironment{example}{\popQED\endexample}

\newcommand{\N}{\mathbb{N}}
\newcommand{\Z}{\mathbb{Z}}
\newcommand{\Q}{\mathbb{Q}}
\newcommand{\R}{\mathbb{R}}
\newcommand{\C}{\mathbb{C}}

\newcommand{\T}{\mathbb{T}}
\newcommand{\TT}{\mathbb{T}}
\renewcommand{\P}{\mathbb{P}}
\newcommand{\CP}{\mathbb{CP}}

\renewcommand{\k}{\Bbbk}

\newcommand{\RR}{\mathcal{R}}
\newcommand{\V}{\mathcal{V}}
\newcommand{\VV}{\mathcal{V}}
\newcommand{\A}{{\mathcal{A}}}
\newcommand{\BB}{{\mathcal{B}}}

\newcommand{\HH}{{\mathfrak{h}}}
\newcommand{\h}{{\mathfrak{h}}}
\newcommand{\B}{{\mathfrak{B}}}

\newcommand{\bo}{{\mathbf 1}}
\newcommand{\bz}{{\mathbf 0}}
\newcommand{\bm}{{\mathbf{m}}}

\newcommand{\sE}{\mathsf{E}}
\newcommand{\sV}{\mathsf{V}}
\newcommand{\sT}{\mathsf{T}}

\DeclareMathOperator{\rank}{rank}

\DeclareMathOperator{\gr}{gr}
\DeclareMathOperator{\im}{im}
\DeclareMathOperator{\coker}{coker}
\DeclareMathOperator{\codim}{codim}

\DeclareMathOperator{\id}{id}
\DeclareMathOperator{\ab}{{ab}}
\DeclareMathOperator{\abf}{{abf}}

\DeclareMathOperator{\Sym}{Sym}

\DeclareMathOperator{\Hom}{{Hom}}
\DeclareMathOperator{\Tor}{{Tor}}

\DeclareMathOperator{\Hilb}{{Hilb}}

\DeclareMathOperator{\LL}{\Lie}
\DeclareMathOperator{\Lie}{Lie}
\DeclareMathOperator{\Spec}{{Spec}}

\DeclareMathOperator{\Aut}{Aut}

\DeclareMathOperator{\Tors}{Tors}
\DeclareMathOperator{\TC}{TC}

\DeclareMathOperator{\loc}{loc}

\DeclareMathOperator{\supp}{{supp}}

\DeclareMathOperator{\ii}{i}

\newcommand{\surj}{\twoheadrightarrow}
\newcommand{\inj}{\hookrightarrow}
\newcommand{\isom}{\xrightarrow{\,\simeq\,}}

\newcommand{\abs}[1]{\left| #1 \right|}

\def\set#1{{\{ #1\}}}

\newcommand{\bwedge}{\mbox{\normalsize $\bigwedge$}}

\definecolor{lime}{HTML}{A6CE39}
\DeclareRobustCommand{\orcidicon}{
	\begin{tikzpicture}
	\draw[lime, fill=lime] (0,0) 
	circle [radius=0.15] 
	node[white] {{\fontfamily{qag}\selectfont \tiny ID}};
	\draw[white, fill=white] (-0.0625,0.095) 
	circle [radius=0.007];
	\end{tikzpicture}
	\hspace{-2mm}
}
\foreach \x in {A, ..., Z}{\expandafter\xdef\csname orcid\x\endcsname{\noexpand\href{%
https://orcid.org/\csname orcidauthor\x\endcsname}{\noexpand\orcidicon}}}

\usepackage{xpatch}
\makeatletter   
\xpatchcmd{\@tocline}
{\hfil\hbox to\@pnumwidth{\@tocpagenum{#7}}\par}
{\ifnum#1<0\hfill\else\dotfill\fi\hbox to\@pnumwidth{\@tocpagenum{#7}}\par}
{}{}
\makeatother     

\makeatletter
 \def\l@subsection{\@tocline{2}{0pt}{4pc}{6pc}{}}
\def\l@subsubsection{\@tocline{3}{0pt}{8pc}{8pc}{}}
 \makeatother
 
\begin{document}

\title[Topology and combinatorics of decomposable arrangements]{%
On the topology and combinatorics of \\ decomposable arrangements} 

\author[Alexander~I.~Suciu]{Alexander~I.~Suciu$^1$\!\!\orcidA{}}
\address{Department of Mathematics,
Northeastern University,
Boston, MA 02115, USA}
\email{\href{mailto:a.suciu@northeastern.edu}{a.suciu@northeastern.edu}}
\urladdr{\href{https://suciu.sites.northeastern.edu}%
{https://suciu.sites.northeastern.edu}}
\thanks{$^1$Supported in part by Simons Foundation Collaboration 
Grants for Mathematicians \#693825}

\subjclass[2020]{Primary
52C35. 
Secondary
16W70, 
17B70,  
20F14,  
20F40,  
32S55,  
57M07.  
}

\keywords{Hyperplane arrangement, decomposable arrangement, 
intersection lattice, fundamental group, holonomy Lie algebra, 
associated graded Lie algebra, Chen ranks, Alexander invariant, 
cohomology jump loci, Milnor fibration}

\dedicatory{To Enrique Artal Bartolo, on his 60th birthday}

\begin{abstract}
We study topological properties of complex hyperplane 
arrangements that are decomposable. When this purely 
combinatorial condition is satisfied, it is known that 
the associated graded Lie algebra of the arrangement 
group $G$ decomposes (in degrees greater than $1$) 
as a direct product of free Lie algebras. 
It follows that the $I$-adic completion of the Alexander invariant 
$B(G)$ also decomposes as a direct sum of ``local" invariants 
and the Chen ranks of $G$ are the sums of the local 
contributions. Moreover, if $B(G)$ is separated, then the degree 
$1$ cohomology jump loci of the arrangement complement 
have only local components, and the algebraic monodromy 
of the Milnor fibration is trivial in degree~$1$. 
\end{abstract}
\maketitle

\tableofcontents
\section{Introduction}
\label{sect:intro}

\subsection{Hyperplane arrangements}
\label{subsec:arr-intro}
An arrangement of hyperplanes is a finite collection of codimension-$1$ 
linear subspaces in a finite-dimensional, complex vector space. 
One of the main goals of the subject is to decide whether a given 
topological invariant of the complement $M=M(\A)$ is combinatorially 
determined, and, if so, to express it explicitly in terms of 
the intersection lattice $L(\A)$. 

At one extreme are the Betti numbers $b_q(M)$, which may be 
computed solely in terms of the M\"{o}bius function $\mu\colon L(\A)\to \Z$, 
and the cohomology ring $H^*(M;\Z)$, which is 
a graded algebra with degree-$1$ generators 
indexed by the hyperplanes, subject to relations defined in terms of $L(\A)$. 
At the other extreme is the fundamental group of the complement, 
$G=G(\A)$, which also admits a presentation with generators indexed 
by $\A$ and with as many com\-mutator-relators as $b_2(M)$, 
yet is not always combinatorially determined. 
Other invariants fall somewhere in between. For instance, 
the LCS ranks $\phi_k(G)$ and the Chen ranks $\theta_k(G)$
of an arrangement group $G$  (see \S\ref{subsec:lcs-alex-intro}) 
are combinatorially determined, yet the torsion in the LCS quotients is not. 
Moreover, the degree-$1$ characteristic varieties $\VV_s(M)$ 
(see \S\ref{subsec:cjl-intro}) are finite unions of subtori that are 
combinatorially determined, and also translated subtori whose 
combinatorial status is still largely unknown.

In this paper, we narrow the lens, and focus on a combinatorially-defined 
class of arrangements for which many of the aforementioned difficulties 
disappear, yet several unresolved questions still remain. Following \cite{PS-imrn04}, 
we say that an arrangement $\A$ is {\em decomposable}\/ if the degree-$3$ 
part of the holonomy Lie algebra $\h(\A)$ consists only of contributions 
coming from the flats in $L_2(\A)$, see \S\ref{subsec:holo-intro}.
When this condition is satisfied, 
the associated graded Lie algebra of $G$ decomposes (in degrees 
greater than $1$) as a direct product of free Lie algebras determined 
by the M\"{o}bius function of $L(\A)$ \cite{PS-imrn04}, and all the nilpotent 
quotients of $G$ are also combinatorially determined \cite{PrS20}. 

The main part of our analysis concerns the {\em Alexander invariant}, 
$B(G)=G'/G''$, viewed as a module over the group ring of 
$G_{\ab}=G/G'$ and endowed with the topology defined by the 
filtration by powers of the augmentation ideal $I$. Pursuing 
work started in \cite{CS-tams99}, we show here the following: 
If $\A$ is decomposable over $\Q$, then the $I$-adic 
completion of $B(G)\otimes \Q$ also decomposes 
as a direct sum of ``local" invariants,  
and the Chen ranks of $G$ are the sums of the local 
contributions. Furthermore, we show that $B(G)$ is decomposable 
if and only if it is separated in the $I$-adic topology and $\A$ 
is decomposable. If this is the case (upon tensoring with $\Q$), 
then the characteristic variety $\V_1(M(\A))$ has only components 
arising from the $2$-flats of $\A$, and the algebraic monodromy 
of the Milnor fibration is trivial in degree $1$. 

\subsection{Lie algebras and Alexander invariants}
\label{subsec:lcs-alex-intro}
To describe in more detail our work, we start by defining the Lie algebras 
that come into play here. Given a group $G$, its lower central series (LCS) 
is defined inductively by setting $\gamma_1 (G)=G$ and 
$\gamma_{k+1}(G) =[G,\gamma_k (G)]$ for $k\ge 1$. 
This series is both normal and central; therefore, 
its successive quotients, $\gr_k(G)= \gamma_k(G)/\gamma_{k+1}(G)$, 
are abelian groups. 
The {\em associated graded Lie algebra}, $\gr(G)$, 
is the direct sum of the groups $\gr_k(G)$, with Lie bracket (compatible 
with the grading) induced from the group commutator. The 
{\em Chen Lie algebra}, $\gr(G/G'')$, is simply the  
associated graded Lie algebra of the maximal 
metabelian quotient of $G$. 

Assume now that $G$ is finitely generated. The LCS quotients of $G$ are then 
also finitely generated; we let $\phi_k(G)\coloneqq \rank(\gr_k(G))$ be the 
ranks of these abelian groups and we let $\theta_k(G)\coloneqq \rank(\gr_k(G/G''))$ 
be the Chen ranks of $G$. As shown by Massey in \cite{Ms-80}, the $I$-adic filtration 
on $B(G)$ coincides with the LCS-filtration on $G/G''$, after a shift of $2$; therefore, 
$\gr_k(B(G))=\gr_{k+2}(G/G'')$ for all $k\ge 0$.

The {\em holonomy Lie algebra}, $\h(G)$ is a finitely presented, quadratic 
Lie algebra built solely in terms of cohomological data associated 
to the group $G$. It is defined as the quotient of the free Lie algebra on 
the free abelian group $G_{\abf}=G_{\ab}/\Tors$ by the Lie ideal generated 
by the image of the dual of the cup-product map $H^1(G)\wedge H^1(G)\to H^2(G)$. 
The Lie algebra $\h(G)$ is a graded Lie algebra that maps 
surjectively to $\gr(G)$; moreover, its maximal metabelian quotient, 
$\h(G)/\h(G)''$, maps surjectively to $\gr(G/G'')$. 
Following \cite{PS-imrn04}, we use the holonomy Lie algebra  
to construct an infinitesimal version of the Alexander invariant, 
$\B(G)=\h(G)'/\h(G)''$, which is a graded module over the symmetric 
algebra $\Sym(G_{\abf})$.

If the group $G$ is $1$-formal---which is the case when $G=G(\A)$ is an 
arrangement group---these seemingly disparate strands tie together much 
more tightly, at least over the rationals. For instance, the map 
$\h(G)\otimes \Q\to \gr(G)\otimes \Q$ is an isomorphism \cite{Sullivan}, 
and so is the map $\h(G)/\h(G)''\otimes \Q\to \gr(G/G'')\otimes \Q$ \cite{PS-imrn04}, 
while the completions of $B(G)$ and $\B(G)$ are isomorphic, 
after tensoring with $\Q$ \cite{DPS-duke}.

\subsection{Holonomy, localization, and decomposability}
\label{subsec:holo-intro}
For an arrangement $\A$, the holonomy Lie algebra $\h(\A)=\h(G(\A))$ 
depends only on the truncated Orlik--Solomon algebra $H^{\le 2} (M(\A);\Z)$, 
and thus, only on the intersection poset $L_{\le 2}(\A)$. An explicit presentation 
for $\h(\A)$ was first given by Kohno in \cite{Kohno-83}.

The localization of $\A$ at a flat $X\in L_2(\A)$ is the 
sub-arrangement $\A_X$ consisting of those hyperplanes that contain $X$. 
The inclusion $\A_X \inj \A$ 
induces a map between complements, $M(\A)\inj M(\A)$, which in turn induces 
a (split) surjection from $G(\A)$ to $G(\A_X) \cong F_{\mu(X)}\times \Z$. 
We obtain in this fashion an epimorphism $\h(\A)\surj \h(\A_X)$. 
As shown in \cite{PS-imrn04}, these maps assemble 
into a morphism of graded Lie algebras, 
\begin{equation}
\label{eq:h-hloc}
\begin{tikzcd}[column sep=20pt]
\h(\A) \ar[r]& \h(\A)^{\loc}\coloneqq \prod_{X\in L_2(\A)} \h(\A_X),
\end{tikzcd}
\end{equation}
which is a surjection in degrees $k\ge 3$ and an isomorphism in degree $k=2$. 
The arrangement is said to be decomposable if $h_3(\A)\cong \h(\A)_3^{\loc}$, 
in which case the maps $\h'(\A)\to \h'(\A)^{\loc}$ and  $\h(G)\to \gr(G)$ are 
isomorphisms, and $\h(\A)$ is torsion-free. 

The decomposability property in inherited by sub-arrangements (\cite{PS-imrn04}) and 
is preserved under taking products of arrangements (Proposition \ref{prop:prod}). 
Graphic arrangements are decomposable precisely when the corresponding graph 
contains no $4$-cliques (\cite{PS-imrn04}), while split-solvable arrangements are always 
decomposable (Corollary \ref{cor:pen-decomp}). 

\subsection{Alexander invariants of arrangements}
\label{subsec:alex-intro}
Given a hyperplane arrangement $\A$, we let 
$B(\A)= G(\A)'/G(\A)''$ be its Alexander invariant, 
viewed as a module over the Laurent polynomial ring $R=\Z[G(\A)_{\ab}]$. 
The structure of this module, and that of its associated graded module, $\gr(B(G))$, 
holds rich and varied information regarding the Chen ranks of $G(\A)$, the 
characteristic and resonance varieties of $M(\A)$, and the algebraic 
monodromy of the Milnor fibration of $\A$.

For each $2$-flat $X\in L_2(\A)$, the homomorphism $G(\A)\surj G(\A_X)$ induces  
an epimorphism $B(G(\A))\surj B(G(\A_X))$. 
Letting $B(\A)^{\loc}$ be the direct sum of all the ``local" Alexander invariants 
$B(G(\A_X))$, viewed as $R$-modules by restriction of scalars, we obtain 
a map of $R$-modules, $B(\A)\to B(\A)^{\loc}$. 

In a similar fashion, we let $\B(\A)=\h(\A)'/\h(\A)''$ be the infinitesimal Alexander 
invariant of the arrangement, viewed as a module over the polynomial ring $S=\gr(R)$. 
The Lie algebra map $\h(\A) \to \h(\A)^{\loc}$ from \eqref{eq:h-hloc} induces a map 
of graded $S$-modules, $\B(\A)\to \B(\A)^{\loc}$.  In Theorem \ref{thm:b-bloc-q}, 
we prove the following.

\begin{theorem}
\label{thm:bbloc-intro}
For any hyperplane arrangement $\A$, the morphisms 
$\B(\A) \to \B(\A)^{\loc}$ and $B(\A)\otimes \Q\to B(\A)^{\loc}\otimes \Q$ 
are surjective.
\end{theorem}

As an application, we recover the following lower bound on the Chen ranks of 
arrangement groups, first established in \cite{CS-tams99} by other methods:
\begin{equation}
\label{eq:chen-bound-intro}
\theta_k(G(\A)) \ge 
(k-1) \sum_{X\in L_2(\A)} \binom{\mu(X)+k-2}{k} ,
\end{equation}
for all $k\ge 2$, with equality for $k=2$.

\subsection{Decomposable Alexander invariants}
\label{subsec:dec-intro}

We say that the Alexander invariant of an arrangement $\A$ decomposes 
if the canonical map $B(\A)\to B(\A)^{\loc}$ is an isomorphism. A similar 
definition was first made in \cite[\S6.4]{CS-tams99} in regards to the 
$I$-adic completion of this map, $\widehat{B(\A)}\to \widehat{B(\A)}{}^{\loc}$. 
In the same spirit, we say that the infinitesimal Alexander invariant decomposes 
if the map $\B(\A)\to \B(\A)^{\loc}$ is an isomorphism. In all three cases, 
analogous definitions work over the rationals. 

A natural question arises: What is the relationship between the decomposability 
of $\A$---a purely combinatorial notion that depends only on $L_{\le 2}(\A)$---and 
that of $B(\A)$---a notion that depends {\it a priori}\/ on the topology of $M(\A)$?
Analogous questions may be raised about the decomposability of $B(\A)$ over $\Q$, 
as well as the decomposability of $\widehat{B(\A)}$ and $\B(\A)$. 

As shown in Theorem \ref{thm:alex-decomp}, the topological decomposability 
notion implies the combinatorial one, that is, if $B(\A)$ decomposes then $\A$ 
decomposes (and likewise over $\Q$). The converse, though, is much more subtle. 
A first step in this direction was done in \cite[Theorem~7.9]{CS-tams99}, where it was shown 
that $\widehat{B(\A)}$ decomposes if $\A$ does. We give in 
Theorem \ref{thm:decomp-alex} 
a completely different proof of this result, albeit over the rationals. 
To ascertain the decomposability of $B(\A)$ itself, 
it remains to decide whether the Alexander invariant $B(\A)$ is separated in the 
$I$-adic topology, or, equivalently, whether the metabelian group $G/G''$ is 
residually nilpotent. This separability condition holds in many examples, 
but it is an open question whether it holds for all decomposable arrangements.

We summarize our results in Corollary \ref{cor:dec-sep}, as follows.

\begin{theorem}
\label{thm:alex-dec-intro}
Let $\A$ be a hyperplane arrangement.
\begin{enumerate}[itemsep=2pt, topsep=-1pt]
\item  \label{bdx1} 
$\B(\A)$ is decomposable (over $\Q$) if and only if $\A$ is decomposable (over $\Q$).
\item  \label{bdx2} 
$B(\A)$ is decomposable (over $\Q$) if and only if 
$\A$ is decomposable and $B(\A)$ is separated (over $\Q$). 
\end{enumerate}
\end{theorem}

The proofs of all these implications are self-contained, except 
for the integral version of the reverse implication in \eqref{bdx2}, 
which relies on the aforementioned result from \cite{CS-tams99}. 
As a consequence of part \eqref{bdx1}, we show in Corollary \ref{eq:chen-decomp} 
that the lower bound for the Chen ranks from \eqref{eq:chen-bound-intro} 
holds as an equality for $\Q$-decomposable arrangements. This 
recovers results from \cite{CS-tams99, PS-cmh06} in a slightly 
stronger form. 

\subsection{Cohomology jump loci and Milnor fibrations}
\label{subsec:cjl-intro}
We conclude with an analysis of the {\em characteristic varieties} 
(the jump loci for homology in rank $1$ local systems) and the 
{\em resonance varieties} (the jump loci of the Koszul complex 
associated to the cohomology algebra) of the complement 
of a decomposable arrangement.

Let $\A=\{H_1,\dots ,H_n\}$ be a hyperplane arrangement in $\C^{d+1}$. 
Its complement, $M=M(\A)$, is a smooth, complex, quasi-projective variety. 
Hence, by a general result of Arapura \cite{Ar}, its  
(degree~$1$) characteristic varieties, $\V_s(M)$, 
are finite unions of torsion-translates of algebraic subtori 
of the character group $\Hom(\pi_1(M),\C^*)=(\C^*)^n$.  
Since $M$ is also a formal space, its resonance varieties, $\RR_s(M)$,  
coincide with the tangent cone at the trivial character to $\V_s(M)$, 
see \cite{CS99, DPS-duke}.  
As shown in \cite{FY}, the resonance varieties of $\A$ 
may be described solely in terms of multinets on sub-arrangements 
of $\A$.  In general, though, $\V_1(M)$ may contain components 
that do not pass through the origin, see \cite{Su02, CDS, DeS-plms}. 

For each rank-$2$ flat with $\mu(X)>1$, consider the linear subspace
$L_X = \big\{ x \in \C^{n} : \sum_{H_i\in \A_X} x_i =0 \ \text{and $x_i= 0$ if 
$H_i\notin \A_X$} \big\}$, and let $T_X=\exp(L_X)\subset (\C^*)^{n}$ be 
the corresponding algebraic subtorus.  In Theorem \ref{thm:decomp-cjl}, 
we prove the following. 

\begin{theorem}
\label{thm:cjl-dec-intro}
Let $\A$ be a $\Q$-decomposable arrangement. For each $s\ge 1$, 
\begin{enumerate}[itemsep=3pt, topsep=2pt]
\item  \label{cjx1} 
$\RR_s(M)=\bigcup_{\substack{X \in L_2(\A)\\ \mu(X)>s}}   L_X$.
\item  \label{cjx2} 
If $B(\A)\otimes \Q$ is separated, then 
$\V_s(M)=\bigcup_{\substack{X \in L_2(\A)\\ \mu(X)>s}}T_X$.
\end{enumerate}
\end{theorem}

For each hyperplane $H\in \A$, let $f_H\colon \C^{d+1} \to \C$ be a linear form 
with kernel $H$.  Assigning a multiplicity vector $\bm=\{m_H\}_{H\in \A}\in \N^{n}$ 
to the hyperplanes, we obtain a polynomial map, 
$f_{\bm}=\prod_{H\in \A} f_H^{m_H}\colon \C^{d+1} \to \C$, 
whose restriction to the complement, $f_{\bm}\colon M(\A) \to \C^{*}$,  
is the projection map of a smooth, locally trivial bundle, 
known as the {\em Milnor fibration}\/ of the multi-arrangement $(\A,\bm)$. 
Let $F_{\bm}$ be the typical fiber and let $h\colon F_{\bm}\to F_{\bm}$ 
be the monodromy of the fibration. A much-studied problem is to compute 
the first Betti number of $F_{\bm}$ and find the eigenvalues of the 
algebraic monodromy acting on $H_1(F_{\bm};\C)$; see for instance 
\cite{Artal, CS95, DeS-plms, PS-plms17}. As an application of the 
previous theorem, we prove in Theorem \ref{thm:decomp-mono} the following. 

\begin{theorem}
\label{thm:milnor-dec-intro}
Let $\A$ be an arrangement of rank $3$ or higher. Suppose 
$\A$ is $\Q$-decomposable and $B(\A)\otimes \Q$ is separated. Then, 
for any choice of multiplicities $\bm$ on $\A$, the algebraic monodromy 
of the Milnor fibration, $h_*\colon H_1(F_{\bm};\Q)\to H_1(F_{\bm};\Q)$, is trivial.
\end{theorem}

For an in-depth study of the topology of the Milnor fibers of hyperplane arrangements  
with trivial algebraic monodromy, we refer to \cite{Su-mono}.

\subsection{Structure of the paper}
\label{subsec:organize}
Roughly speaking, the paper is divided into three parts. The first part covers the 
general theory of Alexander invariants (\S\ref{sect:alexinv}), lower central series, 
associated graded Lie algebras, and Chen Lie algebras (\S\ref{sect:lcs-chen}), 
concluding with holonomy Lie algebras and infinitesimal Alexander invariants (\S\ref{sect:holo}). 

The second part deals with some basic notions regarding hyperplane arrangements.  
In \S\ref{sect:hyp-arr} we discuss the combinatorics of an arrangement,  
as it relates to the topology of its complement, while in \S\ref{sect:holo-arr} 
and  \S\ref{sect:alex-arr} we analyze in detail the holonomy Lie algebras 
and the Alexander invariants of arrangements, respectively.

The third part concentrates on decomposable arrangements. It starts with  
their basic properties (\S\ref{sect:decomp}) and some constructions 
of such arrangements (\S\ref{sect:ex-decomp}). It then 
continues with arrangements whose Alexander invariants decompose 
(\S\ref{sect:alex-decomp}), and concludes with a description of their 
cohomology jump loci (\S\ref{sect:cjl-decomp}) and Milnor fibrations  
(\S\ref{sect:milnor-decomp}).

\section{Alexander invariants}
\label{sect:alexinv}

We start with a detailed overview of the Alexander invariant of a group, 
with an emphasis on its separation and functoriality properties. Along the 
way, we prove a number of technical results that will be needed later on. 

\subsection{Derived series}
\label{subsec:derived}

Let $G$ be a group. If $H$ and $K$ are subgroups of $G$, then 
$[H,K]$ denotes the subgroup of $G$ generated by all elements 
of the form $[a,b]=aba^{-1}b^{-1}$ with $a \in H$ and $b \in K$. 
If both $H$ and $K$ are normal subgroups, then their commutator 
$[H,K]$ is again a normal subgroup; moreover, if 
$\alpha \colon G\to H$ is a homomorphism, then 
$\alpha([H,K])\subseteq [\alpha(H),\alpha(K)]$. 

The {\em derived series}\/ of $G$ is 
defined inductively by  $G^{(r)}=[G^{(r-1)},G^{(r-1)}]$, 
starting with $G^{(0)}=G$. 
In particular, $G^{(1)}=G'$ is the derived subgroup and  $G^{(2)}=G''$. 
The terms of these series are fully invariant subgroups; that is, if 
$\alpha \colon G\to H$ is a group homomorphism, then 
$\alpha(G^{(r)})\subseteq H^{(r)}$, for all $r$. 
Consequently, the derived series is a normal series, i.e., $G^{(r)}\triangleleft G$, 
for all $r$. Moreover, since $G^{(r-1)}/G^{(r)}$ is the abelianization of $G^{(r-1)}$, 
all the successive quotients of the series are abelian groups. 

A group $G$ is said to be {\em solvable}\/ if its derived series of $G$ terminates 
in finitely many steps; that is, $G^{(\ell)}=\{1\}$ for some integer $\ell\ge 0$.  
The smallest such integer, $\ell(G)$, is then called the derived length of $G$. 
Clearly, $\ell(G)\le 1$ if and only if $G$ is abelian, while 
$\ell(G)\le 2$  if and only if $G$ is metabelian.  The maximal  solvable quotient of 
$G$ of length $r$ is $G/G^{(r)}$; in particular, the maximal abelian quotient 
is $G/G'$ and the maximal metabelian quotient is $G/G''$. 

\subsection{Alexander invariant}
\label{subsec:alexinv}

Since the group $G_{\ab}=G/G'$ is commutative, the group-ring $R=\Z[G_{\ab}]$ 
is also commutative. If $G_{\ab}$ is finitely generated, then the 
ring $R$ is Noetherian; if, moreover, $G_{\ab}$ 
is torsion-free, then $R$ is a Noetherian domain. 

Among the successive quotients of the derives series of a group $G$, the second one 
plays a special role. The {\em Alexander invariant}\/ of $G$ is the abelian group 
\begin{equation}
\label{eq:gprime}
B(G)\coloneqq   G'/G''\, , 
\end{equation}
viewed as a module over the group-ring $\Z[G_{\ab}]$; alternatively, 
$B(G)=(G')_{\ab}=H_1(G';\Z)$.  Addition in $B(G)$ 
is induced from multiplication in $G$ via $(xG'')+(yG'')=xy G''$ for 
$x,y\in G'$, while scalar multiplication is induced from conjugation in the 
maximal metabelian quotient, $G/G''$, via the exact sequence 
\begin{equation}
\label{eq:gprimeprime}
\begin{tikzcd}[column sep=20pt]
1\ar[r]& G'/G'' \ar[r]& G/G'' \ar[r]& G/G' \ar[r]& 1\, .
\end{tikzcd} 
\end{equation}
That is, $gG'\cdot xG'' = gxg^{-1}G''$ for $g\in G$, $x\in G'$, 
with the action of $G/G'=G_{\ab}$ 
extended $\Z$-linearly to the whole of $\Z[G_{\ab}]$. 
This action is well-defined, since $g\in G'$ implies $gxg^{-1}x^{-1}\in G''$, 
and so $gxg^{-1}G''=xG''$.

The above construction is functorial. Indeed, let $\alpha\colon G\to H$ be 
a group homomorphism. Then $\alpha$ extends linearly to a ring map,  
$\tilde\alpha \colon \Z [G]\to \Z [H]$. The map $\alpha$ also restricts 
to homomorphisms $\alpha'\colon G'\to H'$ and $\alpha''\colon G''\to H''$, 
and thus induces homomorphisms $G/G'\to H/H'$ and $G'/G''\to H'/H''$, 
which we denote by $\alpha_{\ab}\colon G_{\ab}\to H_{\ab}$ and 
$B(\alpha) \colon B(G) \to B(H)$, respectively.  The map 
$B(\alpha) \colon B(G) \to B(H)$ can then be interpreted as a 
morphism of modules covering the ring map 
$\tilde\alpha_{\ab} \colon \Z [G_{\ab}]\to \Z [H_{\ab}]$; that is, 
\begin{equation}
\label{eq:B-alpha}
B(\alpha)(rm) = \tilde\alpha_{\ab} (r) \cdot B(\alpha)(m)
\end{equation}
for all $r\in \Z [G_{\ab}]$ and $m\in B(G)$. 
Clearly, if $\alpha'\colon G'\to H'$ is surjective, then 
$B(\alpha)\colon B(G)\to B(H)$ is also surjective, and 
if $\alpha'$ is an isomorphism, then 
$B(\alpha)$ is also an isomorphism. 
In particular, if $\alpha$ is surjective, then $\alpha'$ is surjective, 
and so $B(\alpha)$ is also surjective. Nevertheless, if $\alpha$ 
is injective, $B(\alpha)$ need not be injective.  

\begin{remark}
\label{rem:factor}
Given a homomorphism $\alpha\colon G\to H$, 
let $B(H)_{\alpha}$ be the $\Z [G_{\ab}]$-module obtained 
from $B(H)$ by restriction of scalars via the ring map 
$\tilde\alpha_{\ab} \colon \Z [G_{\ab}]\to \Z [H_{\ab}]$. 
Concretely, $B(H)_{\alpha} =B(H)$ as abelian groups, 
with module structure given by $g\cdot m=\alpha(g)m$ 
for $g\in G_{\ab}$ and $m\in B(G)_{\alpha}$. The map 
$B(\alpha)\colon B(G)\to B(H)$ can then be viewed as 
the composite 
\begin{equation}
\label{eq:bg-bh}
\begin{tikzcd}[column sep=20pt]
B(G) \ar[r]& B(H)_{\alpha} \ar[r]& B(H),
\end{tikzcd} 
\end{equation}
where the first arrow is a $\Z [G_{\ab}]$-linear map and the second 
arrow is the identity map of $B(H)$, viewed as covering the ring map 
$\tilde\alpha_{\ab}$.  
\end{remark}

The next lemma gives a formula expressing the Alexander invariant of a product 
of two groups in terms of the Alexander invariants of the factors. Another formula 
of this sort, involving extension of scalars instead of restriction of scalars, 
is given in \cite[Proposition~1.8]{CS-tams99}.

\begin{lemma}
\label{lem:alex-prod}
Let $G=G_1\times G_2$ be a product of two groups, and let $p_i\colon G\surj G_i$ 
be the projections to the factors. Then $B(G) \cong B(G_1)_{p_1} \oplus B(G_2)_{p_2}$. 
\end{lemma}

\begin{proof}
Each projection map yields an epimorphism $B(p_i)\colon B(G)\surj B(G_i)$ 
covering the ring map $\tilde{p}_i\colon \Z[G_{\ab}]\surj \Z[(G_i)_{\ab}]$. By the 
remarks above, we get $\Z[G_{\ab}]$-epimorphisms $q_i\colon B(G)\surj B(G_i)_{p_i}$. 
Note that $G'=G_1'\times G_2'$, and thus also $G''= G_1''\times G_2''$. 
For an element $x=(x_1,x_2)\in G'$, the maps $q_i$ take the coset 
$x G''\in B(G)$ to $x_i G_i'' \in B(G_i)_{p_i}$, where the $\Z[G_{\ab}]$-module 
structure on $B(G_i)_{p_i}$ is given by $g \cdot m =g_i m$ 
for $g=(g_1,g_2)\in G$.  
It follows that the map $q\colon B(G) \to B(G_1)_{p_1} \oplus B(G_2)_{p_2}$, 
$(x_1,x_2) G'' \mapsto (x_1 G_1'', x_2 G_2'')$ is an isomorphism of $\Z[G_{\ab}]$-modules, 
and we are done.
\end{proof}

\subsection{Topological interpretation}
\label{subsec:top-int}
The Alexander invariant of a group admits the following topological interpretation. 
Let $X$ be a connected CW-complex with fundamental group $\pi_1(X,x_0)=G$. 
(Without loss of generality, we may assume $X$ has a single $0$-cell, which 
we then take as the basepoint $x_0$.)  Lifting the cell structure of $X$ to the 
maximal abelian cover, $q\colon X^{\ab}\to X$, we obtain an augmented chain 
complex of free $\Z[G_{\ab}]$-modules, 
\begin{equation}
\label{eq:abcover-cc}
\begin{tikzcd}[column sep=18pt]
\cdots \ar[r] & 
C_{2}(X^{\ab};\Z) \ar[r, "\partial^{\ab}_{2}(X)"] &[14pt] 
 C_{1}(X^{\ab};\Z) \ar[r, "\partial^{\ab}_{1}(X)"] &[14pt] 
  C_{0}(X^{\ab};\Z) \ar[r, "\varepsilon"] & \Z  \ar[r] &0, 
\end{tikzcd}
\end{equation}
where $C_{k}(X^{\ab};\Z) = C_k(X;\Z)\otimes \Z[G_{\ab}]$ and 
$\varepsilon\colon \Z[G_{\ab}]\to \Z$ is the augmentation map.
Since $\pi_1(X^{\ab})=G'$, the Alexander invariant $B(G)=(G')_{\ab}$ 
is isomorphic to $H_1(X^{\ab};\Z)$, the first homology group of the chain 
complex \eqref{eq:abcover-cc}, with module structure induced 
by the action of $G_{\ab}$ on $X^{\ab}$ by deck transformations; 
equivalently, $B(G)=H_1(X;\Z[G_{\ab}])$.

\begin{example}
\label{ex:alex-free}
Let $X=\bigvee^n S^1$ be a wedge of $n$ circles. Identify $\pi_1(X)$ 
with the free group $F_n=\langle x_1,\dots ,x_n\rangle$, its abelianization 
$(F_n)_{\ab}$ with $\Z^n$, and the group ring  $\Z[\Z^n]$ with the ring 
of Laurent polynomials $R=\Z[t_1^{\pm 1}, \dots , t_n^{\pm 1}]$. 
The chain complex of the universal (abelian) cover of the $n$-torus 
$T^n=K(\Z^n,1)$ may be viewed as the Koszul complex on 
$t_1-1, \dots, t_n-1$ over $R$, with differentials 
$\partial_k^{\ab}=\partial_k^{\ab}(T^n)$ given by $\partial^{\ab}_k (e_J)=
\sum_{i=1}^{k}  (-1)^{i-1} (t_i-1)  e_{ J\setminus \{ j_i \} }$, 
where $e_J=e_{j_1} \wedge \cdots \wedge e_{j_k}$ for a $k$-tuple 
$J=\{j_1,\dots, j_k\}$. The Alexander invariant of $F_n$, then, 
may be identified with the $R$-module $B(F_n)=\ker(\partial_1^{\ab})$. 
From the exactness of the Koszul complex, it follows that $B(F_n)$ has 
(finite) presentation
\begin{equation}
\label{eq:bfn}
\begin{tikzcd}[column sep=20pt]
\bwedge^3 \Z^n \otimes R \ar[r, "\partial_3^{\ab}"] &[14pt]
\bwedge^2 \Z^n \otimes R\ar[r] & B(F_n) \ar[r] &0 \, . 
\end{tikzcd}
\hfill \qedhere
\end{equation}
\end{example}

This example leads us to a more general result.

\begin{lemma}
\label{lem:b-fg}
If the group $G$ is finitely generated, then the Alexander invariant 
$B(G)$ is a finitely presented module over the Noetherian ring 
$\Z[G_{\ab}]$.
\end{lemma}

\begin{proof}
Since $G$ is finitely generated (say, with $n$ generators), 
its abelianization $G_{\ab}$ is also finitely generated (by at 
most $n$ elements). Therefore, the group-ring $\Z[G_{\ab}]$ 
is Noetherian. Moreover, there is an epimorphism 
$\alpha\colon F_n\surj G$, which induces an epimorphism 
$B(\alpha)\colon B(F_n)\surj B(G)$. By \eqref{eq:bfn}, 
the $\Z[\Z^n]$-module $B(F_n)$ is generated by $\binom{n}{2}$ 
elements. It follows 
that the Alexander invariant $B(G)$ is also finitely generated 
(in fact, by at most $\binom{n}{2}$ elements), and hence it is 
finitely presented as a module over the Noetherian ring $\Z[G_{\ab}]$.
\end{proof}

For a finitely presented group $G$, 
a finite presentation for the Alexander invariant 
may be found via the Fox differential calculus \cite{Fox}. To start with, 
if $G=\langle  x_1,\dots ,x_{m}\mid r_1,\dots ,r_{s}\rangle$ 
is a finite presentation and $K_G$ is the corresponding presentation 
$2$-complex, then $\partial^{\ab}_2(K_G)\colon \Z[G_{\ab}]^s \to \Z[G_{\ab}]^m$, 
the second boundary map in the chain complex  \eqref{eq:abcover-cc}, 
coincides with the Alexander matrix $\big(\!\ab( \partial r_i/\partial x_k) \big)$ of 
abelianized Fox derivatives of the relators. When $G_{\ab}$ is torsion-free, 
a method for finding a presentation for $B(G)$ is outlined in \cite{Ms-80}.
We illustrate this method with an example that will be needed later on. 

\begin{example}
\label{ex:alex-pres}
By a classical result of R.~Lyndon (see \cite{Fox}), if $v_1, \dots ,v_n$ are 
elements of the ring $\Z[\Z^n]=\Z[t_1^{\pm 1},\dots ,t_n^{\pm 1}]$ that satisfy 
the equation $\sum_{k=1}^n (t_k -1) v_k=0$, then there exists an element 
$r\in F_n'$ such that $v_k= \ab(\partial r/\partial x_k)$, for $1\le k \le n$. 
Therefore, if $f=f(t_1,\dots, t_n)\in \Z[\Z^n]$, then, for each $1\le i<j\le n$, we 
may find an element $r_{i,j}\in F_n'$ such that $\ab(\partial r_{i,j}/\partial x_k)$ 
is equal to $(t_i-1)f$ if $k=i$ and $(1-t_j)f$ if $k=j$, and is equal to $0$, otherwise. 
Now consider the group $G=\langle x_1, \dots, x_n \mid r_{ij}\ (1\le i<j\le n) \rangle$. 
Clearly, $G_{\ab}=\Z^n$. We define a chain map from the chain 
complex of $(K_G)^{\ab}$ to that of $(T^n)^{\ab}$, as follows:
\begin{equation}
\label{eq:bgf}
\begin{tikzcd}[column sep=30pt]
&[-12pt]&\bwedge^2 \Z^n \otimes R\ar[r, "\partial_2^{\ab}(G)"] \ar[d, "\id\otimes f"]
&[4pt] \Z^n\otimes R  \ar[r, "\partial_1^{\ab}"]  \ar[equal]{d} & R\phantom{.}  \ar[equal]{d}
\\
\cdots\ar[r]&[-10pt] \bwedge^3 \Z^n\otimes R \ar[r, "\partial_3^{\ab}"] &
\bwedge^2 \Z^n\otimes R\ar[r, "\partial_2^{\ab}"] 
& \Z^n\otimes R  \ar[r, "\partial_1^{\ab}"] & R .
\end{tikzcd}
\end{equation}
By definition, $B(G)=\ker(\partial_1^{\ab})/\im(\partial_2^{\ab}(G))$. 
Since the Koszul complex on the bottom is exact, a diagram chase yields 
the presentation $B(G)=\coker (\partial_3^{\ab} +\id\otimes f)$. In particular, 
when $n=2$, we have that $B(G)=R/(f)$.
\end{example}

Finally, suppose $f\colon X\to Y$ is a map between 
connected CW-complexes; without loss of generality, 
we may assume $f$ is cellular and basepoint-preserving. 
Let $f_{\sharp}\colon \pi_1(X, x_0)\to \pi_1(Y,y_0)$ 
be the induced homomorphism on fundamental groups, and let 
$f^{\ab}\colon X^{\ab} \to Y^{\ab}$ be the lift to universal abelian 
covers. Then the morphism 
$B(f_{\sharp})\colon B(\pi_1(X))\to B(\pi_1(Y))$ 
coincides with the induced homomorphism in first homology, 
$f^{\ab}_*\colon H_1(X^{\ab};\Z) \to H_1(Y^{\ab};\Z)$, 
and covers the ring map $\tilde{f}_*\colon \Z[H_1(X;\Z)]\to \Z[H_1(Y;\Z)]$. 

\subsection{The $I$-adic completion of $\Z[G_{\ab}]$}
\label{subsec:completion-r}

Let $I=I(G_{\ab})$ be the augmentation ideal of the group-ring $R=\Z[G_{\ab}]$, 
that is, the kernel of the ring map $\varepsilon\colon \Z[G_{\ab}]\to \Z$ given by 
$\varepsilon \big(\sum n_g g\big)=\sum n_g$. As an abelian group, $I$ is freely 
generated by the elements $g-1$ with $g\ne 1$, while its $k$-th power, $I^k$, 
is generated by the $k$-fold products of such elements. 

The powers of the augmentation ideal form a descending, multiplicative 
filtration $R\supset I \supset I^2 \supset \cdots$. This filtration 
defines a topology on $R$, making it into 
a topological ring in which the ideals $I^k$ 
for a basis of open neighborhoods of $0$. 
Let $\widehat{R}=\varprojlim\nolimits_{k} R/I^k$ 
be the completion of $R$ with respect to the $I$-adic topology, 
and let $\gr_I(R)=\bigoplus_{n\ge 0} I^k/I^{k+1}$ be the associated 
graded object. Both $\widehat{R}$ and $\gr_I(R)$ 
acquire in a natural way a ring structure, which is compatible 
with the filtration by the powers of the ideal $\widehat{I}$ (the closure 
of $I$ in $\widehat{R}$), respectively, the grading. It follows that 
$\gr_I(R)=\gr_{\widehat{I}}(\widehat{R})$ is a graded ring. 
Moreover, $\widehat{R}$ is a flat $R$-module. 

When endowed with the $\widehat{I}$-adic 
topology, $\widehat{R}$ is also a topological ring, and the canonical 
map to the completion, $\iota_R\colon R\to \widehat{R}$, is a morphism 
in this category. The injectivity of $\iota_R$ is equivalent to the 
$I$-adic topology on $R$ being Hausdorff; in turn, this is 
equivalent to $\bigcap_{k\ge 1} I^k=\{0\}$, that is, $0$ being 
a closed point.

\begin{example}
\label{ex:laurent}
If $G=\Z^n$, a choice of basis identifies the ring $R=\Z[\Z^n]$ with the 
ring of Laurent polynomials $\Z[t_1^{\pm 1},\dots , t_n^{\pm  1}]$, and the 
ideal $I=I(\Z^n)$ with the maximal ideal $(t_1-1,\dots, t_n-1)$. 
Therefore, $\widehat{R}$ may be identified with the ring of power 
series $\Z[[x_1,\dots, x_n]]$, so that the map $\iota_R$ takes 
$t_i$ to $x_i+1$, while $\gr(R)$ is the polynomial ring
$\Z[x_1,\dots, x_n]$.
\end{example}

\begin{lemma}
\label{lem:faithful}
If $G_{\ab}$ is finitely generated, then the ring $R=\Z[G_{\ab}]$ is Noetherian 
and the map $\iota_R\colon R\to \widehat{R}$ is injective.
\end{lemma}

\begin{proof}
If $G_{\ab}=\Z^n$, then both assertions are well-known. 
If $G_{\ab}=\Z_r$, then $R=\Z[t]/(t^r-1)$ and the assertions 
are easily verified. The general case readily follows.
 \end{proof}

Now let $\alpha\colon G\to H$ be a group homomorphism, let 
$\alpha_{\ab}\colon G_{\ab}\to H_{\ab}$ be its abelianization, and 
let $\tilde\alpha_{\ab}$ be its linear extension to a ring 
morphism from $R=\Z[G_{\ab}]$ to $S=\Z[H_{\ab}]$. 

\begin{lemma}
\label{lem:surj}  
Suppose $\alpha_{\ab}\colon G_{\ab} \to H_{\ab}$ is surjective.
Then the ring maps 
$\tilde\alpha_{\ab}\colon R\to S$, 
$\hat{\tilde\alpha}_{\ab}\colon \hat{R}\to \hat{S}$, and 
$\gr(\tilde\alpha_{\ab})\colon \gr (R)\to \gr(S)$ are all surjective.
\end{lemma}

\begin{proof}
Since $\tilde\alpha_{\ab}(\sum n_g g)=\sum n_g \alpha_{\ab}(g)$, 
the surjectivity of $\tilde\alpha_{\ab}$ is obvious. 
It follows that $\tilde\alpha_{\ab}$ maps the ideal $I=I(G_{\ab})$ 
onto the ideal $J=I(H_{\ab})$. Therefore, $\tilde\alpha_{\ab}$  
induces surjections $R/I^k\surj S/J^k$ for all $k\ge 0$, which yields 
the surjectivity of $\hat{\tilde\alpha}_{\ab}$, and surjections 
$I^k/I^{k+1}\surj J^k/J^{k+1}$ for all $k\ge 0$, which proves the 
surjectivity of $\gr(\tilde\alpha_{\ab})$. 
\end{proof}

\begin{lemma}
\label{lem:inj}
Suppose $H_{\ab}$ is finitely generated and the map 
$\alpha_{\ab}\colon G_{\ab} \to H_{\ab}$ is injective.  
Then the maps 
$\tilde\alpha_{\ab}\colon R\to S$, 
$\hat{\tilde\alpha}_{\ab}\colon \hat{R}\to \hat{S}$, and 
$\gr(\tilde\alpha_{\ab})\colon \gr (R)\to \gr(S)$ are all  injective.
\end{lemma}

\begin{proof}
The injectivity of $\gr(\tilde\alpha_{\ab})$ is proved in \cite[Lemma 6.4]{Su-abexact}. 
To establish the injectivity of $\hat{\tilde\alpha}_{\ab}$, 
it is enough to show that the maps $(\tilde{\alpha}_{\ab})_n\colon R/I^k \to S/J^k$ 
are injective, for all $k\ge 0$. We do this by induction, starting at $k=0$, 
when this is obvious.  The induction step follows from the injectivity of 
$\gr_n(\tilde\alpha_{\ab})$, together with the Snake Lemma applied 
to the commuting diagram
\begin{equation}
\label{eq:snake}
\begin{tikzcd}[column sep=22pt, row sep=22pt]
0\ar[r]  & R/I^{k} \ar[r] \ar[d, "(\tilde{\alpha}_{\ab})_k"] 
& R/I^{k+1} \ar[r]  \ar[d, "(\tilde{\alpha}_{\ab})_{k+1}"] 
& \gr_{k}(R) \ar[r]\ar[d, "\gr_k(\tilde\alpha)"] &0\phantom{.} \\
0\ar[r]  & S/J^k \ar[r]  & S/J^{k+1}\ar[r] & \gr_{k}(S) \ar[r] &0 .
\end{tikzcd}
\end{equation}

To prove the last claim, first observe that our hypotheses imply that 
$G_{\ab}$ is finitely generated; hence, by Lemma \ref{lem:faithful}, the ideal 
$\ker(\iota_R)= \bigcap_{k\ge 1} I^k$ is trivial. Therefore, since 
$\hat{\tilde\alpha}_{\ab}$ is injective, we conclude that 
$\tilde\alpha_{\ab}$ is injective, too.
\end{proof}

\subsection{Completion and associated graded of $B(G)$}
\label{subsec:completion}

Consider now the Alexander invariant $B=B(G)$ of a group $G$, viewed 
as a module over the ring $R=\Z[G_{\ab}]$. The augmentation ideal 
$I=I(G_{\ab})$ defines a descending filtration 
$B\supset IB \supset I^2 B\supset \cdots$, which in turn defines 
the $I$-adic topology on $B$, making it into a topological 
module over the topological ring $R$. We let 
\begin{equation}
\label{eq:bhat}
\widehat{B}=\varprojlim\nolimits_{k} B/I^{k}B
\end{equation}
be the $I$-adic completion of the Alexander invariant, 
and view it as a module over $\widehat{R}$.
Furthermore, we let 
\begin{equation}
\label{eq:grb}
\gr(B)=\bigoplus_{k\ge 0} I^{k}B/I^{k+1}B
\end{equation}
be the associated graded of the Alexander invariant,  
viewed as a (graded) module over the ring $\gr(R)$. 
Note that $\gr(B)$ is generated as a $\gr(R)$-module 
by its degree $0$ piece, $\gr_0(B)=B/IB$.

If the group $G$ is finitely generated, then, by Lemma \ref{lem:b-fg}, 
the Alexander invariant $B=B(G)$ is a finitely presented module over 
the Noetherian ring $R$. Therefore, the canonical map to the completion, 
$\iota_B\colon B\to \widehat{B}$, induces an isomorphism 
$B\otimes_{R}\widehat{R} \to \widehat{B}$, see \cite[Proposition~3.13]{AM}. 
Moreover, by the exactness property of completion (see e.g. \cite[Proposition.~10.12]{AM}), 
the following holds: If $R^m\xrightarrow{\varphi} R^n \to B\to 0$ is a finite 
presentation for $B$, then $\widehat{B}$ has presentation 
$\widehat{R}^m\xrightarrow{\hat{\varphi}} \widehat{R}^n \to B\to 0$. 

\subsection{Separated Alexander invariants}
\label{subsec:sep-alex}
It is readily seen that the kernel of the morphism $\iota_B\colon B\to \widehat{B}$ 
is equal to $\bigcap_{k\ge 1} I^k B$, 
the closure of $\{0\}$ in the $I$-adic topology on $B$. By Krull's Theorem, 
$\ker(\iota_B)$  consists of all the elements of $B$ that are annihilated 
by $1+I$, see \cite[Theorem~10.17]{AM}. 

We  say that the Alexander invariant $B=B(G)$ is {\em ($I$-adically) separated}\/ 
if the morphism $\iota_B$ is injective, that is, the topological module $B$ is a 
Hausdorff space. Here is an example where this happens.  

\begin{example}
\label{ex:bhat-inj}
Let $F_n$ be the free group of rank $n$. As in Example \ref{ex:laurent}, we identify 
the ring $R=\Z[\Z^n]$ with $\Z[t_1^{\pm 1},\dots, t_n^{\pm 1}]$  and $\widehat{R}$ with 
$\Z[[x_1,\dots x_n]]$, so that $\iota_R(t_i)=x_i+1$. From Example \ref{ex:alex-free}, 
we know that the Alexander invariant $B=B(F_n)$ is the kernel of the 
$R$-linear map $\partial_1^{\ab}\colon R^n \to R$ given by $\partial_1^{\ab}(e_{i} )=t_i-1$. 
By exactness of completion, the $\widehat{R}$-module $\widehat{B}$ 
is equal to $\ker (\widehat{\partial}_1^{\ab})$, where 
$\widehat{\partial}_1^{\ab} (e_{i})=x_i$. Since the map 
$\iota_R^n \colon R^n \to \widehat{R}^n$ is injective 
(see Lemma \ref{lem:faithful}), it 
follows that its restriction to the Alexander invariant, 
$\iota_B\colon B\to \widehat{B}$, is also injective.
\end{example}

In general, though, the Alexander invariant of a finitely generated group 
(even a commu\-tator-relators group) is not separated. In fact, as the next 
example shows, $\ker(\iota_B)$ may be equal to $B$, even when $B\ne 0$.

\begin{example}
\label{ex:bhat-not-inj}
Let $f=f(t_1,t_2)$ be a Laurent polynomial in $R=\Z[\Z^2]$ 
and let $G=\langle x_1,x_2 \mid r\rangle $ be the corresponding 
commutator-relator group, constructed in Example \ref{ex:alex-pres}, 
so that $B=R/(f)$. Assume that $\varepsilon (f)=1$, yet $f$ is not a 
monomial in $t_1,t_2$ (for instance, take $f=2-t_1$). 
Then $f-1$ belongs to the ideal $I=I(\Z^2)$ and $B\ne 0$. 
On the other hand, $f$ is invertible in $\widehat{R}$, and 
so $\widehat{B}=\widehat{R}/(f)=0$; alternatively, 
note that $IB=B$, and so $\ker(\iota_B)=\bigcap_{k\ge 1} I^k B=B$. 
\end{example}

As the next lemma shows, the notion of separability behaves well 
with respect to (finite) direct products of groups. 

\begin{lemma}
\label{lem:prod-sep}
Let $G=G_1\times G_2$ be the product of two groups. Then $B(G)$ is 
separated if and only if both $B(G_1)$ and $B(G_2)$ are separated.
\end{lemma}

\begin{proof}
By Lemma \ref{lem:alex-prod}, the $\Z[G_{\ab}]$-module $B(G)$ is 
isomorphic to $B(G_1)_{p_1} \oplus B(G_2)_{p_2}$, where 
$p_i\colon G\surj G_i$ are the projections maps. 
Observe that $I^k B(G_i)_{p_i} = (I_i)^k B(G_i)$
for all $k\ge 1$, where $I$ and $I_i$ are the augmentation ideals of $\Z[G_{\ab}]$ and 
$\Z[(G_i)_{\ab}]$, respectively. Therefore, $B(G_i)_{p_i}$ is separated if and and only if 
$B(G_i)$ is separated. The claim now follows from the fact that completion commutes 
with direct sums. 
\end{proof}

From Example \ref{ex:bhat-inj} and Lemma \ref{lem:prod-sep}, 
we obtain the following immediate corollary.

\begin{corollary}
\label{cor:prod-free}
Let $G=F_{n_1}\times \cdots \times F_{n_r}$ be a finite direct product 
of finitely generated free groups. Then $B(G)$ is separated.
\end{corollary}

\subsection{Naturality properties}
\label{subsec:nat-alex}
We conclude this section with a description of the functoriality properties 
of the aforementioned constructions with Alexander invariants. 

Recall from Section \ref{subsec:alexinv} that a homomorphism 
$\alpha\colon G\to H$ induces a map between Alexander 
invariants, $B(\alpha)\colon B(G)\to B(H)$, 
which covers the ring map $\tilde\alpha_{\ab}\colon R \to S$, 
where $R=\Z[G_{\ab}]$ and $S=\Z[H_{\ab}]$. 
Clearly, $B(\alpha)$ is a morphism of topological modules, 
that is, it sends $I^k B(G)$ to $J^k B(H)$ 
for all $k\ge 1$, where $I=I(G_{\ab})$ and $J=I(H_{\ab})$. 
Moreover, $B(\alpha)$ factors as the composite  
$B(G)\to B(H)_{\alpha}\to B(H)$, where $B(H)_{\alpha}$ 
is the $\Z [G_{\ab}]$-module obtained 
from $B(H)$ by restriction of scalars via the ring map 
$\tilde\alpha_{\ab} \colon \Z [G_{\ab}]\to \Z [H_{\ab}]$.

\begin{lemma}
\label{lem:sep-induced}
If the $\Z[H_{\ab}]$-module $B(H)$ is separated, then 
the induced $\Z[G_{\ab}]$-module $B(H)_{\alpha}$ is also separated.
\end{lemma}

\begin{proof}
By definition, $B(H)_{\alpha} =B(H)$, 
with module structure given by $g\cdot m=\alpha(g)m$ 
for $g\in G_{\ab}$ and $m\in B(G)_{\alpha}$. 
It follows that $I^k B(H)_{\alpha}\subseteq J^k B(H)$ 
for all $k\ge 1$. Therefore, if $\bigcap_{k\ge 1} J^k B(H)=0$, then 
$\bigcap_{k\ge 1} I^k B(H)_{\alpha}=0$, and this proves the claim.
\end{proof}

Since the map $B(\alpha)\colon B(G) \to B(H)$ is filtration-preserving, 
it induces a morphism between the respective completions, 
$\widehat{B}(\alpha)\colon \widehat{B(G)}\to \widehat{B(H)}$. 
This morphism covers the ring map 
$\hat{\tilde\alpha}_{\ab}\colon \widehat{R} \to \widehat{S}$
and fits into the commuting diagram
\begin{equation}
\label{eq:bb-alpha}
\begin{tikzcd}[column sep=32pt]
B(G) \ar[r, "B(\alpha)"] \ar[d, "\iota_{B(G)}"] 
& B(H)\phantom{.}  \ar[d, "\iota_{B(H)}"] 
\\
\widehat{B(G)} \ar[r, "\widehat{B(\alpha)}"] & \widehat{B(H)}. 
\end{tikzcd}
\end{equation}

Passing to associated graded objects, we obtain a morphism, 
$\gr(B(\alpha))\colon \gr(B(G))\to \gr(B(H))$, which covers the 
ring map $\gr(\tilde\alpha_{\ab})\colon \gr (R)\to \gr(S)$.  
Note that $\gr_0(\tilde\alpha_{\ab})\colon R/I\to S/J$ may be identified 
with $\id\colon \Z\to \Z$, while $\gr_1(\tilde\alpha_{\ab})\colon I/I^2\to J/J^2$ 
may be identified with $\alpha_{\ab}\colon G_{\ab}\to H_{\ab}$.

\begin{lemma}
\label{lem:b-surj}
Assume the groups $G$ and $H$ are finitely generated. 
For a homomorphism $\alpha\colon G\to H$, the following 
conditions are equivalent.
\begin{enumerate}[itemsep=2pt]
\item \label{alpha1} The map 
$B(\alpha)\colon B(G)\to B(H)$ is surjective. 
\item \label{alpha2} The map 
$\widehat{B(\alpha)}\colon \widehat{B(G)}\to \widehat{B(H)}$ is surjective. 
\item \label{alpha3} The map 
$\gr(B(\alpha))\colon \gr(B(G))\to \gr(B(H))$ is surjective. 
\end{enumerate}
\end{lemma}

\begin{proof}
By Lemma \ref{lem:b-fg}, the Alexander invariants $B(G)$ and $B(H)$ 
are finitely generated modules over the Noetherian rings 
$\Z[G_{\ab}]$ and $\Z[H_{\ab}]$, respectively. 
Implication \eqref{alpha1} $\Rightarrow$ \eqref{alpha2} follows 
from the exactness property of completion, 
while \eqref{alpha2} $\Rightarrow$ \eqref{alpha3} follows from 
the fact that $\gr(B(H))$ is generated in degree $0$, 
and \eqref{alpha3} $\Rightarrow$ \eqref{alpha1} follows from 
\cite[Lemma 10.23]{AM}.
\end{proof} 

\begin{lemma}
\label{lem:b-inj}
Let $\alpha\colon G\to H$ be a homomorphism between finitely generated groups. 
Then,
\begin{enumerate}[itemsep=2pt,topsep=-1pt]
\item \label{i1} 
If $B(\alpha)\colon B(G)\to B(H)$ is injective, then 
$\widehat{B(\alpha)}\colon \widehat{B(G)}\to \widehat{B(H)}$ 
is also injective. 
\item \label{i2} 
The morphism 
$\widehat{B(\alpha)}\colon \widehat{B(G)}\to \widehat{B(H)}$ 
is injective if and only if the morphism  
$\gr(B(\alpha))\colon \gr(B(G))\to \gr(B(H))$ is injective. 
\item \label{i3} 
If $B(G)$ is separated and $\widehat{B(\alpha)}$ is injective, then 
$B(\alpha)$ is injective.
\end{enumerate}
\end{lemma}

\begin{proof}
Claim \eqref{i1} follows again from the exactness property of completion. 
The forward implication in claim \eqref{i2} follows from the Snake Lemma, 
applied to the analogue of diagram \eqref{eq:snake} 
for the module maps $\gr_k(B(\alpha))$ 
and $B(\alpha)_k\colon B(G)/I^k\to B(H)/J^k$, 
while the other implication follows from \cite[Lemma 10.23]{AM}. 
Finally, claim \eqref{i3} follows from the commutativity of diagram 
\eqref{eq:bb-alpha} and the assumption that both $\iota_{B(G)}$ and 
$\widehat{B(\alpha)}$ are injective.
\end{proof} 

\section{Lower central series and Chen groups}
\label{sect:lcs-chen}

In this section, we review the lower central series quotients and the 
Chen groups of a group $G$, and how the latter relate to the 
Alexander invariant of $G$.

\subsection{Lower central series}
\label{subsec:lcs}

The {\em lower central series}\/ (LCS) of a group $G$ is defined recursively 
by $\gamma_{k+1}(G) =[ G,\gamma_k(G)]$, starting with $\gamma_1(G)=G$. 
An inductive argument shows that $[\gamma_{k}(G),\gamma_{\ell}(G)]\subseteq 
\gamma_{k+\ell}(G)$, for all $k, \ell \ge 1$; in particular, the LCS is a central 
series, i.e., $[G,\gamma_k (G)]\subseteq \gamma_{k+1} (G)$ for all $k$.  
Moreover, its terms are fully invariant (and thus, normal) subgroups of $G$.

The successive quotients of the series, $\gamma_{k}(G)/\gamma_{k+1}(G)$, 
are abelian groups. The first such quotient, $G/\gamma_{2}(G)$, coincides  
with the abelianization $G_{\ab}=H_1(G;\Z)$.  
If $\gamma_{k}(G)\ne 1$ but $\gamma_{k+1}(G)=1$, then 
$G$ is said to be a {\em $k$-step nilpotent group}.  
For each $k\ge 1$, there is a canonical projection, 
$\psi_k\colon G\surj G/\gamma_{k+1}(G)$, where 
the factor group is the maximal $k$-step nilpotent 
quotient of $G$. 

Note that $G^{(k-1)}\subseteq \gamma_{k}(G)$, with 
equality for $k=1$ and $2$. Consequently, every nilpotent 
group is solvable. As another consequence of this observation, 
the Alexander invariant $B(G)=G'/G''$ surjects onto 
the quotient group $\gamma_2(G)/\gamma_3(G)$. 

The group $G$ is called {\em residually nilpotent}\/ if every non-identity 
element $g\in G$ is detected in a nilpotent quotient of $G$; that is, there 
is a surjective homomorphism $\varphi\colon G\to N$ such that $N$ is nilpotent 
and $\varphi(g)\ne 1$. Clearly, such an homomorphism factors through the 
projection $\psi_k\colon G\surj G/\gamma_{k+1}(G)$, for some $k\ge 1$. 
Thus, $G$ is residually nilpotent if and only if the intersection 
of its lower central series, 
$\gamma_{\omega}(G)\coloneqq \bigcap_{k\ge 1} \gamma_{k}(G)$, 
is the trivial subgroup. The finitely generated free groups $F_n$ and the pure braid 
groups $P_n$ are well-known examples of residually (torsion-free) nilpotent groups.

\subsection{Associated graded Lie algebra}
\label{subsec:grg}
The {\em associated graded Lie algebra}\/ of $G$ is the direct 
sum of the successive quotients of the lower central series. 
The addition in $\gr(G)$ is induced from the group multiplication, 
while the Lie bracket is induced from the group commutator. 
The graded pieces are the abelian groups 
$\gr_k(G)= \gamma_{k}(G)/ \gamma_{k+1}(G)$, 
while the Lie bracket is compatible with the grading. 

By construction, the Lie algebra $\gr(G)$ is generated by 
its degree $1$ piece, $\gr_1(G)=G_{\ab}$; 
thus, if $G_{\ab}$ is finitely generated, then so are the 
LCS quotients of $G$. Likewise, the $\Q$-Lie algebra $\gr(G)\otimes \Q$ 
is generated in degree $1$ by the $\Q$-vector space 
$G_{\ab}\otimes \Q=H_1(G;\Q)$. Assume now that the 
first Betti number, $b_1(G)=\dim_{\Q} H_1(G;\Q)$, is finite; 
we may then define the {\em LCS ranks}\/ of $G$ as 
\begin{equation}
\label{eq:lcs-ranks}
\phi_k(G)\coloneqq \dim_{\Q} \gr_k(G) \otimes \Q \, .
\end{equation}

If $\alpha\colon G\to H$ is a group homomorphism,
then $\alpha(\gamma_{k}(G))\subseteq \gamma_{k}(H)$, 
and thus $\alpha$ induces a map 
$\gr(\alpha)\colon \gr(G)\to \gr(H)$.  It is readily 
seen that this map preserves Lie brackets and that 
the assignment $\alpha \leadsto \gr(\alpha)$ is functorial. 
Moreover, if $G$ and $H$ are two groups, it is readily seen that 
$\gamma_k(G\times H)\cong \gamma_k(G)\times \gamma_k(H)$ 
for all $k\ge 1$, from which it follows that $\gr(G\times H)\cong 
\gr(G)\times \gr(H)$, as graded Lie algebras.

For each $k\ge 1$, the canonical projection $G\surj G/\gamma_k(G)$ 
induces an epimorphism $\gr(G) \surj \gr(G/\gamma_k(G))$, 
which is an isomorphism in degrees $j< k$. Moreover, as shown 
in \cite[Lemma 6.4]{SW-jpaa}, the following holds.
\begin{lemma}[\cite{SW-jpaa}]
\label{lem:griso}
For each $r\ge 2$, the quotient map, $G\surj G/G^{(r)}$, induces an  
epimorphism, $\gr_{k}(G) \surj \gr_{k}(G/G^{(r)})$, which is an isomorphism for 
$k\le 2^{r}-1$.
\end{lemma}

In the case when $r=2$, originally studied by 
K.-T. Chen in \cite{Chen51}, the corresponding Lie algebra, $\gr(G/G'')$, 
is called the {\em Chen Lie algebra}\/ of $G$. Assuming $b_1(G)<\infty$, we   
may define the {\em Chen ranks}\/ of $G$ as 
\begin{equation}
\label{eq:chen-ranks}
\theta_{k}(G)\coloneqq \phi_{k}(G/G'') =\dim_{\Q} \gr_{k}(G/G'') \otimes \Q \, .
\end{equation}
In view of the above discussion, we have that 
$\theta_k(G)\le \phi_k(G)$, with equality for $k\le 3$. 

\begin{example}
\label{ex:free group}
Let $F_n$ be the free group of rank $n$. The associated graded  
Lie algebra $\gr(F_n)$ is isomorphic to the free Lie algebra $\Lie(\Z^n)$.
Moreover, all its graded pieces are free abelian groups, with ranks given by 
Witt's formula, 
\begin{equation}
\label{eq:lcs-fn}
\phi_k(F_n)=\tfrac{1}{k}\sum_{d\mid k} \mu(d) n^{k/d}. 
\end{equation}
Equivalently, these LCS ranks can be read off from the equality 
$\prod_{k=1}^{\infty} (1-t^k)^{\phi_k} =1-nt$. The Chen groups 
$\gr_k(F_n/F''_n)$ are also free abelian, 
of ranks given by $\theta_1(F_n)=n$ and 
\begin{equation}
\label{eq:chen-fn}
\theta_k(F_n)=(k-1)\binom{n+k-2}{k}
\end{equation}
for $k\ge 2$, see \cite{Chen51}.
\end{example}

\subsection{Massey's correspondence}
\label{subsec:massey}
In \cite{Ms-80}, Massey established  a simple, 
yet very fruitful connection between the Alexander invariant of a group $G$   
and the lower central series of its maximal metabelian quotient, $G/G''$. 
For a full account---with complete proofs---of Massey's correspondence, 
we refer to \cite[\S6]{Su-abexact}.

Let $1\to K\to G\to Q\to 1$ be an extension of groups. Choosing a set-theoretic 
section $\sigma\colon Q\to G$ of the projection map $G\to Q$ defines a function 
$\varphi \colon Q\to \Aut(K)$ by setting $\varphi(x)(a)=\sigma(x) a \sigma(x)^{-1}$ 
for $x\in Q$ and $a\in K$. Now suppose $K$ is abelian; then the map $\varphi$ 
is a well-defined homomorphism that puts the structure of a $\Z[Q]$-module on $K$. 
Define a filtration $\{K_n\}_{n\ge 0}$ on $K$ inductively, 
by setting $K_0=K$ and $K_{n+1}=[G,K_n]$.  
Then $K_n =I^n  K$ for all $n\ge 0$, 
where $I=I(Q)$ is the augmentation ideal of $\Z[Q]$.
Applying these observations to the exact sequence \eqref{eq:gprimeprime} 
yields the following result.

\begin{theorem}[\cite{Ms-80}]
\label{thm:massey-alexinv}
Let $G$ be a group and let $I=I(G_{\ab})$. 
Then $I^k B(G) = \gamma_{k+2} (G/G'')$, for all $k\ge 0$. 
\end{theorem}
 
Passing to the associated graded objects, the theorem 
yields the following corollary.

\begin{corollary}[\cite{Ms-80}]
\label{cor:alex-chen}
For any group $G$, there are natural isomorphisms 
$\gr_{k}(B(G))\cong \gr_{k+2} (G/G'')$, for all $k\ge 0$. 
\end{corollary}

Using now Lemma \ref{lem:griso}, we obtain isomorphisms 
\begin{equation}
\label{eq:grb-01}
\gr_0(B(G))\cong \gr_2(G)\quad\text{and}\quad \gr_1(B(G))\cong \gr_3(G).
\end{equation}

As another immediate consequence of Theorem \ref{thm:massey-alexinv},
we obtain the following purely group-theoretical characterization of separability 
for the Alexander invariant of a group.

\begin{corollary}
\label{cor:alex-sep-res}
The Alexander invariant $B(G)$ is separated (in the $I$-adic topology) 
if and only if the group $G/G''$ is residually nilpotent. 
\end{corollary}

Now suppose that $b_1(G)<\infty$. Then $\gr(B(G)\otimes \Q)$ is 
a finitely generated graded module over the graded ring $\gr(\Q[G_{\ab}])$. 
Let $\theta_k(G)=\dim_{\Q} \gr_k(G/G'')\otimes \Q$ be the Chen ranks 
of $G$, starting with $\theta_1(G)=b_1(G)$.  
As a consequence of Corollary \ref{cor:alex-chen}, 
the Hilbert series of the rationalization of $\gr(B(G))$ determines 
the Chen ranks of $G$, as follows,
\begin{equation}
\label{eq:hilb-b-chen}
\Hilb (\gr(B(G) \otimes \Q),t) = \sum_{k\ge 0} \theta_{k+2} (G) t^k .
\end{equation}

\begin{example}
\label{ex:alex-free-bis}
Let $F_n$ be the free group of rank $n\ge 2$. 
As we saw in Example \ref{ex:alex-free}, the Alexander 
invariant of $F_n$, viewed as a module over the ring  
$R=\Z[\Z^n]$, has presentation $B(F_n)=\coker(\partial_3^{\ab})$. 
Consider now the associated graded object, $\gr(B(F_n))$, viewed as 
a module over the polynomial ring $S=\gr(R)$. A standard Gr\"{o}bner 
basis argument shows that $\gr(B(F_n))$ is the cokernel of the third differential 
in the Koszul complex $(\bwedge^* \Z^n \otimes S, \partial)$. 
Using the exactness of this chain complex of free $S$-modules, 
we find that 
\begin{equation}
\label{eq:hilb-bfn}
\Hilb (\gr(B(F_n)) ,t)=1-\frac{1-nt}{(1-t)^n} .
\end{equation}
Applying now formula \eqref{eq:hilb-b-chen} recovers the 
computation of the Chen ranks of the free group 
$F_n$ from \eqref{eq:chen-fn}.
\end{example}

\section{Holonomy Lie algebras and infinitesimal Alexander invariants}
\label{sect:holo}

We now review two other objects associated to a group $G$:  
the holonomy Lie algebra $\h(G)$ and the infinitesimal Alexander invariant 
$\B(G)$. 

\subsection{The holonomy Lie algebra of a group}
\label{subsec:holo-lie}

Let $G$ be a group such that the maximal torsion-free abelian quotient 
$G_{\abf}\coloneqq G_{\ab}/\Tors$ is finitely generated. 
Let $A^*=H^*(G;\Z)$ be the cohomology ring of $G$, and 
consider the map $A^1\otimes A^1\to A^2$, 
$a\otimes b\mapsto a\cup b$. 
Since the cup-product in cohomology is graded-commutative, 
this map factors through a homomorphism 
$\cup_G\colon A^1\wedge A^1\to A^2$. 
Due to our finite generation assumption on $G_{\abf}$, we have isomorphisms 
$(A^1)^{\vee} \cong G_{\abf}$ and $(A^1\wedge A^1)^{\vee}\cong 
G_{\abf}\wedge G_{\abf}$, where $(\:)^{\vee}$ signifies $\Z$-dual.
Dualizing the map $\cup_G$, we obtain the comultiplication map, 
\begin{equation}
\label{eq:nabla}
\begin{tikzcd}[column sep=20pt]
\nabla_G=(\cup_G)^{\vee}\colon H^2(G;\Z)^{\vee} \ar[r]& G_{\abf} \wedge G_{\abf}. 
\end{tikzcd}
\end{equation}

In the case when $G_{\ab}$ is torsion-free (and finitely generated), 
we may identify the $\Z$-dual of $H^2(G;\Z)$ with $H_2(G;\Z)$. 
Moreover, if we let $\ab\colon G \surj G_{\ab}$ be the abelianization 
map, then $\nabla_G$ may be identified with the induced 
homomorphism $\ab_*\colon H_2(G;\Z)\to H_2(G_{\ab};\Z)$. 

Returning to the general case, let $\Lie(G_{\abf})$ be the free 
$\Z$-Lie algebra on the free $\Z$-module 
$G_{\abf}$. This is a graded Lie algebra, with grading given by 
bracket length; we denote by $\Lie_k(G_{\abf})$ its degree $k$ piece. 
Furthermore, we identify $\Lie_1(G_{\abf})=G_{\abf}$ and 
$\Lie_2(G_{\abf})=G_{\abf}\wedge G_{\abf}$ via $[x,y]  \mapsto x\wedge y$.
With these notations and identifications, we define $\h(G)$, 
the {\em holonomy Lie algebra}\/ of $G$, 
as the quotient of the free Lie algebra on $G_{\abf}$ by the Lie 
ideal generated by the image of the comultiplication map, 
\begin{equation}
\label{eq:holo-lie}
\h(G) \coloneqq \Lie (G_{\abf})/ \text{ideal}(\im (\nabla_G))\, .
\end{equation}

Note that the ideal generated by the image of $\nabla_G$ 
is a homogeneous ideal---in fact, a quadratic ideal. Thus, 
the holonomy Lie algebra inherits a grading from 
the free Lie algebra, and this grading is compatible with the 
Lie bracket. In fact, $\h(G)=\bigoplus_{k\ge 1} \h_k(G)$ is 
a quadratic Lie algebra: it is generated in degree $1$ by 
$\h_1(G)=G_{\abf}$, and all the relations are in degree $2$. 
As noted in \cite[Proposition~6.2]{SW-jpaa}, the projection maps 
$\psi_k\colon G\surj G/\gamma_k(G)$ induce isomorphisms  
$\h(\psi_k)\colon \h(G)\isom \h(G/\gamma_k(G))$ for all $k\ge 3$. 
In particular, the holonomy Lie algebra of  $G$ depends 
only on its second nilpotent quotient, $G/\gamma_3 (G)$. 

The {\em derived series}\/ of the Lie algebra $\h=\h(G)$ 
is defined inductively by setting $\h^{(r)}=[\h^{(r-1)},\h^{(r-1)}]$, 
starting with $\h^{(0)}=\h$.
In particular, $\h^{(1)}=\h'$ is the derived Lie subalgebra and $\h^{(2)}=\h''$. 
The terms of the derived series are Lie ideals that are preserved by Lie 
algebra maps. Moreover, since $\h$ is a graded Lie algebra, 
we have that $\h'=\bigoplus_{k\ge 2} \h_k$. 

The above construction is functorial. Indeed, 
let $\alpha\colon G \to H$ be a homomorphism between two 
groups as above; then the induced homomorphism 
$\alpha_{\abf}\colon G_{\abf} \to H_{\abf}$ extends 
to a morphism $\Lie(\alpha_{\abf})\colon \Lie(G_{\abf}) \to \Lie(H_{\abf})$
between the respective free Lie algebras. The map $\alpha$ 
also induces a ring map, $\alpha^*\colon H^*(H;\Z)\to H^*(G;\Z)$; 
passing to duals, it follows that $\Lie(\alpha_{\abf})$ sends  
$\im(\nabla_G)$ to $\im(\nabla_H)$. 
Consequently, $\Lie(\alpha_{\abf})$ induces a morphism of 
graded Lie algebras, $\h(\alpha) \colon \h(G)\to \h(H)$.  
It is now readily verified that $\h(\beta\circ\alpha)=\h(\beta)\circ\h(\alpha)$. 
Moreover, if $\alpha$ is surjective, then $\h(\alpha)$ is also surjective.

As shown in the next lemma, the construction also works well with 
direct products.

\begin{lemma}
\label{lem:holo-prod}
Let $G$ and $H$ be two groups with $G_{\abf}$ and $H_{\abf}$ finitely generated. 
There is then an isomorphism of graded Lie algebras, 
\begin{equation*}
\label{eq:holo-prod}
\h(G\times H)\cong \h(G)\times \h(H).
\end{equation*}
\end{lemma}

\begin{proof}
Using the K\"{u}nneth formula to compute cup products in $H^{\le 2}(G\times H;\Z)$, 
the claim follows in a straightforward manner from the definition of the holonomy Lie algebra.
\end{proof}

In a completely analogous fashion, one may define the holonomy 
Lie algebra $\h(A)$ of a graded, graded-commutative algebra $A$ 
over a field $\k$, provided that $A^0=\k$ and $A^1$ is finite-dimensional, 
by letting $\nabla_A\colon A_2\to A_1\wedge A_1$ be the $\k$-dual of 
the multiplication map $A^1\wedge A^1\to A^2$ and setting 
$\h(A)= \Lie (A_1)/ \text{ideal}(\im (\nabla_A))$. 

If $G$ is a group with $\dim_{\k} H_1(G;\k)<\infty$, 
we set $\h(G,\k)=\h(H^*(G;\k))$. In fact, if $X$ 
is any path-connected space with $G=\pi_1(X)$, then we may 
define $\h(X,\k)\coloneqq \h(H^*(X;\k))$, after which one checks that 
$\h(X,\k)\cong \h(G,\k)$. Moreover, if $G$ is finitely generated, 
then $\h(G,\Q)\cong \h(G)\otimes \Q$.

On a historical note, the holonomy Lie algebra of a (finitely generated) 
group $G$ was first defined (over $\k=\Q$) by K.-T.~Chen \cite{Chen73}, 
and later studied by Kohno  \cite{Kohno-83} 
in the case when $G$ is the fundamental group of the complement 
of a complex projective hypersurface.  In \cite{Markl-Papadima}, 
Markl and Papadima extended the definition of the holonomy 
Lie algebra to integral coefficients.  Further in-depth studies 
were done by Papadima--Suciu \cite{PS-imrn04}
and Suciu--Wang \cite{SW-jpaa,SW-forum}. 

\subsection{A comparison map}
\label{subsec:comparison}
A notable fact about the holonomy Lie algebra
is its relationship to the associated graded Lie algebra. 
This relationship is detailed in the next theorem. 

\begin{theorem}[\cite{Markl-Papadima,PS-imrn04,SW-jpaa}]
\label{thm:holo-epi}
Let $G$ be a group such that $G_{\abf}$ is finitely generated. 
There exists a natural epimorphism of graded Lie algebras, 
$\Psi\colon \h(G) \surj \gr(G)$, 
which induces isomorphisms in degrees $1$ and $2$ and descends to epimorphisms  
$\Psi^{(r)}\colon \h(G)/\h(G)^{(r)} \surj \gr(G/G^{(r)})$ for all $r\ge 2$.  
\end{theorem}

When the group $G$ is $1$-formal, 
the maps $\Psi\otimes \Q$ and $\Psi^{(r)} \otimes \Q$ are all isomorphisms.  
In general, though, the map $\Psi\otimes \Q$ fails to be injective, even in 
degree $3$. Nevertheless, as we shall see next, the map $\Psi_3\colon \h_3(G) 
\to \gr_3(G)$ is an isomorphism for a large class of (not necessarily $1$-formal) 
groups. We summarize Theorems 3.1, 4.1, and 4.3 from \cite{PrS20}, as follows. 

\begin{theorem}[\cite{PrS20}]
\label{thm:gr3=h3}
Suppose $G_{\ab}$ is finitely generated and torsion-free. Then, 
\begin{enumerate}[itemsep=2pt]
\item \label{h3-a}
There is a natural exact sequence
\begin{equation}
\label{eq:h3}
\begin{tikzcd}[column sep=18pt]
0 \ar[r] & \h_3(G) \ar[r] & H_2(G/\gamma_3(G);\Z) \ar[r]
	& H_2(G;\Z)/(\ker \nabla_G) \ar[r] & 0\, ,
\end{tikzcd}
\end{equation}
where $\nabla_G\colon H_2(G;\Z) \to H_1(G;\Z) \wedge H_1(G;\Z)$ is the 
comultiplication map. 

\item \label{h3-b}
If, moreover, the map $\nabla_G$ is injective, then, for each $k\ge 3$, 
there is a natural, split exact sequence
\begin{equation}
\label{eq:grng}
\begin{tikzcd}[column sep=18pt]
0 \ar[r] & \gr_{k}(G) \ar[r] & H_2(G/\gamma_{k}(G);\Z) \ar[r]
	& H_2(G;\Z)/(\ker \nabla_G) \ar[r] & 0\, .
\end{tikzcd}
\end{equation}
Furthermore, the canonical projection $\Psi_3\colon \h_3(G)\to  \gr_3(G)$ 
is an isomorphism.
\end{enumerate}
\end{theorem}

Consequently, when $G_{\ab}$ is finitely generated and torsion-free and $\nabla_G$ 
is injective, the Schur multiplier of the second nilpotent quotient of $G$ decomposes as 
\begin{equation}
\label{eq:h2-nilp3}
H_2(G/\gamma_3(G);\Z) \cong H_2(G;\Z) \oplus \h_3(G).
\end{equation}

\subsection{Infinitesimal Alexander invariant}
\label{subsec:inf-alex-inv}
Let $G$ be a group and assume that its torsion-free abelianization,
$G_{\abf}=G_{\ab}/\Tors$, is finitely generated. Then $G_{\abf}$ is a 
free abelian group of rank $r=b_1(G)$. 
The symmetric algebra on this group, $S=\Sym(G_{\abf})$, is naturally 
isomorphic to $\gr(\Z[G_{\abf}])$. In concrete terms, if we identify 
$G_{\abf}$ with $\Z^r$, then $\Sym(G_{\abf})$ 
gets identified with the polynomial ring $\Z[x_1,\dots,x_r]$. 

A homomorphism $\alpha\colon G \to H$ between two groups 
as above induces a homomorphism $\alpha_{\abf}\colon G_{\abf} \to H_{\abf}$, 
which extends to a ring map, 
$\tilde\alpha_{\abf}\colon \Sym(G_{\abf})\to  \Sym(H_{\abf})$. 
If we identify these symmetric algebras with the corresponding 
polynomial rings, the map $\tilde\alpha_{\abf}$ is  the 
linear change of variables defined by the matrix of  $\alpha_{\abf}$. 
Consequently, if  $\alpha_{\abf}$ is injective (respectively, surjective), 
then $\tilde\alpha_{\abf}$ is also injective (respectively, surjective).

Following the approach from \cite{PS-imrn04}, 
we define the {\em infinitesimal Alexander invariant}\/ of 
$G$ to be the quotient (abelian) group
\begin{equation}
\label{eq:b-lin}
\B(G)\coloneqq \h(G)'/\h(G)'', 
\end{equation}
viewed as a graded module over the ring $S=\Sym(G_{\abf})$. 
Setting $\h=\h(G)$, the module structure on $\B(G)$ comes 
from the exact sequence 
$0\to \h'/\h'' \to \h/\h'' \to \h/\h' \to 0$ 
via the adjoint action of $\h/\h' =\h_1 =G_{\abf}$ 
on $\h'/\h''$ given by $g\cdot \bar{x} = \overline{[g,x]}$ 
for $g\in \h_1$ and $x\in \h'=\bigoplus_{k\ge 2} \h_k$. 
Since the grading for $S$ starts with $S_0=\Z$, we are led to define 
the grading on $\B(G)$ as 
\begin{equation}
\label{eq:bb grading}
\B_{k}(G) = (\h'/\h'')_{k+2}, 
\end{equation}
for $k\ge 0$. This equality may be viewed as an infinitesimal analogue 
of Massey's isomorphism from Corollary \ref{cor:alex-chen}.

When $G$ admits a finite, commutator-relators presentation 
(as happens for instance with an arrangement group), 
the $S$-module $\B(G)$ is isomorphic to the ``linearization" 
of the Alexander invariant $B(G)$, see \cite[Proposition~9.3]{PS-imrn04}. 

The above construction is functorial. More precisely, if $\alpha\colon G \to H$ 
is a homomorphism between two groups as above, then $\alpha$ 
induces a morphism of graded Lie algebras, $\h(\alpha) \colon \h(G)\to \h(H)$, 
which preserves the respective derived series. Hence, the restriction 
$\h'(\alpha) \colon \h(G)'\to \h(H)'$ further restrics to a map 
$\h''(\alpha) \colon \h(G)''\to \h(H)''$, and thus induces a 
homomorphism  $\B(\alpha)\colon \B(G)\to \B(H)$. A routine 
check shows that $\B(\alpha)$ is a morphism of graded modules covering 
the ring map $\tilde\alpha_{\abf}\colon \Sym(G_{\abf})\to  \Sym(H_{\abf})$, 
and that $\B(\beta\circ\alpha)=\B(\beta)\circ\B(\alpha)$. 
Clearly, if $\h'(\alpha)\colon \h(G)'\to \h(H)'$ is surjective, 
then $\B(\alpha)\colon \B(G)\to \B(H)$ is also surjective, 
and if $\h'(\alpha)$ is an isomorphism, then 
$\B(\alpha)$ is also an isomorphism.
Consequently, if $\alpha$ is surjective, then $\h'(\alpha)$ is surjective, 
and so $\B(\alpha)$ is also surjective.

Denoting by $\B(H)_{\alpha}$ the module obtained from $\B(H)$ by 
restriction of scalars along $\tilde\alpha_{\abf}$, 
we may view the map $\B(\alpha)$ as the composite 
$\B(G) \to \B(H)_{\alpha} \to \B(H)$, 
where the first arrow is a $\Sym(G_{\abf})$-linear map and the second 
arrow is the identity map of $\B(H)$, thought of as covering the ring 
map $\tilde\alpha_{\abf}$. 

\subsection{Chen ranks and $1$-formality}
\label{subsec:chen-formal}
Let $G$ be a group with $b_1(G)<\infty$. The {\em holonomy Chen ranks}\/ 
of $G$ are defined as  $\bar\theta_k(G)=\dim_{\Q} \left(\h(G;\Q)/\h(G;\Q)''\right)_k$. 
By Theorem \ref{thm:holo-epi}, we have that $\theta_{k}(G)\le \bar\theta_{k}(G)$ 
for all $k\ge 1$, with equality for $k=1$ and $2$. Using the grading convention 
from \eqref{eq:bb grading} and the proof of \cite[Proposition~8.1]{SW-mz}, we arrive 
at the following infinitesimal version of formula \eqref{eq:hilb-b-chen}. 

\begin{proposition}[\cite{SW-mz}]
\label{prop:inf-massey-corr}
Let $G$ be a group with $b_1(G)<\infty$. Then, for all $k\ge 2$, 
\begin{equation}
\label{eq:bt-bg}
\bar\theta_{k}(G)=\dim_{\Q} \B_{k-2}(G;\Q).
\end{equation}
\end{proposition}

As we shall see below, when the group $G$ is $1$-formal, 
the Chen ranks and the holonomy Chen ranks of $G$ are equal. 
The key result towards establishing this fact (proved 
in \cite[Theorem~5.6]{DPS-duke}) uses the $1$-formality hypothesis to 
construct a functorial isomorphism between the 
Alexander invariant and its infinitesimal version, at the level of 
completions (over $\Q$).

\begin{theorem}[\cite{DPS-duke}]
\label{thm:linalex-c}
Let $G$ be a $1$-formal group.  There is then a 
filtration-preserving isomorphism of completed modules, 
$\Phi_G\colon \widehat{B(G)}\otimes \Q\isom\widehat{\B(G)}\otimes \Q$. 
\end{theorem}

This isomorphism is functorial, in the following sense. 
Let $\alpha\colon G\to H$ be a homomorphism between 
two $1$-formal groups, and let $B(\alpha)\colon B(G)\to B(H)$ 
and $\B(\alpha)\colon \B(G)\to \B(H)$ be the induced morphisms 
between the two types of Alexander invariants. The following 
diagram then commutes.
\begin{equation}
\label{eq:bb-func}
\begin{tikzcd}[column sep=42pt, row sep=26pt]
\widehat{B(G)}\otimes \Q  \ar[r, "\widehat{B(\alpha)}\otimes \Q"]
\ar[d, "\Phi_G", "\cong"']
& \widehat{B(H)}\otimes \Q \phantom{.}\ar[d, "\Phi_H" , "\cong"']\\
\widehat{\B(G)}\otimes \Q \ar[r, "\widehat{\B(\alpha)}\otimes \Q"]
& \widehat{\B(H)}\otimes \Q .
\end{tikzcd}
\end{equation}

Passing to associated graded modules, we obtain as an 
immediate corollary the following result (see also \cite[\S8]{SW-mz} 
for a different approach.)

\begin{corollary}
\label{cor:linalex-gr}
If $G$ is $1$-formal, then $\gr(B(G) \otimes \Q)\cong \B(G) \otimes\Q$, 
as graded modules over $\gr(\Q[G_{\ab}]) \cong \Sym(H_1(G;\Q))$.
\end{corollary}

\begin{corollary}[\cite{PS-imrn04}]
\label{cor:ps-chen}
If $G$ is $1$-formal, then $\theta_k(G)=\bar\theta_k(G)$, 
for all $k\ge 2$, and so 
\[
\sum_{k\ge 0} \theta_{k+2} (G) t^k = \Hilb (\B(G) \otimes \Q,t) .
\]
\end{corollary}

\section{Complements of hyperplane arrangements}
\label{sect:hyp-arr}

We now turn to the study of complex hyperplane arrangements. 
This section contains a brief review of the theory of arrangements, with an emphasis 
on some of the topological invariants associated to their complements. 

\subsection{Hyperplane arrangements}
\label{subsec:hyp arr-pre}

An {\em arrangement of hyperplanes}\/ is a finite set $\A$ of 
codimension-$1$ linear subspaces in a finite-dimensional, 
complex vector space $\C^{d+1}$. The combinatorics of the 
arrangement is encoded in its {\em intersection lattice}, $L(\A)$,   
that is, the poset of all intersections of hyperplanes in $\A$ (also 
known as flats), ordered by reverse inclusion, and ranked by 
codimension.  For a flat $X=\bigcap_{H\in \B} H$ defined 
by a sub-arrangement $\B\subseteq \A$, 
we let $\rank X=\codim X$; we also write 
$L_k(\A)=\{X\in L(\A) \mid \rank X=k\}$.

Unless otherwise stated, we will assume 
that the arrangement is {\em central}, that is, all the hyperplanes 
pass through the origin.  For each hyperplane $H\in \A$, let 
$f_H\colon \C^{d+1} \to \C$ be a linear form with kernel $H$. The product 
$f=\prod_{H\in \A} f_H$, 
then, is a defining polynomial for the arrangement, 
unique up to a non-zero constant factor. 
Notice that $f$ is a homogeneous polynomial 
of degree equal to $n=\abs{\A}$, the number of hyperplanes comprising $\A$. 

By definition, the complement of the arrangement is the topological space 
$M(\A)=\C^{d+1}\setminus \bigcup_{H\in\A}H$. This 
is a connected, smooth, complex quasi-projective variety.  Moreover,  
$M=M(\A)$ is a Stein manifold, and thus it has the homotopy type of a 
CW-complex of dimension at most $d+1$.  In fact, $M$ splits 
off the linear subspace $\bigcap_{H\in \A} H$; if the dimension of 
this subspace (which we call the {\em corank}\/ of $\A$) is equal to $0$, 
we say that $\A$ is {\em essential}.

The group $\C^*$ acts freely on $\C^{d+1}\setminus \set{0}$ via 
$\zeta\cdot (z_0,\dots,z_{d})=(\zeta z_0,\dots, \zeta z_{d})$. 
The orbit space is the complex projective space of dimension $d$, 
while the orbit map, $\pi\colon \C^{d+1}\setminus \set{0} \to \CP^{d}$,  
$z \mapsto [z]$, 
is the Hopf fibration. The set  $\P(\A)=\set{\pi(H)\colon H\in \A}$ is an 
arrangement of codimension $1$ projective subspaces in $\CP^{d}$. 
Its complement, $U=U(\A)$, coincides with the quotient 
$\P(M)=M/\C^*$. 
The Hopf map restricts to a bundle map, $\pi\colon M\to U$, with fiber 
$\C^{*}$. Fixing a hyperplane $H\in \A$, we see that $\pi$ is 
also the restriction to $M$ of the bundle map 
$\C^{d+1}\setminus H\to \CP^{d} \setminus \pi(H) \cong \C^{d}$.  
This latter bundle is trivial, and so we have a diffeomorphism  
$M \cong U\times \C^*$.

\subsection{Cohomology ring}
\label{subsec:OS}

The cohomology ring of a hyperplane arrangement 
complement $M=M(\A)$ was computed by Brieskorn in \cite{Br}, 
building on the work of Arnol'd on the cohomology 
ring of the pure braid group. In \cite{OS}, Orlik and 
Solomon gave a simple description of this ring, solely 
in terms of the intersection lattice $L(\A)$, as follows.  
Fix a linear order on $\A$, and let $E=E(\A)$ be the exterior 
algebra over $\Z$ with generators $\set{e_H \mid H\in \A}$ 
in degree $1$.  Next, define a differential $\partial \colon E\to E$ 
of degree $-1$, starting from $\partial(1)=0$ 
and $\partial(e_H)=1$, and extending $\partial$ to 
a linear operator on $E$, using the graded Leibniz rule. 
Finally, let $I(\A)$ be the ideal of $E$ generated by $\partial e_{\B}$, 
for all sub-arrangements $\B\subset \A$ such that 
$\codim \bigcap_{H\in \B} H < \abs{\B}$, 
where $e_\B\coloneqq \prod_{H\in \B} e_H$.  Then 
\begin{equation}
\label{eq:OS-alg}
H^*(M(\A);\Z)=E(\A)/I(\A).
\end{equation}

Every arrangement complement $M=M(\A)$ is formal, 
that is, its rational cohomology algebra can be connected by a 
zig-zag of quasi-isomorphisms to  the algebra 
of polynomial differential forms on $M$ defined by Sullivan in \cite{Sullivan}. 
Indeed, for each $H\in \A$, the $1$-form 
$\omega_H= \frac{1}{2\pi \ii} d \log f_H$ on $\C^{d+1}$ restricts 
to a $1$-form on $M$.  As shown by Brieskorn \cite{Br},  
if $\mathcal{D}$ denotes the subalgebra 
of the de~Rham algebra $\Omega^*_{\rm dR}(M)$ 
generated over $\R$ by these $1$-forms, 
the correspondence $\omega_H \mapsto [\omega_H]$ 
induces an isomorphism $\mathcal{D} \to H^*(M;\R)$, 
and then Sullivan's machinery implies that $M$ is formal. 

\subsection{Localized sub-arrangements}
\label{subsec:local}

The {\em localization}\/ of an arrangement $\A$ at a flat $X\in L(\A)$ 
is the sub-arrangement $\A_X\coloneqq \{H\in \A : H\supset X\}$. 
The inclusion $\A_X\subset \A$ gives rise to an inclusion of complements, 
$j_X\colon M(\A) \inj M(\A_X)$.  The inclusions $\{j_X\}_{X\in L(\A)}$ assemble into a map 
\begin{equation}
\label{eq:prod-map}
\begin{tikzcd}[column sep=20pt]
j\colon M(\A) \ar[r]& \prod_{X\in L(\A)} M(\A_X) . 
\end{tikzcd}
\end{equation}

As shown by Brieskorn \cite[Lemma~3]{Br} 
(see also \cite[Lemma~5.91]{Orlik-Terao}) 
the  homomorphism induced in cohomology by $j$  
is an isomorphism in each degree $k\ge 1$.  Moreover, the groups 
$H^k(M(\A_X);\Z)$ are torsion-free, and so, by the K\"{u}nneth formula, 
we have isomorphisms 
\begin{equation}
\label{eq:brieskorn}
H^k(M(\A);\Z) \cong \bigoplus_{X\in L_k(\A)} H^k(M(\A_X);\Z)
\end{equation}
for all $k\ge 1$.  Likewise, for each $k\ge 2$, the degree $k$ piece of the Orlik--Solomon 
ideal, $I^k(\A_X)$, decomposes as the direct sum of the groups $I^k(\A_X)$, taken 
over all $X\in L_k(\A)$.
It follows that the homology groups of the complement of $\A$ 
are torsion-free, with ranks given by
\begin{equation}
\label{eq:betti-arr}
b_k(M(\A))=\sum_{X\in L_k(\A)} (-1)^k \mu(X) , 
\end{equation}
where $\mu\colon L(\A) \to \Z$ is the M\"{o}bius function 
of the intersection lattice, defined inductively by $\mu(\C^{n})=1$ 
and $\mu(X)=-\sum_{Z\supsetneq X} \mu(Z)$.  In particular, 
$H_1(M(\A);\Z)$ is free abelian of rank equal to the cardinality 
of the arrangement, $\abs{\A}$. 

As noted in \cite{Ra97} and \cite{MS00}, the fact that $H^*(M(\A);\Z)$ is 
torsion-free (as a group) and generated in degree $1$ (as an algebra) implies 
that, for each $2\le i\le d+1$, the Hurewicz homomorphism 
$h\colon \pi_{i}(M(\A))\to H_{i}(M(\A);\Z)$ is the zero map.

\subsection{Fundamental group}
\label{subsec:pi1}
Fix a basepoint $x_0$ in the complement of $\A$, and consider the 
fundamental group $G(\A)=\pi_1(M(\A),x_0)$. 
For each hyperplane $H\in \A$, pick a meridian curve about $H$, 
oriented compatibly with the complex orientations on $\C^{d+1}$ and $H$, 
and let $x_H$ denote the based homotopy class of this curve, 
joined to the basepoint by a path in $M$.  By the van Kampen theorem, 
then, the arrangement group, $G=G(\A)$, is generated by the set 
$\set{x_H : H\in \A}$. Using the braid monodromy algorithm from 
\cite{CS-cmh97}, one may obtain a finite presentation 
of the form $G=F_n/R$, where $F_n$ is the rank $n$ free group on 
the set of meridians and the relators in $R$ belong to the commutator 
subgroup $F_n'$, so that $G_{\ab}=H_1(G;\Z)\cong \Z^n$.

Probably the best-known arrangement is the braid arrangement 
$\A_n$, consisting of the diagonal hyperplanes in 
$\C^{n}$.  It is readily seen that $L(\A_{n})$ is the lattice of partitions 
of $[n]=\{1,\dots,n\}$, ordered by refinement, while   
$M(\A_{n})$ is the configuration space $F(\C,n)$ of $n$ 
ordered points in $\C$, which is a classifying space for $P_{n}$, 
the pure braid group on $n$ strings.

More generally, if the intersection lattice $L(\A)$ is supersolvable, 
the complement $M=M(\A)$ is a classifying space for the group $G=G(\A)$. 
On the other hand, there are many arrangements $\A$ (for instance, generic 
arrangements) for which $M$ is not a $K(G,1)$. Nevertheless, as noted above, the 
Hurewicz map $h\colon \pi_2(M)\to H_2(M;\Z)$ is the zero map. 
Therefore, if $g\colon M\to K(G,1)$ is a classifying map, then, 
by the Hopf exact sequence, the induced homomorphism 
$g_*\colon H_i(M;\Z)\to H_i(G;\Z)$ is an isomorphism for $i\le 2$.

For the purpose of computing the group $G=G(\A)$, it is enough to assume that 
the arrangement $\A$ lives in $\C^3$, in which case $\bar{\A}=\P(\A)$ 
is an arrangement of (projective) lines in $\CP^2$. This is clear when the 
rank of $\A$ is at most $2$, and may be achieved otherwise 
by taking a generic $3$-slice, an operation that does not
change either the poset $L_{\le 2}(\A)$ or the group $G$. 
For a rank-$3$ arrangement, the set $L_1(\A)$ is in $1$-to-$1$ correspondence 
with the lines of $\bar{\A}$, while $L_2(\A)$ is in $1$-to-$1$ correspondence 
with the intersection points of $\bar{\A}$.  Moreover, the poset structure of $L_{\le 2}(\A)$ 
mirrors the incidence structure of the point-line configuration $\bar{\A}$.   

\subsection{Localization and fundamental groups}
\label{subsec:loc-pi1}
Choosing a point $x_0$ sufficiently close to $\bz\in \C^{d+1}$, we can make 
$x_0$ a common basepoint for both $M(\A)$ and all the local complements $M(\A_X)$.  
The next result (Lemmas 4.1 and 4.3 from \cite{DSY16}) gives a rough analog of Brieskorn's 
Lemma at the level of fundamental groups. For each flat $X\in L(\A)$, let 
$j_X\colon M(\A) \to M(\A_X)$ be the corresponding inclusion. 

\begin{lemma}[\cite{DSY16}]
\label{lem:dsy}
There exist basepoint-preserving maps $r_X\colon M(\A_X)\to M(\A)$ 
such that $j_X\circ r_X\simeq \id$ relative to $x_0$. Moreover, if 
$H\in \A$ and $H\not\supset X$, then the 
composite $r_X \circ j_X \circ r_H$ is null-homotopic.
\end{lemma}

Setting $G(\A_X)=\pi_1(M(\A_X),x_0)$, it follows from the lemma that 
the induced homomorphism $(j_X)_{\sharp}\colon G(\A) \to G(\A_X)$ 
is surjective and $(r_X)_{\sharp}\colon G(\A_X) \to G(\A)$ 
is injective.  

Of particular interest to us is what happens when $X$ has codimension $2$.  
For a $2$-flat $X$, the localized sub-arrangement $\A_X$
is a pencil of $\abs{X}=\mu(X)+1$ hyperplanes.
Consequently, $M(\A_X)$ is homeomorphic to 
$(\C \setminus \{\text{$\mu(X)$ points}\}) \times \C^* \times \C^{d-1}$,
and so $M(\A_X)$ is a classifying space for the group 
\begin{equation}
\label{eq:gax}
G(\A_X) \cong F_{\mu(X)}\times \Z.
\end{equation}

\subsection{The second nilpotent quotient}
\label{subsec:nilp2-arr}

Let $G=G(\A)$ be an arrangement group.  
Plainly, the abelianization $G_{\ab}=H_1(M(\A);\Z)$ 
is the free abelian group on $\{x_H\}_ {H\in \A}$.  
On the other hand, as noted for instance in \cite{MS00}, 
the abelian group $\gr_2(G)$ is the $\Z$-dual of $I^2(\A)$, 
the degree-$2$ part of the Orlik--Solomon ideal. 
In particular, $\gr_2(G)$ is also torsion-free. 

\begin{proposition}[\cite{MS00}]
\label{prop:ms}
For any arrangement $\A$, the second nilpotent quotient of 
$G(\A)$ fits into a central extension of the form 
\begin{equation}
\label{eq:2nq-arr-ext}
\begin{tikzcd}[column sep=18pt]
0 \ar[r] & (I^2(\A))^{\vee} \ar[r] & G(\A)/\gamma_3(G(\A))  
\ar[r] & H_1(M(\A);\Z) \ar[r] & 0 .
\end{tikzcd}
\end{equation}
Furthermore, the $k$-invariant of this extension, $\chi_2 \colon 
H_2( G_{\ab} ;\Z)\to \gr_2(G)$, is the dual of the inclusion map 
$I^2(\A)\inj  E^2(\A)=\bigwedge^2 G_{\ab}$.  
\end{proposition}

Let $F=F(\A)$ be the free group on generators $\{x_H : H\in \A\}$. 
It follows that $G/\gamma_3(G)$ is the quotient of the free, 
$2$-step nilpotent group $F/\gamma_3(F)$ by all commutation 
relations of the form 
\begin{equation}
\label{eq:nilp2-arr}
r_{H,X}\coloneqq \Big[x_H , \prod_{\substack{K\in \A\\[2pt] K\supset X}} x_{K}\Big] \, ,
\end{equation}
indexed by pairs of hyperplanes $H\in \A$ and flats $X\in L_2(\A)$ 
such that $H\supset X$ (see \cite{Rybnikov-1,MS00}).  
From this description, it is apparent 
that the second nilpotent quotient of an arrangement group is 
combinatorially determined.  More precisely, suppose $\A$ and $\BB$ are 
two arrangements such that $L_{\le 2} (\A)\cong L_{\le 2} (\BB)$, meaning, 
there are bijections $\alpha\colon \A\to \BB$ and 
$\beta\colon L_2(\A) \to L_2(\BB)$ such that 
$\beta(H_1\cap\cdots \cap H_r)=\alpha(H_1)\cap\cdots\cap\alpha(H_r)$. 
Then $\alpha$ extends to an isomorphism of free groups, $\alpha\colon F(\A)\to F(\BB)$, 
that sends $r_{H,X}\in F(\A)$ to $r_{\alpha(H),\beta(X)}\in F(\BB)$, 
and thus descends to an isomorphism 
$G(\A)/\gamma_3(G(\A))\isom G(\BB)/\gamma_3(G(\BB))$.

\section{Holonomy Lie algebras of arrangements}
\label{sect:holo-arr}

We now study the holonomy Lie algebra $\h(\A)$ of an arrangement group $G(\A)$, 
how it behaves under certain operations with arrangements, and 
how it relates to the associated graded Lie algebra $\gr(G(\A))$.

\subsection{Holonomy Lie algebra}
\label{subsec:holo lie arr}
The holonomy Lie algebra of an arrangement $\A$ is 
defined as the holonomy Lie algebra of the fundamental 
group $G(\A)=\pi_1(M(\A))$ of the complement of $\A$, 
\begin{equation}
\label{eq:holo-A}
\h(\A)\coloneqq \h(G(\A))\, .
\end{equation}  

An explicit presentation for this Lie algebra was given by 
Kohno \cite{Kohno-83}, as follows. 
Let $\LL(\A)\coloneqq \Lie(H_1(M(\A);\Z))$ be the free 
Lie algebra on variables $\{x_H\}_{H\in \A}$.  From the Orlik--Solomon 
description of the cohomology ring $H^*(M(\A);\Z)$ as the quotient 
$E/I(\A)$, it follows that 
\begin{equation}
\label{eq:holo arr}
\h(\A)=\LL( \A)/J(\A)\,  ,
\end{equation}
where $J(\A)$ is the Lie ideal generated by all the Lie brackets of the form 
\begin{equation}
\label{eq:lie-ideal}
\mathfrak{r}_{H,X}\coloneqq 
\Big[x_H , \sum_{\substack{K\in \A\\[2pt] K\supset X}} x_{K}\Big]\, ,
\end{equation}
indexed by pairs of hyperplanes $H\in \A$ and flats $X\in L_2(\A)$ 
such that $H\supset X$. Plainly, $J(\A)$ is a homogeneous ideal, and 
so $\h(\A)$ inherits the structure of a graded Lie algebra from $\LL( \A)$. 
In fact, $\h(\A)$ is a finitely generated, quadratic Lie algebra, with 
generators $x_H$ in degree $1$ and relators $\mathfrak{r}_{H,X}$ 
in degree $2$.  Completely analogous considerations apply to 
the holonomy $\k$-Lie algebra $\h(\A,\k)\coloneqq \h(G(\A),\k)$ 
over a field $\k$.

\begin{example}
\label{ex:holo-pencil}
Let $\A$ be a pencil of $n$ lines through the origin of $\C^2$. 
Then $\h(\A)$ is the quotient of the free Lie algebra $\Lie(n)=\Lie(x_1,\dots ,x_n)$ 
by the ideal generated by the relators $[x_i, \sum_{j=1}^n x_j]$ for $1\le i<n$. 
Clearly, $\h(\A)\cong \Lie(n-1) \times \Lie(1)$. A basis for $\h_2(\A)$ consists 
of all the brackets $[x_i,x_j]$ with $1\le i<j<n$, while a basis for $\h_3(\A)$ consists 
of all the triple brackets  $[x_i,[x_j,x_k]]$ with $1\le i,j,k <n$, $i\ge j$, and $j<k$.
\end{example}

From the presentation of the holonomy Lie algebra of a hyperplane arrangement 
$\A$ given in \eqref{eq:holo arr}--\eqref{eq:lie-ideal}, it follows that $\h(\A)$ 
depends only on the underlying matroid of the arrangement, and in fact, 
only on the ranked poset $L_{\le 2} (\A)$.  Indeed, suppose $\A$ and $\BB$ are 
two arrangements such that there exist compatible bijections $\alpha\colon \A\to \BB$ 
and $\beta\colon L_2(\A) \to L_2(\BB)$. 
Then $\alpha$ extends to a isomorphism between the respective free Lie algebras, 
$\LL(\alpha)\colon \LL(\A)\isom \LL(\BB)$, that sends 
$\mathfrak{r}_{H,X}\in \LL_2(\A)$ to $\mathfrak{r}_{\alpha(H),\beta(X)}\in \LL_2(\BB)$. 
Therefore, $\LL(\alpha)$ descends to an isomorphism 
$\h(\alpha)\colon \h(\A)\isom \h(\BB)$ of graded Lie algebras.

\subsection{A comparison map}
\label{subsec:holo-gr}
By Theorem \ref{thm:holo-epi}, 
there is a surjective Lie algebra map, $\Psi=\Psi_{\A}\colon \h(\A)\surj \gr(G(\A))$. 
The $1$-formality of the group $G(\A)$ implies that the associated graded
Lie algebra $\gr(G(\A))$ and the holonomy Lie algebra $\h(\A)$
are rationally isomorphic; therefore, the map 
\begin{equation}
\label{eq:kohno}
\begin{tikzcd}[column sep=20pt]
\Psi\otimes \Q\colon \h(\A)\otimes \Q\ar[r]& \gr(G(\A))\otimes \Q 
\end{tikzcd}
\end{equation}
is an isomorphism. Consequently, all the LCS ranks of an arrangement 
group $G(\A)$ are determined by the truncated intersection lattice $L_{\le 2}(\A)$, 
as follows:
\begin{equation}
\label{eq:lcs-arr}
\phi_k(G(\A))=\dim_{\Q} \h_k(\A)\otimes \Q \,.
\end{equation}

In general, there exist arrangements $\A$ for which the map $\Psi_{\A}$ is not injective.  
Nevertheless, we have the following result from \cite{PrS20}, that shows this cannot 
happen in degree $k=3$; for completeness, we provide a quick proof.

\begin{theorem}[\cite{PrS20}]
\label{thm:gr3h3-arr}
For an arrangement $\A$ with complement $M=M(\A)$ and group 
$G=\pi_1(M)$, the following hold.
\begin{enumerate}[itemsep=2.5pt, topsep=-1pt]
\item \label{mg1}
$H_2(G/\gamma_3(G);\Z) \cong H_2(M;\Z) \oplus \h_3(\A)$.
\item \label{mg2}
The map $\Psi_3\colon \h_3(\A)\to  \gr_3(G)$ is an isomorphism.
\end{enumerate}
\end{theorem}

\begin{proof}
Since the cohomology algebra $H^*(M;\Z)$ is generated in degree $1$, 
the cup-product map $\cup_M\colon H^1(M;\Z)\wedge  H^1(M;\Z)\to  H^2(M;\Z)$ 
is surjective; thus the comultiplication map 
$\nabla_M\colon H_2(M;\Z) \to H_1(M;\Z) \wedge H_1(M;\Z)$ is injective. 
Moreover, as noted previously, any classifying map $g\colon M\to K(G,1)$ 
induces an isomorphism on $H_{\le 2}$; hence, by naturality of cup-products, 
the map $\nabla_G$ is also injective. Both claims now follow from 
Theorem \ref{thm:gr3=h3}.
\end{proof}

Consequently, the group $\gr_3(G(\A))$ is combinatorially determined; 
that is, if $\A$ and $\B$ are two arrangements 
such that $L_{\le 2} (\A)\cong L_{\le 2} (\B)$, then 
$\gr_3(G(\A))\cong \gr_3(G(\B))$.  These considerations 
naturally lead to the following question. 

\begin{question}
\label{quest:h3-tf}
For an arrangement $\A$, is the (finitely generated abelian) group 
$\h_3(\A)=\gr_3(G(\A))$ torsion-free?
\end{question}

In view of Theorem \ref{thm:gr3h3-arr} and the fact that $H_*(M(\A);\Z)$ is 
torsion-free, the question may be rephrased as: Is the Schur multiplier of 
the finitely generated, $2$-step nilpotent group $G(\A)/\gamma_3(G(\A))$ torsion-free? 

\begin{remark}
\label{rem:falk}
In \cite{Fa88}, Falk sketched the construction of a Sullivan $1$-minimal 
model for the complement of an arrangement $\A$, and used this to 
show that $\phi_3(G(\A))$ is equal 
to the nullity of the multiplication map $E^1\otimes I^2 \to E^3$ 
over $\Q$.  Further information on the ranks of the LCS quotients 
of an arrangement group can be found in \cite{Su01, SS-tams02}.
\end{remark}

\begin{remark}
\label{rem:AGV}
As first noted in \cite{Su01}, there exist arrangements 
$\A$ for which $\gr_k(G(\A))$ has non-zero torsion 
for some $k>3$.  This naturally raised the question whether 
such torsion in the LCS quotients of arrangement groups is 
combinatorially determined.  The question was recently 
answered in the negative by Artal Bartolo, Guerville-Ball\'{e}, and Viu-Sos 
\cite{AGV20}, who produced a pair of arrangements 
$\A$ and $\BB$ with $L_{\le 2} (\A)\cong L_{\le 2} (\BB)$, yet 
$\gr_4(G(\A))\not\cong \gr_4(G(\BB))$; the difference 
lies in the $2$-torsion of the respective groups. 
\end{remark}

\subsection{Products of arrangements}
\label{subsec:prod-arr}
Let us recall a definition from \cite{Orlik-Terao}. If $\A$ is an arrangement in $\C^r$ 
and $\BB$ is an arrangement in $\C^s$, then their product, $\A\times \BB$, is 
the arrangement in $\C^{r+s}=\C^r\times \C^s$ given by 
\begin{equation}
\label{eq:arr-prod}
\A\times \BB \coloneqq \{H \times \C^s : H\in \A\}\cup \{\C^r \times K : K\in \BB\}.
\end{equation}
It is readily seen that the intersection lattice of the product arrangement is isomorphic 
to the product of the respective intersection posets, that is,
\begin{equation}
\label{eq:lat-prod}
L(\A\times \BB)\cong L(\A)\times L(\BB).
\end{equation}

Alternatively, if $f=f(z_1,\dots,z_r)$ and $g=g(w_1,\dots,w_s)$ are defining polynomials 
for $\A$ and $\BB$, respectively, then a defining polynomial for $\A\times \BB$ is 
$k(z_1,\dots ,z_r,w_1, \dots ,w_s)\coloneqq f(z_1,\dots,z_r)g(w_1,\dots,w_s)$. 
Plainly, $k(z,w)\ne 0$ if and only if $f(z)\ne 0$ and $g(w)\ne 0$; thus, the complement 
of the product arrangement, $M(\A\times \BB)$, is diffeomorphic to the product of the 
complements of the two arrangements, $M(\A)\times M(\BB)$. 
Consequently, the group $G(\A\times \BB)$ is isomorphic to the direct 
product $G(\A)\times G(\BB)$, and it follows from Lemma \ref{lem:holo-prod} that
\begin{equation}
\label{eq:holo-arr-prod}
\h(\A\times \BB)\cong \h(\A)\times \h(\BB).
\end{equation}

\subsection{Maps between holonomy Lie algebras}
\label{subsec:holo-maps}
We now review some constructions from \cite{PS-cmh06} 
putting them into the context that will be needed here. 
For an arrangement $\A$, we denote by $\Z^{\A}$ the
free abelian group on $\A$, with basis $\{x_H\}_{H\in \A}$, 
and we let $\Lie(\A)=\Lie(\Z^{\A})$ be the free Lie algebra 
on this group. 
For each sub-arrangement $\BB\subset \A$, we define 
two maps between the respective holonomy Lie algebras 
(in opposite directions), as follows. 

First, let $\pi_{\BB}\colon \Z^{\A} \to \Z^{\BB}$
be the canonical projection map, defined by 
\begin{equation}
\label{eq:pi-map}
\pi_{\B}(x_H)=
\begin{cases} 
x_H&\text{if $H\in \BB$},\\[2pt]
0&\text{otherwise},
\end{cases}
\end{equation}
and let $\LL (\pi_{\BB})\colon \LL(\A)\surj \LL(\BB)$ be its extension
to free Lie algebras.  Clearly, $\LL (\pi_{\BB})$ takes an element 
$\mathfrak{r}^{\A}_{H,X}\in J(\A)$ as in \eqref{eq:lie-ideal} to the corresponding element 
$\mathfrak{r}^{\BB}_{H,X}\in J(\BB)$. Thus, this maps descends to an epimorphism of 
Lie algebras,
\begin{equation}
\label{hhpi}
\begin{tikzcd}[column sep=18pt]
\HH (\pi_{\BB})\colon \HH(\A) \ar[r, two heads]& \HH(\BB).
\end{tikzcd}
\end{equation}

Next, let $\iota_{\BB}\colon \Z^{\BB} \to \Z^{\A}$ be the canonical
inclusion, given by $\iota_{\BB}(x_H)=x_H$, and let
$\LL (\iota_{\BB})\colon \LL(\BB)\to \LL(\A)$ be its extension to free
Lie algebras.  In general, this map does not take $J(\BB)$ to $J(\A)$. 
However, suppose $\BB$ is a {\em closed}\/ sub-arrangement of $\A$, that is, 
the only linear combinations of defining forms for the hyperplanes in $\BB$ which
are defining forms for hyperplanes in $\A$ are (up to constants)
the defining forms for the hyperplanes in $\BB$.
Then $L_2(\BB)=\{ X \in L_2(\A) \mid \A_X \subset \BB\}$.
Thus, $\LL (\iota_{\BB})(J(\BB))\subset J(\A)$, and so we
obtain a map of graded Lie algebras,
\begin{equation}
\label{hhiota}
\begin{tikzcd}[column sep=18pt]
\HH (\iota_{\BB})\colon \HH(\A) \ar[r]& \HH(\BB).
\end{tikzcd}
\end{equation}

\subsection{Local holonomy Lie algebras}
\label{subsec:local-holo}
Now let $X$ be a rank-$2$ flat in $L_2(\A)$. Clearly, the localized arrangement 
$\A_X$ is a closed sub-arrangement of $\A$. Setting $\pi_X= \pi_{\A_X}$ and 
$\iota_X=\iota_{\A_X}$, we obtain a pair of maps  
between the respective holonomy Lie algebras, 
$\h (\pi_{X})\colon \h(\A) \to \h(\A_X)$ and 
$\h (\iota_{X})\colon \h(\A_X) \to \h(\A)$. 
Plainly, $\pi_X\circ \iota_X$ is the identity of $\Z^{\A_X}$, and so 
$\h (\pi_{X})\circ \h (\iota_{X})$ is the identity of $\h(\A_X)$. 
Furthermore, it is shown in \cite[Lemma 3.1]{PS-cmh06} that 
$\h'(\pi_X)\circ\h'(\iota_Y) = 0$ if $Y$ is a $2$-flat different from $X$. 

On the other hand, recall from \S\ref{subsec:loc-pi1} that we also 
have pointed maps  
$j_X\colon M(\A) \inj M(\A_X)$ and $r_X\colon M(\A_X)\to M(\A)$. 
Let  $(r_X)_{\sharp}\colon G(\A_X) \to G(\A)$ and 
$(j_X)_{\sharp}\colon G(\A) \to G(\A_X)$ be the 
induced homomorphisms on fundamental groups,  
and let 
 $\h((r_X)_{\sharp})\colon \h(\A_X) \to \h(\A)$ and 
 $\h((j_X)_{\sharp})\colon \h(\A) \to \h(\A_X)$
be the corresponding maps on 
holonomy Lie algebras. 
 
The connection between the various maps defined above 
is made by the following lemma.

\begin{lemma}
\label{lem:j-pi}
For every $X\in L_2(\A)$, the following hold.
\begin{enumerate}[itemsep=2pt]
\item  \label{jr1}
The homomorphism $(j_X)_*\colon H_1(M(\A);\Z)\to H_1(M(\A_X);\Z)$ 
may be identified with the homomorphism $\pi_X\colon \Z^{\A}\to \Z^{\A_X}$. 
\item  \label{jr2}
The homomorphism $(r_X)_*\colon H_1(M(\A_X);\Z)\to H_1(M(\A);\Z)$ 
may be identified with the homomorphism $\iota_X\colon \Z^{\A_X}\to \Z^{\A}$. 
\item  \label{jr3}
$ \h((j_X)_{\sharp})=\h(\pi_X)$ 
and $\h((r_X)_{\sharp})=\h(\iota_X)$.
\end{enumerate}
\end{lemma}

\begin{proof}
The homomorphism $(j_X)_{\sharp}$ sends a generator $x_H\in G(\A)$ to 
the corresponding generator $x_{H,X}\in G(\A_X)$  if $H\supset X$ and 
sends it to $0$ if $H\not\supset X$, while 
the homomorphism $(r_X)_{\sharp}$ sends  
$x_{H,X}\in G(\A_X)$ to $x_H\in G(\A)$.  
The first two claims now follow from the 
way the maps $\pi_X$ and $\iota_X$ were defined.
Claim \eqref{jr3} follows from  \eqref{jr1} and \eqref{jr2}.
\end{proof}

As we saw in \eqref{eq:gax}, for each $2$-flat $X\in L_2(\A)$, 
the group $G(\A_X)$ is isomorphic to $F_{\mu(X)}\times \Z$; 
therefore, $\gr(G(\A_X)) \cong \LL({\mu(X)}) \times \LL(1)$. 
Furthermore, as noted in Example \ref{ex:holo-pencil}, the Lie algebra 
$\h(\A_X)$ is also isomorphic to $\LL({\mu(X)}) \times \LL(1)$. 
Consequently, the map $\Psi_{\A_X}\colon \h(\A_X) \to  \gr(G(\A_X))$ 
is an isomorphism. 

\section{Decomposable arrangements}
\label{sect:decomp}

In this section, we focus on a class of arrangements $\A$ for which several of the 
algebraic invariants of the group $G(\A)$ discussed previously  ``decompose'' 
in terms of the corresponding invariants of the groups $G(\A_X)$, indexed 
by the $2$-flats in $L_{2}(\A)$. 

\subsection{Local to global maps}
\label{subsec:loc}
Let $j\colon M(\A) \to \prod_{X\in L(\A)} M(\A_X)$ be the map 
from \eqref{eq:prod-map}. Projecting onto the factors 
corresponding to rank $2$ flats we obtain a map 
\begin{equation}
\label{eq:prod-map-2}
\begin{tikzcd}[column sep=20pt]
j\colon M(\A) \ar[r]& \prod_{X\in L_2(\A)} M(\A_X)
\end{tikzcd}
\end{equation}
so that  projection onto an $X$-factor composed with $j$ coincides 
with the map $j_X\colon M(\A)\inj M(\A_X)$ induced by the inclusion 
$\A_X\subset \A$. 
In what follows, we will identify the group $H_1( \prod_{X} M(\A_X);\Z)$ 
with $\bigoplus_{X} H_1(M(\A_X);\Z)$. 

\begin{lemma}
\label{lem:j-inj}
The homomorphism $j_{*}\colon H_1(M(\A);\Z) \to \bigoplus_{X\in L_2(\A)} H_1(M(\A_X);\Z)$ 
is injective.
\end{lemma}

\begin{proof}
The group $H_1(M(\A);\Z)$ is free abelian on generators $\{x_H : H\in \A\}$, 
whereas $H_1(M(\A_X);\Z)$ is free abelian on generators $\{x_{H,X} : H\in \A, H\supset X\}$. 
In these bases, we have that $j_*(x_H) =\sum_{X: X\subset H} x_{H,X}$. Since every 
hyperplane $H\in \A$ contains a flat $X\in L_2(\A)$, the matrix of $j_*$ has a minor 
of size $n=\abs{\A}$ equal to $1$, and so the matrix has maximal rank. 
Hence, $j_*$ is injective.
\end{proof}

For each flat $X\in L_2(\A)$, 
the inclusion $j_X\colon M(\A)\inj M(\A_X)$ induces a  
homomorphism $(j_X)_{\sharp}\colon G(\A) \to G(\A_X)$ 
on fundamental groups, which in turn induces a map 
$\h((j_X)_{\sharp})\colon \h(\A) \to \h(\A_X)$ between 
the corresponding holonomy Lie algebras. Let us define 
a ``local" version of the holonomy Lie algebra by setting 
\begin{equation}
\label{eq:ha-loc}
\h(A)^{\loc}\coloneqq  \h\Big(\prod_{X\in L_2(\A)} G(\A_X) \Big) = 
\prod_{X\in L_2(\A)} \h(\A_X).
\end{equation}
The maps $\h((j_X)_{\sharp})$ then assemble into a map to 
$\h(j_{\sharp}) \colon \h(\A)  \to \h(A)^{\loc}$. 
By Lemma \ref{lem:j-pi}, this map coincides with the map 
$\pi=\left(\HH(\pi_X)\right)_X\colon \HH(\A)\to
\prod_{X} \HH(\A_X)$  from \cite[(2.2)]{PS-cmh06}. 

Likewise, the maps $\h(\iota_X)=\h((r_X)_{\sharp})\colon \h(\A_X) \to \h(\A)$ 
assemble into a homomorphism of graded abelian groups, 
$\iota\colon \h(A)^{\loc} \to \h(\A)$. 
Let $\pi'$ and $\iota'$ be the restrictions of the aforementioned maps 
to derived Lie subalgebras. 
It is shown in \cite[Proposition~2.1]{PS-cmh06} that $\pi'\circ \iota'=\id$, 
and thus, $\pi'$ is surjective (it is also injective in degree $2$). 
We summarize this result, as follows. 

\begin{proposition}[\cite{PS-cmh06}]
\label{prop:PS-holo-prod}
The morphism of graded Lie algebras
\begin{equation}
\label{eq:pimap}
\begin{tikzcd}[column sep=18pt]
\h(j_{\sharp}) \colon \h(\A)\ar[r] & \h(\A)^{\loc}
\end{tikzcd}
\end{equation}
is a surjection in degrees $k\ge 3$ and an isomorphism in degree $k=2$.
\end{proposition}

By comparing the ranks of the source and target in \eqref{eq:pimap}, 
while also using \eqref{eq:gax} and \eqref{eq:kohno}, 
we recover a lower bound for the LCS ranks of an arrangement group,
first established in \cite[Proposition~3.8]{Fa89} by other methods.

\begin{corollary}[\cite{Fa89, PS-cmh06}] 
\label{cor:lcs-bound}
For any arrangement $\A$, the LCS ranks of the group $G(\A)$ admit the 
lower bound
\begin{equation}
\label{eq:phis-min}
\phi_k(G(\A))\ge \sum_{X\in L_2(\A)}\phi_k(F_{\mu(X)}) 
\end{equation}
for all $k\ge 2$, with equality for $k=2$.
\end{corollary}

As illustrated in the next example, the epimorphisms $\h_k(j_{\sharp})$ with 
$k\ge 3$ are far from being injective, in general. Thus, the lower bound 
from \eqref{eq:phis-min} may be strict, even for $k=3$.

\begin{example}
\label{ex:braid-holo}
Let $\BB$ be the braid arrangement in $\C^3$, defined by the polynomial 
$f=(x+y)(x-y)(x+z)(x-z)(y+z)(y-z)$.  Labeling the hyperplanes as the 
factors of $f$, the flats in $L_2(\BB)$ are $\{136, 145, 235, 246, 12, 34, 56\}$, 
and so $\h_3(\BB)^{\loc}=\Z^8$. Nevertheless, 
$\h_3(\BB)=\Z^{10}$, and thus $\ker(\h_3(j_{\sharp}))=\Z^2$, 
generated by the triple Lie brackets $[x_2 ,[x_4, x_5]]$ and 
$[x_5, [x_2, x_4]]$.
\end{example}

\subsection{Decomposable arrangements}
\label{subsec:decomp}
These considerations motivate the following definition, which is key to 
this work.

\begin{definition}[\cite{PS-cmh06}]
\label{def:ps-decomp}
A hyperplane arrangement $\A$ is said to be {\em decomposable}\/ if the map 
$\h_3(j_{\sharp})\colon  \h_3(\A)\to \h_3(\A_X)^{\loc}$ 
is an isomorphism. Likewise, the arrangement is {\em decomposable over $\Q$}\/ 
if the map $\h_3(j_{\sharp})\otimes \Q$ is an isomorphism. 
\end{definition}

That is to say, $\A$ is decomposable if $\h_3(\A)$ is free abelian of rank as 
small as possible, given the information encoded in $L_2(\A)$, namely, of rank 
equal to 
\begin{equation}
\label{eq:h3rank}
\sum_{X\in L_2(\A)} \rank \h_3(\A_X)= 2 \sum_{X\in L_2(\A)} \binom{\mu(X)+1}{3}.
\end{equation}
Alternatively, if we set $\widetilde{L}_2(\A)\coloneqq \{X\in L_2(\A) : \mu(X)>1\}$, 
then the decomposability condition for $\A$ is equivalent to 
\begin{equation}
\label{eq:h3rank-bis}
\h_3(\A) \cong \bigoplus_{X\in \widetilde{L}_2(\A)} \h_3(\A_X).
\end{equation}

Since the holonomy Lie algebra $\h(\A)$ depends only on the intersection 
poset $L_{\le 2}(\A)$ the property of being decomposable (or $\Q$-decomposable) 
is combinatorially determined. We formalize this observation in the next lemma.

\begin{lemma}
\label{lem:dec-comb}
Let $\A$ and $\BB$ be two arrangement with $L_{\le 2}(\A)\cong L_{\le 2}(\BB)$. 
If $\A$ is decomposable (over $\Q$), then $\BB$ is also decomposable (over $\Q$).
\end{lemma}

The perhaps weaker condition of $\Q$-decomposability only requires that 
$\phi_3(G(\A))=\dim_{\Q}\h_3(\A)\otimes \Q$ be equal to \eqref{eq:h3rank}, 
or, equivalently, that the lower bound from \eqref{eq:phis-min} to hold as equality 
for $k=3$. An alternate definition was given in \cite[Definition~2.10]{SS-tams02}, 
where an arrangement $\A$ is said to be {\em minimal linear strand}\/ if 
\begin{equation}
\label{eq:mls}
b'_{2,3}(\A)=2 \sum_{X\in L_2(\A)} \binom{\mu(X)+1}{3},
\end{equation}
where $b'_{i,j}\coloneqq \dim_{\Q} \Tor^E_i(A,\Q)_j$ are the bigraded Betti numbers 
of the (rational) Orlik--Solomon algebra $A=E/I(\A)$. Since, as noted in display 
(3.13) of that paper, $\phi_3(G(\A))=b'_{2,3}(\A)$, the two notions---decomposable 
over $\Q$ and minimal linear strand---coincide. 

These considerations raise the following question. 

\begin{question}
\label{quest:decomp-Q}
For an arrangement $\A$, are decomposability and $\Q$-decomposability 
equivalent conditions? Put another way: If $\A$ is decomposable over $\Q$, 
is $\h_3(\A)$ torsion-free? 
\end{question}

Of course, if Question \ref{quest:h3-tf} has a positive answer for all arrangements $\A$ 
(that is, if $\h_3(\A)$ is always torsion-free), then Question \ref{quest:decomp-Q} 
also has a positive answer, but the converse may or may not hold. An affirmative 
answer to Question \ref{quest:decomp-Q} would make the algorithms for detecting 
decomposability (implemented in the Macaulay2 packages \cite{DSS, LL}, 
but only over a field of fixed characteristic) work over $\Z$.

The following theorem completely describes the structure of the 
holonomy Lie algebra and the associated graded Lie algebra 
of a decomposable arrangement. A similar proof, in a more abstract 
setting has since been given in \cite{Lofwall-16} (see also \cite{Lofwall-20}).

\begin{theorem}[\cite{PS-cmh06}]
\label{thm:PS-decomp}
If $\A$ is a decomposable arrangement, then the following hold:
\begin{enumerate} 
\item\label{dec1}
The map 
$\h'(j_{\sharp}) \colon \h'(\A)\to \prod_{X\in L_2(\A)} \h'(\A_X)$
is an isomorphism of graded Lie algebras. 
\item\label{dec2} 
The map $\Psi_{\A}\colon \h(\A)\to \gr(G(\A))$ 
is an isomorphism of graded Lie algebras. 
\end{enumerate}
\end{theorem}

It follows from this theorem that $\h_n(\A)\cong \gr_n(G(\A))$ for all $n\ge 1$, 
and all these groups are torsion-free, with ranks $\phi_n=\phi_n(G(\A))$ given by 
\begin{equation}
\label{eq:decomp-lcs}
\prod_{n=1}^{\infty}(1-t^n)^{\phi_n}=
(1-t)^{\abs{\A}-\sum_{X\in L_2(\A)} \mu(X)}
\prod_{X\in L_2(\A)} (1- \mu(X) t)\, .
\end{equation}
Moreover, since the holonomy Lie algebra of any arrangement 
is combinatorially determined, we have the following immediate corollary.

\begin{corollary}
\label{cor:decomp-comb}
Let $\A$ and $\BB$ be two decomposable arrangements with
$L_{\le 2}(\A) \cong L_{\le 2}(\BB)$. Then $\gr_{\ge 2}(G(\A))\cong \gr_{\ge 2}(G(\BB))$.
\end{corollary}

\subsection{Nilpotent quotients and localized arrangements}
\label{subsec:nilp-decomp}

Building on the work from \cite{PS-cmh06}, we showed in \cite{PrS20} 
that the tower of nilpotent quotients of the fundamental 
group of the complement of a decomposable arrangement is fully 
determined by the intersection lattice. To explain this result, 
we start with some preparatory material on the second homology  
of these nilpotent groups. To start with, we have the following lemma, 
which is based on Theorems \ref{thm:gr3=h3}  and \ref{thm:PS-decomp}. 

\begin{lemma}[\cite{PrS20}]
\label{lem:h2-decomp}
Let $\A$ be a decomposable arrangement, with complement $M=M(\A)$ and 
group $G=G(\A)$. For every $k\ge 3$, there is a natural, split exact sequence
\begin{equation}
\label{eq:h2-decomp}
\begin{tikzcd}[column sep=16pt]
0\ar[r]
	&\h_{k}(\A) \ar[r]
	& H_2(G/\gamma_k(G);\Z)  \ar[r]
	& H_2(M;\Z)\ar[r]
	& 0 .
\end{tikzcd}
\end{equation}
\end{lemma}

For an arbitrary arrangement $\A$ and 
for a $2$-flat $X\in L_2(\A)$, we let $\A_X$ 
be the corresponding localized arrangement, and write 
$G_X=G(\A_X)$. The inclusion map $j_X\colon M(\A)\to M(\A_X)$ 
induces a homomorphism $(j_X)_{\sharp}\colon G\to G_X$ on 
fundamental groups, which in turn induces homomorphisms 
$N_k(j_X) \colon G/\gamma_n(G)\to G_X/\gamma_k(G_X)$ 
on the respective nilpotent quotients. Assembling these maps, 
we obtain homomorphisms 
\begin{equation}
\label{eq:nilp-n-map}
\begin{tikzcd}[column sep=20pt]
N_k(j) \colon G/\gamma_k(G)\ar[r] 
&\prod_{X\in L_2(\A)} G_X/\gamma_k(G_X)\, 
\end{tikzcd}
\end{equation}
for all $k\ge 1$.

\begin{proposition}[\cite{PrS20}]
\label{prop:nilp2-arr}
For any arrangement $\A$, and 
for each $k\ge 3$, the map $N_k(j)$   
induces a surjection in second homology, 
\begin{equation}
\begin{tikzcd}[column sep=20pt]
N_k(j)_*\colon H_2(G/\gamma_k(G);\Z)\ar[r, two heads] 
& \bigoplus_{X\in L_2(\A)} H_2(G_X/\gamma_k(G_X);\Z)\, .
\end{tikzcd}
\end{equation}
Moreover, if $\A$ is decomposable, then the maps $N_k(j)_*$ 
are isomorphisms, for all $k\ge 3$. 
\end{proposition}

In \cite{Rybnikov-1,Rybnikov-2} Rybnikov showed that,
in general, the third nilpotent quotient of an arrangement group is {\em not}\/ 
determined by the intersection lattice.  Specifically, he produced 
a pair of arrangements $\A$ and $\BB$, each one consisting of 
$13$ hyperplanes in $\C^3$, such that $L (\A)\cong L (\BB)$, yet 
$G(\A)/\gamma_4(G(\A))\not\cong G(\BB)/\gamma_4(G(\BB))$. 
In \cite{ACCM-2007} Artal, Carmona, Cogolludo, and Marco gave 
a different proof of this result, based on a study of the truncations 
of the Alexander invariant. 

By contrast, as shown in \cite[Theorem~8.8]{PrS20}, the phenomenon detected 
by Rybnikov cannot happen among decomposable arrangements. 

\begin{theorem}[\cite{PrS20}]
\label{thm:decomposable-n}
Suppose $\A$ and $\BB$ are decomposable arrangements such that 
$L_{\le 2}(\A) \cong L_{\le 2}(\BB)$. Then, for each $k\ge 2$, 
there is an isomorphism
\begin{equation}
\label{eq:dec-iso}
G(\A)/\gamma_{k}(G(\A))\cong G(\BB)/\gamma_{k}(G(\BB))\, .
\end{equation}
\end{theorem}

In view of this theorem and of Lemma \ref{lem:dec-comb}, all the 
nilpotent quotients of a decomposable arrangement group are 
combinatorially determined. This leaves open the following question.

\begin{question}
\label{quest:ryb}
Is the group of a decomposable arrangement combinatorially determined? 
\end{question}

Another open problem is whether decomposable arrangement groups 
are residually nilpotent. For more on this question, we refer to \cite[\S 4.5]{CFR20} 
and \cite{CF21} (see also Remark \ref{ref:CF-braids}).

\section{Constructions of decomposable arrangements}
\label{sect:ex-decomp}

We now discuss some operations that preserve decomposability, and 
outline several methods of constructing decomposable arrangements. 

\subsection{Sub-arrangements and product arrangements}
\label{subsec:sub-prod}

{\!}Decomposability behaves well with respect to some natural operations 
on arrangements. We start with a hereditary property, first established 
in \cite[Proposition~3.3]{PS-cmh06}.

\begin{prop}[\cite{PS-cmh06}]
\label{prop:heredity}
Let $\A$ be a decomposable arrangement (over $\Q$). If $\BB$ is a 
sub-arrangement of $\A$, then $\BB$ is also decomposable (over $\Q$).
\end{prop}

This gives a convenient criterion for ruling out decomposability. For instance, 
if $\A$ is an arrangement that contains the braid arrangement $\BB$ from 
Example \ref{ex:braid-holo} as a sub-arrangement, then, since $\BB$ is 
not decomposable, $\A$ is not decomposable, either. 

Next, we analyze the behavior of decomposability with respect to products. 
Let $\A$ and $\BB$ be arrangements in $\C^r$ and $\C^s$, respectively.
The product $\A\times \BB$ defined in \eqref{eq:arr-prod} is 
an arrangement in $\C^{r+s}$ with $M(\A\times \BB)\cong M(\A)\times M(\BB)$. 
By \eqref{eq:lat-prod}, the ranked poset $L(\A\times \BB)$ is isomorphic to 
$L(\A)\times L(\BB)$; therefore, 
$\widetilde{L}_2(\A\times \BB)=\widetilde{L}_2(\A)\cup \widetilde{L}_2(\BB)$. 
Moreover, by \eqref{eq:holo-arr-prod}, the graded Lie algebra 
$\h(\A\times \BB)$ is isomorphic to $\h(\A)\oplus \h(\BB)$; in particular,  
$\h_3(\A\times \BB)\cong \h_3(\A)\oplus \h_3(\BB)$. 

\begin{prop}
\label{prop:prod}
Let $\A$ and $\BB$ be two decomposable arrangements (over $\Q$). Then 
$\A\times \BB$ is also decomposable (over $\Q$). 
\end{prop}

\begin{proof}
Suppose that both $\A$ and $\BB$ are decomposable, that is, 
$\h_3(\A)$ is isomorphic to $\bigoplus_{X\in \widetilde{L}_2(\A)} \h_3(\A_X)$ and 
$\h_3(\BB)$ is isomorphic to $\bigoplus_{Y\in \widetilde{L}_2(\BB)} \h_3(\BB_Y)$. Then 
\begin{gather}
\begin{aligned}
\h_3(\A\times \BB)&\cong \h_3(\A)\oplus \h_3(\BB) \\
&\cong \bigg(\bigoplus_{X\in \widetilde{L}_2(\A)} \h_3(\A_X)\bigg)\oplus
\bigg(\bigoplus_{Y\in \widetilde{L}_2(\BB)} \h_3(\BB_Y)\bigg)\\
&\cong \bigoplus_{Z\in \widetilde{L}_2(\A\times \BB)} \h_3\big((\A\times \BB)_Z\big),
\end{aligned}
\end{gather}
and so $\A\times \BB$ is decomposable. The case when $\A$ and $\BB$ are 
decomposable over $\Q$ is treated in the same manner.
\end{proof}

\subsection{Pencils in general position}
\label{subsec:prod-pencils}
Let $m=(m_1,\dots , m_r)$ be an $r$-tuple of integers with $m_i\ge 2$, 
and let $\hat{\A}(m)$ be the arrangement hyperplanes in $\C^{r+1}$ 
defined by the polynomial 
\begin{equation}
\label{eq:def-poly-am}
f(z_0,z_1,\dots , z_r)=
z_0(z_0^{m_1}-z_1^{m_1}) (z_0^{m_2}-z_2^{m_2})  \cdots (z_0^{m_r}-z_r^{m_r}) .
\end{equation}
Observe that the projectivized complement, $U(\hat{\A}(m))$, is 
homeomorphic to the product 
$(\C\setminus \{\text{$m_1$ points}\})\times \cdots 
\times (\C\setminus \{\text{$m_r$ points}\})$.

Now let $\A(m)$ be a generic $3$-slice of this arrangement 
(said to be ``split-solvable" in \cite{CDP}).  
Its projectivization, $\bar\A(m)$, is an arrangement of lines in 
$\CP^2$, consisting of $r$ pencils meeting the line at infinity, $z_0=0$, 
in points with multiplicities  $m_1+1,\dots , m_r+1$; moreover, 
the lines of any two distinct pencils are in general position with 
each other. In an affine chart, $\bar\A(m)$ consists of $r$ packets 
of $m_i$ parallel lines, with the lines in distinct packets having different slopes. 

\begin{lemma}
\label{lem:pen-dec}
The arrangements $\A(m)$ are decomposable.
\end{lemma}

\begin{proof}
From the above description of the line arrangement $\bar\A(m)$, it is readily 
seen that $\widetilde{L}_2(\A(m))=\{X_1,\dots ,X_r\}$, where $\mu(X_i)=m_i$. 
On the other hand, by a Lefschetz-type theorem of Hamm and L\^{e}, the 
group $\pi_1(U(\A(m)))$ is isomorphic to $\pi_1(U(\hat{\A}(m)))$, which, 
by a previous observation, is isomorphic to 
$F_{m_1}\times \cdots \times F_{m_r}$. It follows that the holonomy 
Lie algebra $\h(\A(m))=\h(\pi_1(M(\A(m)))$ is isomorphic to 
$\Lie(1)\times \Lie(m_1)\times \cdots \times \Lie(m_r)$, 
from which we conclude that $\h_3(\A(m))=\bigoplus_{i=1}^r \h_3(\A(m)_{X_i})$, 
thus showing that $\A(m)$ is decomposable.
\end{proof}

Since the decomposability of an arrangement $\A$ depends only on the 
(truncated) intersection lattice $L_{\le 2}(\A)$, we obtain the following 
immediate corollary.

\begin{corollary}
\label{cor:pen-decomp}
Let $\A$ be an arrangement such that $L_{\le 2}(\A)\cong L_{\le 2}(\A(m))$, for some 
$r$-tuple $m=(m_1,\dots , m_r)$ with $m_i\ge 2$. Then $\A$ is decomposable.
\end{corollary}

Using the work of Jiang and Yau \cite{JY93, JY94} relating the topology and 
combinatorics of line arrangements with ``nice" intersection posets, 
Choudary, Dimca, and Papadima proved in \cite[Corollary~1.7]{CDP} 
the following result, which we shall need later on.

\begin{theorem}[\cite{CDP}]
\label{thm:cdp}
Let $\A$ be a central arrangement in $\C^3$. The following 
are equivalent:
\begin{enumerate}[itemsep=2pt, topsep=-1pt]
\item \label{cdp1}  
$L(\A)\cong L(\A(m))$, for some 
$r$-tuple $m=(m_1,\dots,m_r)$.
\item \label{cdp2}  
$U(\A)\cong U(\A(m))$. 
\item \label{cdp3}
$\pi_1(U(\A))$ is isomorphic to 
$F_{m_1}\times \cdots \times F_{m_r}$ 
via an isomorphism preserving the standard 
generators in $H_1$.  
\end{enumerate}
\end{theorem}

\subsection{Decomposable graphic arrangements}
\label{subsec:dec-graphic}

Let $\Gamma=(\sV,\sE)$ be a finite simplicial graph with vertex set 
$\sV=[n]\coloneqq \{1,\dots,n\}$ and edge set $\sE\subset 2^{[n]}$. 
To such a graph there corresponds a {\em graphic arrangement}, 
denoted by $\A_{\Gamma}$, which consists of the hyperplanes 
$H_e=\{z_i-z_j=0\}$ in $\C^{n}$ indexed by the edges $e=\{i,j\}$ of $\sE$. 
For example, if $\Gamma=K_{n}$, the complete graph on $n$ vertices,
then $\A_n\coloneqq \A_{K_{n}}$ is the braid arrangement in $\C^{n}$ 
from Section \ref{subsec:pi1}.
Thus, any  graphic arrangement $\A_{\Gamma}$ can be viewed as a 
sub-arrangement of the braid arrangement $\A_n$, where $n=\abs{\sV}$. 

For each $2$-flat $X\in L_2(\A_{\Gamma})$, there are either $2$ or $3$
hyperplanes containing $X$. Under the bijection between $\A_{\Gamma}$ and $\sE$ 
sending $H_e$ to $e$, a flat of size $3$ of corresponds to a triangle in the graph,
while a flat of size $2$ corresponds to a pair of edges which
is not included in any element of the triangle-set $\sT$.
Therefore, the holonomy Lie algebra $\h(\Gamma)\coloneqq\h(G(\A_{\Gamma}))$ 
is the quotient of the free Lie algebra on variables $e\in \sE$ by
the  corresponding ideal of quadratic relations,
\begin{equation}
\label{holograph}
\h(\Gamma)=\Lie(\sE)\Big\slash \text{ideal} \left\{
\begin{array}{ll}
\left[e_1 , e_2 + e_3 \right], & \text{if\,  $\{e_1,e_2,e_3\}\in \sT$}
\\
\left[ e_1, e_2 \right] , &\text{if\, $\{e_1,e_2,e\} \notin \sT,\:
\forall e\in \sE$}
\end{array}
\right\} \Big. .
\end{equation}

\begin{proposition}[\cite{PS-cmh06}]
\label{prop:graphic-decomp}
For a graphic arrangement $\A_{\Gamma}$, the following conditions are
equivalent.
\begin{enumerate}[itemsep=2pt, topsep=-1pt]
\item \label{gr1}  $\A_{\Gamma}$ is decomposable.
\item \label{gr2}  $\A_{\Gamma}$ is decomposable over $\Q$.
\item \label{gr3}  $\Gamma$ contains no complete subgraphs on $4$ vertices.
\end{enumerate}
\end{proposition}

Let $\kappa_s=\kappa_s(\Gamma)$ be the number of $K_{s+1}$
subgraphs of $\Gamma$, so that $\kappa_0=\abs{\sV}$,
$\kappa_1=\abs{\sE}$, and $\kappa_2=\abs{\sT}$.  If $\Gamma$ contains no 
$K_4$ subgraphs, i.e., $\kappa_3=0$, then the above proposition together 
with formula \eqref{eq:decomp-lcs} show that the LCS ranks 
$\phi_k=\phi_k(G(\A_{\Gamma}))$ are given by
\begin{equation}
\label{eq:graphic-decomp}
\prod_{k=1}^{\infty} (1-t^k)^{\phi_k} =
(1-t)^{\kappa_1-2\kappa_2} (1-2t)^{\kappa_2}.
\end{equation}
In fact, as conjectured in \cite{SS-tams02} and as proved in \cite{LS09}, 
the following LCS formula holds for {\em any}\/ graphic arrangement in $\C^n$,
\begin{equation}
\label{eq:lcs-graph}
\prod_{k=1}^{\infty} \left(1-t^k\right)^{\phi_k}=
\prod_{j=1}^{n-1} \left(1-jt\right)^{\sum_{s=j}^{n-1}(-1)^{s-j} 
\tbinom{s}{j} \kappa_{s}} ,
\end{equation}
or, equivalently, $\phi_k=\sum_{j=1}^{k}\sum_{s=j}^{k} (-1)^{s-j} \binom{s}{j}
\kappa_{s} \phi_k(F_j)$.

\begin{remark}
\label{ref:CF-braids}
In \cite[Theorem~5.5]{CF21}, Cohen and Falk show the following: If $\Gamma$ is $K_4$-free, 
then the graphic arrangement group $P_{\Gamma}=\pi_1(G(\A_\Gamma))$ embeds as a subgroup 
of a finite direct product of pure braid groups, and thus is residually (torsion-free) 
nilpotent. It would be interesting to know whether $P_{\Gamma}/P_{\Gamma}''$ is also 
residually nilpotent. 
\end{remark}

\subsection{More decomposable arrangements}
\label{subsec:dec-other}

We conclude this section with three more examples of decomposable arrangements 
that are not of the aforementioned types. 
More details about these well-known arrangements (the 
$\operatorname{X}_3$, $\operatorname{X}_2$, and non-Pappus arrangements) 
can be found in \cite{Su01}.

The first is the $\operatorname{X}_3$ arrangement, with defining polynomial 
$f=xyz(x+y)(x+z)(y+z)$.  Ordering the hyperplanes of $\A$ as the linear factors 
of $f$, the $2$-flats in $\widetilde{L}_2(\A)$ are $\{1, 2, 4\}$, $\{1, 3, 5\}$, and 
$\{2, 3, 6\}$.  The holonomy Lie algebra $\h(\A)$ has degree $3$ piece isomorphic 
to $\h_3^{\loc}(\A)\cong \Z^6$, with basis $[x_1 ,[x_1 ,x_4]]$, $[x_4, [x_1, x_4]]$, 
$[x_1 , [x_1, x_5]]$, $[x_5 , [x_1, x_5]]$, $[x_2,[ x_2, x_6]]$, $[x_6, [x_2 ,x_6]]$. 
Therefore, $\A$ is decomposable, 
and so the group $G(\A)$ has the same LCS ranks as the group  
$F_2^{\times 3}=F_2\times F_2\times F_2$. Nevertheless, $G(\A)$ 
is not isomorphic to a finite direct product of finitely generated free groups.  
Indeed, this arrangement group is isomorphic to the celebrated Stallings group 
(the kernel of the epimorphism $F_2^{\times 3}\to \Z$ which sends  
each standard generator to $1$), see \cite[Remark~12.4]{Su-imrn}. 
It follows that $H_3(G(\A); \Z)$ is not finitely generated, and thus $G(\A)$ 
is of type $F_2$ (finitely presented), but not of type $F_3$ (it does 
not admit a classifying space with finite $3$-skeleton).

The next example is the $\operatorname{X}_2$ arrangement, with defining polynomial 
$f=xyz(y-z)(x-z)(x+y)(x+y-2z)$. The $2$-flats in $\widetilde{L}_2(\A)$ are 
$\{1, 2, 6\}$, $\{1, 3, 5\}$, $\{2, 3, 4\}$, $\{3, 6, 7\}$, and $\{4, 5, 7\}$. 
Once again, $\h_3(\A)$ is isomorphic to $\h_3^{\loc}(\A)\cong \Z^{10}$, 
and so $\A$ is decomposable. Therefore, $\phi_k(G(\A))=\phi_k(F_2^{\times 5})$ for 
all $k\ge 2$, yet clearly $G(\A)$ is not isomorphic to a direct product of finitely 
generated free groups, since $b_1(G(\A))=7<10$.

Finally, let $\A$ be the non-Pappus arrangement; this is a realization of the 
$(9_{3})_2$ configuration of Hilbert and Cohn-Vossen, with defining polynomial 
$f=xyz(x+y)(y+z)(x+3z)(x+2y+z)(x+2y+3z)(2x+3y+3z)$. There are $9$  flats in 
$\widetilde{L}_2(\A)$: $\{1,2,4\}$, $\{1,3,6\}$, $\{1,5,9\}$, $\{2,3,5\}$, $\{2,6,8\}$,
$\{3,7,8\}$, $\{4,5,7\}$, $\{4,8,9\}$, and $\{6,7,9\}$. It is readily checked that $\h_3(\A)$ 
is isomorphic to $\h_3^{\loc}(\A)\cong \Z^{18}$, and thus $\A$ is decomposable; 
nevertheless, $G(\A)$ is not a direct product of finitely generated free groups,

\section{Alexander invariants and Chen ranks of arrangements}
\label{sect:alex-arr}

In this section we study the Alexander invariants of hyperplane arrangements, 
their relation to the ``local" Alexander invariants, and how all this informs on the 
Chen ranks of arrangement groups. 

\subsection{Alexander invariants of arrangements}
\label{subsec:alex-arr}

Let $\A$ be a complex hyperplane arrangement, with complement 
$M(\A)$ and group $G=G(\A)=\pi_1(M(\A))$.  
Following  \cite{CS-conm95, CS-tams99} (see also \cite{ACCM-2007}), 
we define the {\em Alexander invariant of $\A$}\/ as the 
Alexander invariant of its fundamental group, 
\begin{equation}
\label{eq:ba-def}
B(\A)\coloneqq B(G(\A)) = G'/G'',
\end{equation}
viewed as a module over the group ring $R=\Z[G_{\ab}]=\Z[H_1(M(\A);\Z)]$. 
Upon choosing an ordering $\A=\{H_1,\dots, H_n\}$ on the hyperplanes, 
we may identify $G_{\ab}$ with the free abelian group on the meridians 
$\{x_1,\dots , x_n\}$, and the ring $R$ with the ring of Laurent polynomials 
$\Z[t_1^{\pm 1}, \dots , t_n^{\pm 1}]$.

The Alexander invariant $B=B(\A)$ comes equipped with a filtration by the powers 
of the augmentation ideal, $I=\ker(\varepsilon \colon R\to \Z)$. The $I$-adic 
completion of the Alexander invariant, $\widehat{B}$, is a module over 
the ring $\widehat{R}$, filtered by the powers of the ideal $\widehat{I}$. 
Moreover, the associated graded modules, $\gr(B)$ and $\gr(\widehat{B})$, 
are isomorphic as modules over the graded ring $\gr(R)=\gr(\widehat{R})$, 
which may be identified with the polynomial ring $S=\Z[x_1,\dots ,x_n]$, 
with variables $x_i$ in degree $1$ corresponding to the elements 
$t_i-1\in I$.

We do not know of any arrangement $\A$ for which the Alexander invariant 
$B(\A)$ is not separated in the $I$-adic topology, but we suspect that additional 
assumptions (such as decomposability) are needed in order to ensure separability. 
We shall come back to this topic in Question \ref{quest:B-separated} (see also 
Example \ref{ex:pappus}).

\subsection{From global to local Alexander invariants}
\label{subsec:global-local}
For each $2$-flat $X\in L_2(\A)$, we have a ``local'' Alexander invariant,  
$B(\A_X)$, viewed as a module over the group ring $R_X=\Z[H_1(M(\A_X);\Z)]$.
As noted previously, $G(\A_X)\cong F_{\mu(X)}\times \Z$; thus, by Corollary \ref{cor:prod-free}, 
the module $B(\A_X)\cong B(F_{\mu(X)}\times \Z)$ is separated in the $I_X$-adic topology, 
where $I_X$ is the augmentation ideal of $R_X$.

Recall that the inclusion $j^X\colon M(\A)\inj M(\A_X)$ induces a surjective 
homomorphism $j^X_{\sharp}\colon G(\A) \to G(\A_X)$. Therefore, 
we get a surjective morphism   
\begin{equation}
\label{eq:bjx}
\begin{tikzcd}[column sep=20pt]
B(j^X_{\sharp})\colon B(\A) \ar[r]& B(\A_X)
\end{tikzcd}
\end{equation}
between the respective Alexander invariants, that covers the 
ring map $\tilde{j}^X_{*}\colon R\to R_X$ induced by the epimorphism 
$j^X_*\colon H_1(M(\A);\Z) \surj H_1(M(\A_X);\Z)$. 

Now let $j_{\sharp}\colon G(\A) \to \prod_{X\in L_2(\A)} G(\A_X)$ be 
the homomorphism induced on fundamental groups by the map 
$j\colon M(\A) \to \prod_{X\in L_2(\A)} M(\A_X)$ from \eqref{eq:prod-map-2}. 
The abelianization of this homomorphism coincides with the induced homomorphism 
$j_*\colon H_1(M(\A);\Z) \to \bigoplus_{X\in L_2(\A)} H_1(M(\A_X);\Z)$, 
which was shown in Lemma \ref{lem:j-inj} to be injective. Therefore, 
by Lemma \ref{lem:inj}, the extension of the map $j_*=(j_{\sharp})_{\ab}$ 
to group rings, 
\begin{equation}
\label{eq:jtilde}
\begin{tikzcd}[column sep=20pt]
\tilde{j}_* =(\tilde{j}^X_*)_X \colon \Z[H_1(M(\A);\Z)] \ar[r]& 
\prod_{X\in L_2(\A)} \Z[H_1(M(\A_X);\Z)]. 
\end{tikzcd}
\end{equation}
is also injective. The maps $B(j^X_{\sharp})$ from \eqref{eq:bjx} assemble into a morphism 
\begin{equation}
\label{eq:prod-bmap}
\begin{tikzcd}[column sep=22pt]
B(j_{\sharp})=(B(j^X_{\sharp}))_X\colon B(\A) \ar[r]& \bigoplus_{X\in L_2(\A)} B(\A_X)
\end{tikzcd}
\end{equation}
that covers the injective ring map $\tilde{j}_*\colon R\to \prod_{X} R_X$ 
from \eqref{eq:jtilde}. Let 
\begin{equation}
\label{eq:ba-loc}
B(\A)^{\loc} \coloneqq \bigg(\bigoplus_{X\in L_2(\A)} B(\A_X)\bigg)_{j_{\sharp}}
=\bigoplus_{X\in L_2(\A)} B(\A_X)_{j^X_{\sharp}}
\end{equation}
be the $R$-module obtained from $\bigoplus_{X} B(\A_X)$ 
by restriction of scalars along the map $\tilde{j}_*$. This module 
was first considered in \cite{CS-tams99}, where it was called the 
{\em coarse combinatorial Alexander invariant of $\A$.} 
By Remark \ref{rem:factor}, the map $B(j_{\sharp})$ from \eqref{eq:prod-bmap} 
factors as 
\begin{equation}
\label{eq:ba-baloc}
\begin{tikzcd}[column sep=23pt]
B(\A) \ar[r, "\Pi"]&  B(\A)^{\loc} \ar[r]& \bigoplus_{X\in L_2(\A)} B(\A_X),
\end{tikzcd} 
\end{equation}
where the first arrow is an $R$-linear map and the second 
arrow is the identity map of $\bigoplus_{X} B(\A_X)$, 
viewed as covering the ring map $\tilde{j}_*$.  

\subsection{Infinitesimal Alexander invariants of arrangements}
\label{subsec:linalex-arr}
Once again, let $\A$ be a hyperplane arrangement, with group $G(\A)$ 
and holonomy Lie algebra $\h(\A)=\h(G(\A))$.  We define the {\em infinitesimal 
Alexander invariant of $\A$}\/ as 
\begin{equation}
\label{eq:linba-def}
\B(\A)\coloneqq \B(G(\A))=\h'(\A)/\h''(\A),
\end{equation}
viewed as a graded module over the symmetric algebra $S=\Sym[G_{\ab}]$. 
Since $G_{\ab}=H_1(M(\A);\Z)$ and $R=\Z[H_1(M(\A);\Z)]$, the ring 
$S$ is isomorphic (as a graded ring) to $\gr(R)$. As noted previously, 
fixing an ordering $H_1,\dots, H_n$ of the hyperplanes in $\A$ 
yields an isomorphism $S\cong \Z[x_1, \dots , x_n]$.

Recall that the arrangement group $G(\A)$ is $1$-formal. Thus, by Corollary \ref{cor:linalex-gr}, 
the $S\otimes \Q$-module $\B(\A)\otimes \Q$  is isomorphic to $\gr(B(\A))\otimes \Q$.

To each flat $X\in L_2(\A)$ there corresponds a ``local'' infinitesimal  
Alexander invariant, $\B(\A_X)$, viewed as a module over the polynomial ring 
$S_X=\Sym[H_1(M(\A_X);\Z)]\cong \gr(R_X)$. 
Recall that the inclusion $j^X\colon M(\A)\inj M(\A_X)$ induces a surjective   
homomorphism $j^X_{\sharp}\colon G(\A) \surj G(\A_X)$. Therefore, 
we get an epimorphism   
$\B(j^X_{\sharp})\colon \B(\A) \surj \B(\A_X)$ that covers the 
ring map $\tilde{j}^X_*\colon S\surj S_X$ obtained by extending to symmetric algebras 
the homomorphism $j^X_*\colon H_1(M(\A);\Z) \surj H_1(M(\A_X);\Z)$. 

The maps $\B(j^X_{\sharp})$ assemble into a morphism 
$\B(j_{\sharp})=(\B(j^X_{\sharp}))_X$ 
between the modules $\B(\A)$ and $\bigoplus_{X} \B(\A_X)$ 
that covers the ring map $\bar{j}_*\colon S\to \prod_{X} S_X$ corresponding to 
$\gr(\tilde{j}_*)\colon \gr(R)\to \prod_{X} \gr(R_X)$. 
By Lemmas \ref{lem:inj} and \ref{lem:j-inj}, this ring map is injective. Let 
\begin{equation}
\label{eq:bba-loc}
\B(\A)^{\loc} \coloneqq \bigg(\bigoplus_{X\in L_2(\A)} \B(\A_X)\bigg)_{j_{\sharp}} 
=  \bigoplus_{X\in L_2(\A)} \B(\A_X)_{j^X_{\sharp}} 
\end{equation}
be the $S$-module obtained from $\bigoplus_{X} \B(\A_X)$ 
by restriction of scalars along the map $\gr(\tilde{j}_*)$. 
By Remark \ref{rem:factor}, the map $\B(j_{\sharp})$
may be viewed as the composite 
\begin{equation}
\label{eq:bba-bbaloc}
\begin{tikzcd}[column sep=23pt]
\B(\A) \ar[r, "\bar\Pi"]&  \B(\A)^{\loc} \ar[r]& \bigoplus_{X\in L_2(\A)} \B(\A_X),
\end{tikzcd} 
\end{equation}
where the first arrow is an $S$-linear map and the second 
arrow is the identity map of $\bigoplus_{X} \B(\A_X)$, 
viewed as covering the ring map $\bar{j}_*$.

\subsection{Bounding the Alexander invariants arrangements}
\label{subsec:Alex-lub}
We are now ready to state and prove the main result of this section.

\begin{theorem}
\label{thm:b-bloc-q}
For any arrangement $\A$, the following hold.
\begin{enumerate}[itemsep=3pt, topsep=-1pt]
\item \label{pi1}
The morphism $\bar\Pi \colon \B(\A) \to \B(\A)^{\loc}$ is surjective. 
\item  \label{pi2}
The morphism $\widehat{\Pi}\otimes \Q\colon \widehat{B(\A)}\otimes \Q\to 
\widehat{B(\A)^{\loc}}\otimes \Q$ 
is surjective. 
\item  \label{pi3}
The morphism $\Pi\otimes \Q\colon B(\A)\otimes \Q\to B(\A)^{\loc}\otimes \Q$ 
is surjective. 
\end{enumerate}
\end{theorem}

\begin{proof}
Let $j\colon M(\A) \to \prod_{X\in L_2(\A)} M(\A_X)$ be the 
map from \eqref{eq:prod-map-2}. The corresponding homomorphism on 
fundamental groups, $j_{\sharp}\colon G(\A) \to \prod_{X} G(\A_X)$, 
induces a map of graded Lie algebras, 
\begin{equation}
\label{eq:hjsharp-map}
\begin{tikzcd}[column sep=22pt]
\h(j_{\sharp})\colon \h(\A) \ar[r]& \h\big(\!\prod_{X} \A_X\big)\cong \prod_{X} \h(\A_X),
\end{tikzcd}
\end{equation}
where the isomorphism on the right is the one provided by Lemma \ref{lem:holo-prod}. 
By Proposition \ref{prop:PS-holo-prod}, the restriction 
$\h'(j_{\sharp}) \colon\h'(\A) \to \prod_{X} \h'(\A_X)$ 
is surjective. This map induces a map on quotients, 
\begin{equation}
\label{eq:h-map}
\begin{tikzcd}[column sep=22pt]
\h'(\A)/\h''(\A) \arrow[two heads]{r}& \prod_{X\in L_2(\A)} \h'(\A_X)/\h''(\A_X),
\end{tikzcd}
\end{equation}
which also must be surjective.  Observe that the map \eqref{eq:h-map} 
coincides with the composite map from \eqref{eq:bba-bbaloc}. 
Therefore, the $S$-morphism $\bar\Pi \colon\B(\A) \to \B(\A)^{\loc}$ 
is surjective, and this proves claim \eqref{pi1}.

Let us pass now to completions. Since arrangement groups are 
$1$-formal, Theorem \ref{thm:linalex-c} gives a 
filtration-preserving isomorphism of $\widehat{R}\otimes \Q$-modules,  
$\Phi_{\A}\colon \widehat{B(\A)}\otimes \Q \isom \widehat{\B(\A)}\otimes \Q$, and, 
for each $X\in L_2(\A)$, 
an isomorphism of $\widehat{R}_X\otimes \Q$-modules, 
$\Phi_{\A_X}\colon \widehat{B(\A_X)}\otimes \Q \isom \widehat{\B(\A_X)}\otimes \Q$. 
Furthermore, these isomorphisms are compatible; indeed, by \eqref{eq:bb-func}, we 
have a commutative diagram,
\begin{equation}
\label{eq:bb-bx}
\begin{tikzcd}[column sep=42pt, row sep=26pt]
\widehat{B(\A)}\otimes \Q  \ar[r, "\widehat{B(j_{\sharp}^X)}\otimes \Q"]
\ar[d, "\Phi_{\A}", "\cong"']
& \widehat{B(\A_X)}\otimes \Q \phantom{.}\ar[d, "\Phi_{\A_X}" , "\cong"']\\
\widehat{\B(\A)}\otimes \Q \ar[r, "\widehat{\B(j_{\sharp}^X)}\otimes \Q"]
& \widehat{\B(\A_X)}\otimes \Q 
\end{tikzcd}
\end{equation}
with horizontal arrows covering the ring map 
$\widehat{\tilde{j}}^{\,X}_{*}\colon \widehat{R}\surj \widehat{R}_X$. 
Taking direct sums over all the $2$-flats of $\A$, and letting 
$\Phi_{\A}^{\loc}\coloneqq \bigoplus_{X\in L_2(\A)} \Phi_{\A_X}$, 
we obtain the following commutative diagram in the category 
of $\widehat{R}\otimes \Q$-modules,
\begin{equation}
\label{eq:bb-bloc}
\begin{tikzcd}[column sep=42pt, row sep=26pt]
\widehat{B(\A)}\otimes \Q  \ar[r, "\widehat{\Pi}\otimes \Q"]
\ar[d, "\Phi_{\A}", "\cong"']
& \widehat{B(\A)}^{\loc}\otimes \Q \phantom{.}\ar[d, "\Phi_{\A}^{\loc}" , "\cong"']\\
\widehat{\B(\A)}\otimes \Q \ar[r, "\widehat{\bar\Pi}\otimes \Q"]
& \widehat{\B(\A)}^{\loc}\otimes \Q .
\end{tikzcd}
\end{equation}
By part \eqref{pi1} and Lemma \ref{lem:b-surj}, the bottom arrow in this diagram 
is surjective. Thus, the top arrow is also surjective, and this proves claim \eqref{pi2}.

Finally, using again  Lemma \ref{lem:b-surj}, it follows from part \eqref{pi2} 
that the map $\Pi\otimes \Q$ is an epimorphism of $R\otimes \Q$-modules;  
this proves claim \eqref{pi3} and completes the proof.
\end{proof}

\begin{remark}
\label{rem:cs-alex}
An explicit description of the $R$-morphism $\Pi\colon B(\A) \to  B(\A)^{\loc}$ 
was given in \cite[Theorem~6.3]{CS-tams99}. This description involves certain 
presentations for the respective Alexander invariants, obtained by means 
of the Fox Calculus applied to the braid monodromy presentations of the  
groups $G(\A)$ and $G(\A_X)$. That approach (which requires considerably more 
work), shows that, in fact, the map $\Pi\colon B(\A)\to B(\A)^{\loc}$ itself is surjective, 
not just the rationalization $\Pi\otimes \Q$ from Theorem \ref{thm:b-bloc-q}, part \eqref{pi3}.
Moreover, it was shown in \cite[Corollary~6.6]{CS-tams99} that the $\widehat{R}$-module 
$\widehat{B(\A)}$ admits an (explicit) presentation with 
$\binom{b_1}{2}-b_2$ generators and $\binom{b_1}{3}$ relations, 
where the Betti numbers $b_k=b_k(M(\A))$ are given by \eqref{eq:betti-arr}. No such 
presentation is known in general for either $B(\A)$ or $\gr(B(\A))$.
\end{remark}

\subsection{Bounding the Chen ranks of arrangements}
\label{subsec:Chen-lub}
The study of the Alexander invariants of arrangement groups undertaken in 
\cite{CS-tams99} led to lower bounds for the Chen ranks of those 
groups. Notably, those bounds were expressed solely in terms of 
the M\"{o}bius function of the intersection lattice of the arrangement. 
We present here a different approach to obtaining those bounds, 
without appealing to explicit presentations for the (completed) 
Alexander invariants of arrangement groups.

To start with, recall that arrangement groups are $1$-formal. 
Hence, as a direct consequence of Corollary \ref{cor:ps-chen}, 
we obtain the following result, which is the content of 
\cite[Theorem~11.1]{PS-imrn04}.

\begin{corollary}[\cite{PS-imrn04}]
\label{cor:chen-arr}
The generating series for the Chen ranks of an arrangement group $G(\A)$ is 
given by 
\begin{equation}
\label{eq:chen-arr}
\sum_{k\ge 0} \theta_{k+2} (G(\A)) t^k = \Hilb (\B(\A) \otimes \Q,t).
\end{equation}
Consequently, the Chen ranks of $G(\A)$ are determined 
by $L_{\le 2}(\A)$.
\end{corollary}

To each flat $X\in L_2(\A)$, there corresponds a ``local'' infinitesimal  
Alexander invariant, $\B(\A_X)$, viewed as a module over the polynomial ring 
$S_X=\Sym[H_1(M(\A_X);\Z)]\cong \gr(R_X)$. Recall from \eqref{eq:gax}
that the group $G(\A_X)$ is isomorphic to $F_{\mu(X)}\times \Z$; in particular, 
$\theta_1(G(\A_X))=\mu(X)+1$. For $k\ge 2$,  Corollary \ref{cor:chen-arr}, 
together with the computations from Examples \ref{ex:alex-free} 
and \ref{ex:alex-free-bis}, yields the equalities
\begin{equation}
\label{eq:chen-gx}
\theta_k(G(\A_X))= \dim_{\Q} (\B_{k-2}(\A_X)\otimes \Q )
= (k-1)\binom{\mu(X)+k-2}{k}.
\end{equation}
Note that if $\mu(X)=1$, then $G(\A_X) \cong \Z^2$ and $\theta_k(G(\A_X))=0$ 
for all $k\ge 2$.

The next result recovers \cite[Corollary~7.2]{CS-tams99} 
using our alternative approach, based on infinitesimal Alexander invariants. 

\begin{corollary}[\cite{CS-tams99}]
\label{cor:chen-lub}
The Chen ranks of a hyperplane arrangement $\A$ admit the lower bound
\begin{equation}
\label{eq:chen-bound}
\theta_k(G(\A)) \ge \sum_{X\in \widetilde{L}_2(\A)}  \theta_k(G(\A_X)) = 
(k-1) \sum_{X\in \widetilde{L}_2(\A)} \binom{\mu(X)+k-2}{k} ,
\end{equation}
valid for all $k\ge 2$, with equality for $k=2$.
\end{corollary}

\begin{proof}
Fix an integer $k\ge 2$. Then,
\begin{gather}
\label{eq:chain}
\begin{aligned}
\theta_k(G(\A)) 
&= \dim_{\Q} \B_{k-2}(\A)\otimes \Q
&&\text{by Corollary \ref{cor:chen-arr}}\\
&\ge \dim_{\Q} \B_{k-2}(\A)^{\loc}\otimes \Q
&&\text{by Theorem \ref{thm:b-bloc-q}, part \eqref{pi1}}\\
&=  \sum_{X\in L_2(\A)} \dim_{\Q} \B_{k-2}(\A_X)\otimes \Q
&&\text{by  \eqref{eq:bba-loc}} \\
&=(k-1) \sum_{X\in \widetilde{L}_2(\A)} \binom{\mu(X)+k-2}{k} 
&&\text{by \eqref{eq:chen-gx}.} 
\end{aligned}
\end{gather}
Finally, recall that $\theta_k(G(\A))=\phi_k(G(\A))$ for $k\le 3$; thus, 
by Corollary \ref{cor:lcs-bound}, equality holds in \eqref{eq:chen-bound} 
when $k=2$.
\end{proof}

\section{Decomposable Alexander invariants}
\label{sect:alex-decomp}

In this section, we investigate the relationship between the decomposability of 
an arrangement $\A$, defined in terms of the holonomy Lie algebra $\h(\A)$, 
and the decomposability of the Alexander-type invariants $B(\A)$, $\widehat{B(\A)}$, and 
$\B(\A)$.

\subsection{Decomposable Alexander invariants}
\label{subsec:decomp-alex}
Staying with the notations from the previous section, we now make several definitions.

\begin{definition}
\label{def:alex-decomp}
We say that the Alexander invariant of a hyperplane arrangement $\A$ is 
{\em decomposable}\/ if the map $\Pi\colon B(\A)\to B(\A)^{\loc}$ is an 
isomorphism of $R$-modules. 
\end{definition}

A similar definition works over the rationals: the Alexander invariant 
is {\em decomposable over $\Q$}\/ if the map $\Pi\otimes \Q\colon 
B(\A)\otimes \Q\to B(\A)^{\loc}\otimes \Q$ 
is an isomorphism of modules over the ring $R\otimes \Q$. 
By Theorem \ref{thm:b-bloc-q}, part \eqref{pi3}, this condition is equivalent to 
$\Pi\otimes \Q$ being injective. 

Following \cite[\S6.4]{CS-tams99}, we say that the completion of the 
Alexander invariant is {\em decomposable}\/ if the map $\widehat{\Pi}\colon 
\widehat{B(\A)}\to \widehat{B(\A)^{\loc}}$ 
is an isomorphism of modules over $\widehat{R}$. 
A similar definition works over $\Q$; 
by Theorem \ref{thm:b-bloc-q}, part \eqref{pi2}, the $\Q$-decomposability 
of $\widehat{B(\A)}$ is equivalent to the injectivity of $\widehat{\Pi}\otimes \Q$. 
Clearly, if $B(\A)$ is decomposable (over $\Z$ or over $\Q$), then $\widehat{B(\A)}$ 
is decomposable (over $\Z$ or over $\Q$).

\begin{definition}
\label{def:linalex-decomp}
We say that the infinitesimal Alexander invariant is {\em decomposable}\/ 
if the map $\bar\Pi\colon \B(\A)\to \B(\A)^{\loc}$ is an isomorphism of $S$-modules. 
\end{definition}

Likewise, the infinitesimal Alexander invariant of $\A$ is {\em decomposable 
over $\Q$}\/ if the map $\bar\Pi\otimes \Q \colon \B(\A)\otimes \Q\to 
\B(\A)^{\loc}\otimes \Q$ is an isomorphism of modules over $S\otimes \Q$. 
By Theorem \ref{thm:b-bloc-q}, part \eqref{pi1}, this condition is equivalent to 
$\bar\Pi\otimes \Q$ being injective. Evidently, the property of $\B(\A)$ being 
decomposable (or $\Q$-decomposable) depends only on the intersection lattice 
of the arrangement. 

These definitions raise the following problem, which is the rough analogue of 
Question \ref{quest:decomp-Q} in this context. 

\begin{question}
\label{quest:B-decomp}
For an arrangement $\A$, are the decomposability and $\Q$-decom\-pos\-abil\-ity properties 
of $B(\A)$ (and likewise for $\widehat{B(\A)}$ and $\B(\A)$) equivalent? 
\end{question} 

There is a connection between decomposability and separability of the 
Alexander invariant of an arrangement, that we now spell out.

\begin{proposition}
\label{prop:decomp-sep}
Let $\A$ be an arrangement. If $B(\A)$ is decomposable (over $\Q$), then 
$B(\A)$ is separated (over $\Q$).
\end{proposition}

\begin{proof}
First suppose the Alexander invariant is decomposable, that is, $B(\A)$ is isomorphic to 
$B(\A)^{\loc} = \bigoplus_{X\in L_2(\A)} B(\A_X)_{j_{\sharp}^X}$. 
Since $G(\A_X)=F_{\mu(X)}\times \Z$, it follows from Corollary \ref{cor:prod-free}
and Lemma \ref{lem:sep-induced} that $B(\A_X)_{j_{\sharp}^X}$ is separated.
Since completion commutes with direct sums, we conclude that $B(\A)$ is separated. 
The claim over $\Q$ is proved in a similar fashion.
\end{proof}

\subsection{Alexander invariants of decomposable arrangements}
\label{subsec:alex-decomp}
We are now in a position to state and prove the main results of this section.
We start with a lemma.

\begin{lemma}
\label{lem:gr1b-h3}
For every arrangement $\A$, there are isomorphisms 
\[
\gr_1(B(\A))\cong \B_1(\A) \cong \h_3(\A).
\]
\end{lemma}

\begin{proof}
Set $G=G(\A)$. By Theorem \ref{thm:gr3h3-arr}, Lemma \ref{lem:griso}, and 
Corollary \ref{cor:alex-chen}, respectively, we have natural isomorphisms
\begin{equation}
\label{eq:3isos}
\begin{tikzcd}[column sep=19pt]
\h_3(G) \ar[r, "\cong"] & \gr_3(G) \ar[r, "\cong"]  &  \gr_3(G/G'') \cong \gr_1(B(G)).
\end{tikzcd}
\end{equation}
Moreover, by the grading convention \eqref{eq:bb grading}, we have the equality 
$\B_1(G)=\h_3(G)$, and this completes the proof.
\end{proof}

\begin{theorem}
\label{thm:alex-decomp}
Let $\A$ be a hyperplane arrangement.
\begin{enumerate}[itemsep=2pt, topsep=-1pt]
\item  \label{bd1} 
If $B(\A)$ is decomposable (over $\Q$), then $\A$ is decomposable (over $\Q$).
\item  \label{bd2} 
If $\B(\A)$ is decomposable (over $\Q$), then $\A$ is decomposable (over $\Q$).
\end{enumerate}
\end{theorem}

\begin{proof}
Suppose $B(\A)$ is decomposable, that is, the map $\Pi\colon B(\A)\to B(\A)^{\loc}$ 
is an isomorphism. Passing to associated graded and taking degree $1$ pieces, 
we obtain an isomorphism, $\gr_1(\Pi)\colon \gr_1(B(\A))\to \gr_1(B(\A)^{\loc})$. 
Using the naturality of the isomorphisms from \eqref{eq:3isos}, we conclude that 
the map $\h_3(G)\to \h_3(G)^{\loc}$ is also an isomorphism, that is, $\A$ is decomposable. 
Similar arguments work for $\B(\A)$ and over $\Q$.
\end{proof}

Note that the rational versions of statements \eqref{bd1} and \eqref{bd2} above 
are equivalent, since $\gr(B(\A))\otimes \Q\cong \B(A)\otimes \Q$, 
but the integral versions of the two statements are {\it a priori}\/ different.
We now consider the reverse implications of these statements. 

\begin{theorem}
\label{thm:decomp-alex}
Let $\A$ be a hyperplane arrangement. 
\begin{enumerate}[itemsep=2pt, topsep=-1pt]
\item  \label{b1} 
If $\A$ is decomposable, then $\B(\A)$ is decomposable.
\item  \label{b2} 
If $\A$ is $\Q$-decomposable, then both $\B(\A)$ and $\widehat{B(\A)}$  
are $\Q$-decomposable.
\item  \label{b3} 
If $\A$ is $\Q$-decomposable and $B(\A)\otimes \Q$ is 
separated, then $B(\A)$ is $\Q$-decom\-posable.
\end{enumerate}
\end{theorem}

\begin{proof}
First suppose $\A$ is decomposable. 
Let $j\colon M(\A) \to \prod_{X\in L_2(\A)} M(\A_X)$ be the 
map from \eqref{eq:prod-map-2}. The homomorphism 
$j_{\sharp}\colon G(\A) \to \prod_{X} G(\A_X)$ 
induces a map of graded Lie algebras, 
$\h(j_{\sharp})\colon \h(\A) \to \prod_{X} \h(\A_X)$. 
Theorem \ref{thm:PS-decomp}, part \eqref{dec1} 
ensures that the restriction $\h'(j_{\sharp}) \colon\h'(\A) \to \prod_{X} \h'(\A_X)$ 
is an isomorphism. It follows that the induced map on quotients, 
$\h'(\A)/\h''(\A) \to \prod_{X} \h'(\A_X)/\h''(\A_X)$, 
is an isomorphism of graded abelian groups.  As noted previously, 
this map coincides with the morphism of (graded) $S$-modules 
$\bar\Pi$ from \eqref{eq:bba-bbaloc}.  Therefore, the map 
$\bar\Pi \colon\B(\A) \to \B(\A)^{\loc}$ is an isomorphism of 
$S$-modules, and this proves claim \eqref{b1}.

To prove the last two claims, suppose $\A$ is $\Q$-decomposable. 
Proceeding as above, we infer that the map 
$\bar\Pi\otimes \Q \colon\B(\A)\otimes \Q\to  \B(\A)^{\loc}\otimes\Q$ 
is an isomorphism of $S\otimes \Q$-modules. Therefore, by 
Lemmas \ref{lem:b-surj} and \ref{lem:b-inj}, 
its completion, $\widehat{\bar\Pi}\otimes \Q$, is also an isomorphism. 
This map is the bottom arrow from the commuting square \eqref{eq:bb-bloc}, 
in which the side arrows are also isomorphisms. Therefore, the top arrow, 
$\widehat{\Pi}\otimes \Q\colon 
\widehat{B(\A)}\otimes \Q\to \widehat{B(\A)^{\loc}}\otimes \Q$  
is an isomorphism of $\widehat{R}\otimes \Q$-modules. 
If, in addition, the $R\otimes \Q$-module $B(\A)\otimes \Q$ is 
separated, then, by Lemmas \ref{lem:b-surj} and \ref{lem:b-inj} again, 
the map $\Pi\otimes \Q\colon B(\A)\otimes \Q\to B(\A)^{\loc}\otimes \Q$  
is an isomorphism. This completes the proof of claims \eqref{b2} and \eqref{b3}.
\end{proof}

Using different methods, it was shown in \cite[Theorem~7.9]{CS-tams99} that the following 
implication also holds.

\begin{theorem}[\cite{CS-tams99}]
\label{thm:cs-dec}
If $\A$ is decomposable, then $\widehat{B(\A)}$ is decomposable. 
\end{theorem}

Putting things together, we obtain the following corollary.

\begin{corollary}
\label{cor:dec-sep}
For an arrangement $\A$, the following are equivalent.
\begin{enumerate}[itemsep=2.5pt, topsep=-1pt]
\item \label{ba1} 
$B(\A)$ is decomposable (over $\Q$).
\item \label{ba2}
$\A$ is decomposable and $B(\A)$ is separated (over $\Q$).

\end{enumerate}
\end{corollary}

\begin{proof}
The implication \eqref{ba1} $\Rightarrow$ \eqref{ba2} (over either $\Z$ or $\Q$) 
follows at once from Proposition \ref{prop:decomp-sep} and Theorem \ref{thm:alex-decomp},  
part \eqref{bd1}.

The implication \eqref{ba2} $\Rightarrow$ \eqref{ba1} over $\Q$ is the content of 
Theorem \ref{thm:decomp-alex}, part \eqref{b3}. To prove the same implication over $\Z$, 
we proceed in a similar manner. To start with, since $\A$ is decomposable, 
Theorem \ref{thm:cs-dec} ensures that  $\widehat{B(\A)}$ is decomposable, 
that is, the map  $\widehat{\Pi}\colon \widehat{B(\A)}\to \widehat{B(\A)^{\loc}}$  
is an isomorphism of $\widehat{R}$-modules. Now, since 
$B(\A)$ is assumed to be separated, Lemmas \ref{lem:b-surj} and \ref{lem:b-inj} 
imply that the map $\Pi\colon B(\A)\to B(\A)^{\loc}$ is an isomorphism of $R$-modules, 
and we are done.
\end{proof}

\subsection{Discussion and examples}
\label{subsec:discuss-dec-alex}
The third part of Theorem \ref{thm:decomp-alex} raises the following basic question 
regarding the Alexander invariants of decomposable arrangements.

\begin{question}
\label{quest:B-separated}
Let $\A$ be an arrangement.
\begin{enumerate}[itemsep=2pt, topsep=-1pt]
\item \label{q1} 
Suppose $\A$ is decomposable. Is then the Alexander invariant 
$B(\A)$ separated in the $I$-adic topology?
\item \label{q2} 
If $\A$ is decomposable over $\Q$, is $B(\A)\otimes \Q$ separated?
\end{enumerate}
\end{question}

If the answer to part \eqref{q2} were to be yes, then we could dispense with 
the separation hypothesis in Theorem \ref{thm:decomp-alex}, part \eqref{b3}, 
and conclude that $\Q$-decomposability of $\A$ implies $\Q$-decomposability 
of $B(\A)$. 

The following result gives a combinatorial criterion for deciding the decomposability 
of the Alexander invariants of a class of arrangements.

\begin{proposition}
\label{prop:pend-b-dec}
Let $\A$ be an arrangement such that $L_{\le 2}(\A)\cong L_{\le 2}(\A(m))$, 
for some $r$-tuple $m=(m_1,\dots,m_r)$ with $m_i\ge 2$. Then $B(\A)$ is 
decomposable and separated.
\end{proposition}

\begin{proof}
By Theorem \ref{thm:cdp}, the group $G(\A)$ is isomorphic to 
$G(\A(m))=F_{m_1}\times \cdots \times F_{m_r} \times \Z$. 
By Lemma \ref{lem:alex-prod}, 
$B(\A(m))\cong B(F_{m_1})_{p_1}\oplus \cdots \oplus B(F_{m_r})_{p_r}$, 
where $p_i$ are the projection maps of $G(\A(m))$ onto its factors. 
On the other hand, recall that $\widetilde{L}_2(\A(m))=\{X_1,\dots, X_r\}$, 
with $\mu(X_i)=m_i$. Consequently, $B(F_{m_i})\cong B(\A_{X_i})$, and 
so $B(\A(m))\cong B(\A(m))^{\loc}$, showing that $B(\A(m))$ is decomposable. 
By Corollary \ref{cor:prod-free}, $B(\A(m))$ is also separated. 
The claim follows.
\end{proof}

\begin{remark}
\label{rem:am-dec}
In view of Theorem \ref{thm:alex-decomp}, this proposition shows that the 
arrangements $\A(m)$ are decomposable, thus giving a different proof 
of Lemma \ref{lem:pen-dec}. Alternatively, one can use Lemma \ref{lem:pen-dec}
together with Theorem \ref{thm:cdp} and Corollary \ref{cor:dec-sep} to give 
another proof of Proposition \ref{prop:pend-b-dec}. 
\end{remark}

\subsection{Chen ranks of $\Q$-decomposable arrangements}
\label{subsec:chen-decomp}

As an application of Theorem \ref{thm:decomp-alex}, we obtain an explicit formula 
for the Chen ranks of an arrangement $\A$, provided that $\A$ is decomposable over $\Q$.
Similar formulas were given in \cite[Theorem~7.9]{CS-tams99} and \cite[Theorem~6.2]{PS-cmh06} 
under the (possibly stronger) assumption that $\A$ is decomposable. 

\begin{corollary}
\label{eq:chen-decomp}
Let $\A$ be a $\Q$-decomposable arrangement. Then the Chen ranks 
of $G(\A)$ are given by $\theta_1(G(\A)) =\abs{\A}$ and 
\[
\theta_k(G(\A)) = \sum_{X\in \widetilde{L}_2(\A)}  \theta_k(G(\A_X)) =
 (k-1) \sum_{X\in \widetilde{L}_2(\A)} \binom{\mu(X)+k-2}{k} 
\]
for all $k\ge 2$.
\end{corollary}

\begin{proof}
The proof is similar to that of Corollary \ref{cor:chen-lub}. 
The only difference occurs on the second line of display \eqref{eq:chain},
where we need to replace the inequality provided by 
Theorem \ref{thm:b-bloc-q}, part \eqref{pi1}
with the equality provided by Theorem \ref{thm:decomp-alex}, 
part \eqref{b2}.
\end{proof}

\section{Cohomology jump loci and decomposability}
\label{sect:cjl-decomp}

In this section, we briefly review the  (degree $1$) resonance and characteristic 
varieties of arrangements, and then describe those varieties in a decomposable setting.

\subsection{Characteristic varieties}
\label{subsec:cvs}
Let $G$ be a finitely generated group. 
The character group, $\TT_G=\Hom(G,\C^*)$, is an abelian, 
complex  algebraic group, with identity $\bo$ the trivial representation. The  
coordinate ring of $\TT_G$ is the group algebra $\C[G_{\ab}]$;  
thus, we may identify $\TT_G$ with $\Spec(\C[G_{\ab}])$, 
the maximal spectrum of this $\C$-algebra. Since each 
character $\rho\colon G\to \C^*$ factors through the abelianization 
$G_{\ab}$, the map $\ab\colon G\surj G_{\ab}$ induces an isomorphism, 
$\ab^*\colon \TT_{G_{\ab}} \isom \TT_{G}$.  

Let $X$ be a connected CW-complex with finite $1$-skeleton 
and with $\pi_1(X)=G$. Upon identifying a point 
$\rho \in \TT_G=H^1(X;\C^*)$ with a rank one local 
system $\C_{\rho}$ on $X$, we define for each $s\ge 1$ the 
{\em depth $s$ characteristic variety}\/ of $X$ as 
\begin{equation}
\label{eq:cvar}
\VV_s(X)\coloneqq \{ \rho \in H^1(X;\C^*) \mid  
\dim_{\C} H_1(X; \C_{\rho}) \ge s \} .  
\end{equation}
Clearly, $\bo\in \VV_s(X)$ if and only if $b_1(X)\ge s$. 
Since a classifying space $K(G,1)$ may be constructed by attaching to $X$ 
cells of dimension $3$ and higher, it is straightforward to verify that  $\VV_s(X)$ 
coincides with $\VV_s(G)\coloneqq \VV_s(K(G,1))$, for all $s\ge 1$. 

Every group homomorphism $\alpha\colon G\to H$ 
induces a morphism between character groups, 
$\alpha^*\colon \TT_H\to \TT_G$, 
given by $\alpha^* (\rho)(g)=\alpha(\rho(g))$.   
Now suppose $G$ is finitely generated and 
$\alpha$ is surjective. Then the morphism $\alpha^*$ 
is injective, and it restricts to an embedding 
$\VV_s(H) \inj \VV^1_s(G)$, for each $s\ge 1$.

It has been known since the work of Libgober \cite{Li92} and 
E.~Hironaka \cite{Hi97} that the characteristic and Alexander 
varieties of spaces and groups are intimately related. We record 
here the result we shall need; a proof valid in all depths $s\ge 1$ 
was given in \cite{Su-abexact}. 

\begin{theorem}
\label{thm:cvb}
Let $G$ be a finitely generated group. Then, for all $s\ge 1$, 
\begin{equation}
\label{eq:cvar-ann}
\V_s(G)=\supp \big(\bwedge^k B(G)\otimes \C \big) ,
\end{equation}
at least away from the identity $\bo\in \T_{G}$.
\end{theorem}

Part of the importance of the characteristic varieties lies in the fact that 
they control the Betti numbers of regular, finite abelian covers $p\colon Y\to X$.  
For instance, suppose that the deck-trans\-formation group is 
cyclic of order $N$, classified by an 
epimorphism $\chi \colon G=\pi_1(X)\surj \Z_N$. 
The induced morphism between character groups, 
$\chi^*\colon \T_{\Z_N}\to \T_{G}$, is injective, 
and so $\im(\chi^*)\cong \Z_N$. 
A result proved in various levels of generality in 
\cite{Li92, Sa95, Hi97, MS02, DeS-plms} now shows that 
\label{eq:b1-cover}
\begin{equation}
b_1(Y)=\sum_{s\ge 1}\abs{\im(\chi^*) \cap  \V_s(X)}.
\end{equation}

\subsection{Resonance varieties}
\label{subsec:res}

Let $G$ be a group, and let $H^{*}=H^{*}(G;\C)$ be its cohomology 
algebra over $\C$.  For our purposes here, we will only consider 
the truncated algebra $H^{\le 2}$; moreover, 
we will assume that $b_1(G)=\dim_{\C} H^1$ is finite.
For each element $a\in H^1$, we have $a^2=0$, 
and so left-multiplication by $a$ defines a cochain complex, 
\begin{equation}
\label{eq:aomoto}
\begin{tikzcd}[column sep=26pt]
(H , \delta_a)\colon  \ 
H^0\ar[r, "\delta^0_a"] & H^1\ar[r, "\delta^1_a"]& H^2 , 
\end{tikzcd}
\end{equation}
with differentials $\delta^i_a(u)=a\cdot u$ for $u\in H^i$. 
The  resonance varieties witness the extent to which this 
complex fails to be exact in the middle. More precisely, for each $s\ge 1$, 
the {\em depth $s$ resonance variety}\/ of $G$ is defined as 
\begin{equation}
\label{eq:rv}
\RR_s(G)\coloneqq \{a \in H^1 \mid \dim_{\C} H^1(H, \delta_a) \ge s\}.
\end{equation}

These sets are homogeneous algebraic subvarieties of the affine space 
$H^1=H^1(G;\C)$. Clearly, $\bz\in \RR_k(G)$ 
if and only if $s\le b_1(G)$; in particular, $\RR_1(G)= \emptyset $ 
if and only if $b_1(G)=0$.  Furthermore, we have a descending filtration, 
\begin{equation}
\label{eq:res-strat}
H^1(G;\C) \supseteq \RR_1(G)  \supseteq \RR_2(G)  \supseteq \cdots 
 \supseteq \RR_{b}(G) \supseteq \RR_{b+1}(G)=\emptyset  ,
\end{equation}
where $b=b_1(G)$.  

The next result identifies the depth-$s$ resonance variety of a group as the 
support locus of the $s$-th exterior power of its infinitesimal Alexander invariant. 

\begin{theorem}[\cite{DPS-serre, DP-ann, Su-abexact}]
\label{thm:res-supp}
Let $G$ be a group with $b_1(G)<\infty$. Then 
\begin{equation}
\label{eq:res-supp}
\RR_s(G) = \supp \big( \bwedge^s \B(G)\otimes \C\big) 
\end{equation}
for all $s\ge 1$, at least away from $\bz\in H^1(G;\C)$.
\end{theorem}

Since $\B(G)=\h'(G)/\h''(G)$, we infer from Theorem \ref{thm:res-supp} that the resonance 
varieties of $G$ only depend on the holonomy Lie algebra $\h(G)$. More precisely, 
let $G_1$ and $G_2$ be two groups with finite first Betti number, and suppose 
that $\h(G_1; \C)\cong \h(G_2;\C)$, as graded Lie algebras. 
There is then a linear isomorphism, $H^1(G_1;\C) \cong H^1(G_2;\C)$, 
restricting to isomorphisms $\RR_s(G_1)\cong \RR_s(G_2)$ for all $s\ge 1$.

We shall identify the tangent space to 
the character group $\T_G=H^1(G;\C^*)$ with the linear space $H^1(G;\C)$, 
and we will denote by $\TC_{\bo}(\V_s(G))$ the tangent cone at the identity 
$\bo$ to the characteristic variety $\V_s(G)$. 
It is known that $\TC_{\bo}(\V_s(G))$ is always a (homogeneous) subvariety 
of the resonance variety $\RR_s(G)$. The basic relationship between the 
characteristic and resonance varieties in the $1$-formal setting is 
encapsulated in the ``Tangent Cone formula" from \cite[Theorem~A]{DPS-duke}, 
which we recall next.

\begin{theorem}[\cite{DPS-duke}]
\label{thm:tcone}
If $G$ is a $1$-formal group, then $\TC_{\bo}(\VV_s(G))=\RR_s(G)$, 
for all $s\ge 1$.
\end{theorem}

\subsection{Cohomology jump loci of arrangements}
\label{subsec:cv-res}
Let $\A$ be a complex hyperplane arrangement.  
Since complement $M=M(\A)$  is a smooth, quasi-projective variety, work of 
Arapura \cite{Ar} shows that the characteristic varieties $\V_s(M)$ are finite 
unions of torsion-translates of algebraic subtori of 
$H^1(M;\C^*)\cong (\C^*)^{\abs{\A}}$.  
Furthermore, the projection map $\pi\colon M\surj \P(M)=U$ 
induces a monomorphism $H^1(U;\C^*)\inj H^1(M;\C^*)$ that restricts to 
isomorphisms $\V_s(U)\isom \V_s(M)$.  Since $M$ is also a formal space, 
each resonance variety $\RR_s(M)$ coincide with the tangent cone 
at the trivial character to $\V_s(M)$, see \cite{CS99, DPS-duke}.  
As shown in \cite{FY}, building on work from \cite{Fa97, LY00}, 
the varieties $\RR_s(M)$ may be described solely in terms of multinets 
on sub-arrangements of $\A$.  In general, though, the varieties $\V_s(M)$
may contain components that do not pass through the origin, see 
\cite{Su02, CDS, DeS-plms}. 

Now let $X\in L_2(\A)$ be a rank-$2$ flat, and let $\A_X$ be localization of $\A$ at $X$. 
Since $\pi_1(U(\A_X))=F_{\mu(X)}$, a standard computation shows that, for each $s\ge 1$, 
\begin{equation}
\label{eq:cv-free}
\V_s(U(\A_X))=\begin{cases}
H^1(U(\A_X);\C^*)\cong (\C^*)^{\mu(X)} &\text{if $s<\mu(X)$}
\\
\bo & \text{if $s=\mu(X)$}
\\
\emptyset & \text{if $s>\mu(X)$},
\end{cases}
\end{equation}
and likewise for $\RR_s(U(\A_X))=\TC_{\bo}(\V_s(U(\A_X)))$. 

Consider now the inclusion $j_X\colon M(\A)\inj M(\A_X)$ and the 
induced monomorphism of complex algebraic groups, 
$j_X^*\colon H^1(M(\A_X));\C^*)\inj H^1(M(\A));\C^*)$. 
If $\mu(X)>s$, the map $j_X^*$ restricts to an isomorphism from 
$\V_s(M(\A_X))\cong (\C^*)^{\mu(X)}$ 
to the subtorus $T_X\subset H^1(M(\A);\C^*)$ given by
\begin{equation}
\label{eq:tx}
T_X = \Big\{ t : \prod_{H_i\in \A_X} t_i =1 \ \text{and $t_i= 1$ if  $H_i\notin \A_X$} \Big\} .
\end{equation}
Likewise, the map $j_X^*\colon H^1(M(\A_X));\C)\inj 
H^1(M(\A));\C)$ restricts to an isomorphism from $\RR_s(M(\A_X))\cong \C^{\mu(X)}$ 
to the linear subspace $L_X\subset H^1(M(\A);\C)$ given by
\begin{equation}
\label{eq:lx}
L_X = \Big\{ x : \sum_{H_i\in \A_X} x_i =0 \ \text{and $x_i= 0$ if $H_i\notin \A_X$} \Big\} .
\end{equation}

\subsection{Cohomology jump loci of decomposable arrangements}
\label{subsec:cvs-decomp}

The next result gives a complete description of the irreducible components of the 
resonance and characteristic varieties of $\Q$-decomposable arrangements, 
under a separation assumption on the rationalized Alexander invariant for the 
latter varieties.

\begin{theorem}
\label{thm:decomp-cjl}
Let $\A$ be a $\Q$-decomposable arrangement, with complement $M=M(\A)$. 
For each $s\ge 1$, 
the following hold.
\begin{enumerate}[itemsep=3pt, topsep=-1pt]
\item  \label{cj1} 
$\RR_s(M)=\bigcup_{\substack{X \in L_2(\A)\\ \mu(X)>s}}   L_X$.
\item  \label{cj2} 
If $B(\A)\otimes \Q$ is separated, then 
$\V_s(M)=\bigcup_{\substack{X \in L_2(\A)\\ \mu(X)>s}}T_X$.
\end{enumerate}
\end{theorem}

\begin{proof}
By Theorem \ref{thm:res-supp}, we have that $\RR_s(M)=\supp(\bwedge^s \B(\A)\otimes \C)$.
Likewise, for each $X\in L_2(\A)$, we have that 
$\RR_s(M(\A_X))=\supp(\bwedge^s \B(\A_X)\otimes \C)$, and this variety is either included 
in $\{\bz\}$ if $\mu(X)\le s$ or is isomorphic to $L_X$ under the embedding 
$j_X^*\colon H^1(M(\A_X);\C)\inj H^1(M;\C)$ if $\mu(X)>s$. 

Now, since $\A$ is assumed to be $\Q$-decomposable, 
it follows from Theorem \ref{thm:decomp-alex}, part \eqref{b2} 
that $\B(\A)\otimes \C\cong \B(\A)^{\loc}\otimes \C$, where recall 
$\B(\A)^{\loc}=\bigoplus_{X\in L_2(\A)} \B(\A_X)_{\bar{j}^X_*}$. Therefore, 
\begin{gather}
\label{eq:rs-dec}
\begin{aligned}
\RR_s(M)
&=\supp\big(\bwedge^s \B(\A)^{\loc}\otimes \C\big)\\
&=\bigcup_{X\in L_2(\A)} j^*_X\Big(\supp \big(\bwedge^s \B(\A_X)\otimes \C\big)\Big)\\
&=\bigcup_{\substack{X \in L_2(\A)\\ \mu(X)> s}} L_X,
\end{aligned}
\end{gather}
and this proves part \eqref{cj1}. The proof of part \eqref{cj2} is entirely similar, 
using Theorem \ref{thm:cvb} as the initial step and then 
Theorem \ref{thm:decomp-alex}, part \eqref{b3}.
\end{proof}

If Question \ref{quest:B-separated}, part \eqref{q2} were to have a positive answer, 
then we could dispense with the separation assumption in part \eqref{cj2} 
of the above theorem. {\it A priori}, though, if $\A$ is $\Q$-decomposable but 
$B(\A)\otimes \Q$ is not separated, the characteristic variety $\V_1(M)$ may 
have irreducible components that do not pass through the identity of $H^1(M;\C^*)$.  

\subsection{Chen ranks and resonance}
\label{subsec:CRC}

As an application of the above theorem and of Corollary \ref{eq:chen-decomp},  
we now show that the Chen Ranks conjecture from \cite{Su01} holds in the 
$\Q$-decom\-posable setting, in the sharpest possible range.

\begin{corollary}
\label{eq:chenranks-decomp}
Let $\A$ be a $\Q$-decomposable arrangement. 
For each $r\ge 2$, let $h_r$ be the number of irreducible components of 
$\RR_1(M(\A))$ of dimension $r$. Then 
\[
\theta_k(G(\A)) = \sum_{r\ge 2} h_r  \theta_k(F_r), 
\]
for all $k\ge 2$.
\end{corollary}

\begin{proof}
It follows from Theorem \ref{thm:decomp-cjl}, part \eqref{cj1} that the 
decomposition into irreducible components of $\RR_1(M(\A))$ is of the form 
\begin{equation}
\label{eq:r1-dec}
\RR_1(M(\A))=\bigcup_{\substack{X \in L_2(\A)\\ \mu(X)> 1}} L_X,
\end{equation}
where each component $L_X$ is a $\C$-linear 
subspace of $H^1(M(\A);\C)$ of dimension $\mu(X)>1$. 
Hence, $h_r=\#\{ X\in L_2(\A) : \mu(X)=r\}$, for all $r\ge 2$.

Now, for each $k\ge 2$, we have
\begin{gather}
\label{eq:chen-loc}
\begin{aligned}
\theta_k(G(\A)) &= \sum_{X \in L_2(\A)} \theta_k(G(\A_X)) 
&&\text{by Corollary \ref{eq:chen-decomp}}\\
&=\sum_{r\ge 2} \sum_{\substack{X \in L_2(\A)\\ \mu(X)=r}}  \theta_k(F_r)
&&\text{since $G(\A_X)\cong F_{\mu(X)}\times \Z$}\\
&=\sum_{r\ge 2} h_r  \theta_k(F_r),
\end{aligned}
\end{gather}
and this completes the proof.
\end{proof}

\section{Milnor fibrations of decomposable arrangements}
\label{sect:milnor-decomp}

In this final section, we relate the $\Q$-decomposability of an arrangement (under 
a separability assumption) to the triviality of the algebraic monodromy of its 
Milnor fibrations. 

\subsection{The Milnor fibration of a multi-arrangement}
\label{sect:mf-multi}
Let $\A$ be a central arrangement of hyperplanes in $\C^{d+1}$. 
To each hyperplane $H\in \A$, we may assign a multiplicity $m_H\in \N$.  
This yields a multi-arrangement $(\A,\bm)$, where $\bm=(m_H)_{H\in \A}$, 
and a homogeneous polynomial, 
\begin{equation}
\label{eq:fm}
f_{\bm}=\prod_{H\in \A} f_H^{m_H}
\end{equation}
of degree $N=\sum_{H\in \A} m_H$. 
The polynomial map $f_{\bm}\colon \C^{d+1} \to \C$ restricts 
to a map $f_{\bm}\colon M(\A) \to \C^{*}$.  As shown by 
Milnor \cite{Mi} in a much more general context, $f_{\bm}$ 
is the projection map of a smooth, locally trivial bundle, 
known as the {\em (global) Milnor fibration}\/ of the 
multi-arrangement $(\A,\bm)$, 
\begin{equation}
\label{eq:mfib}
\begin{tikzcd}[column sep=26pt]
F_{\bm} \ar[r]  & M \ar[r, "f_{\bm}"] & \C^*. 
\end{tikzcd}
\end{equation}

The typical fiber of this fibration, $f_{\bm}^{-1}(1)$, 
is a smooth manifold of dimension $2d$, called the {\em Milnor fiber}\/ 
of the multi-arrangement, denoted $F_{\bm}=F_{\bm} (\A)$. 
It is readily seen that $F_{\bm}$ is a Stein domain of complex 
dimension $d$, and thus has the homotopy type of a finite CW-complex 
of dimension at most $d$. Moreover, $F_{\bm}$ is connected if and only 
if $\gcd(\bm)=1$, a condition we will assume henceforth. 
In the case when all the multiplicities $m_H$ are equal to $1$, the polynomial 
$f=f_{\bm}$ is the usual defining polynomial, of degree $n=\abs{\A}$, 
and $F=F_{\bm}$ is the usual Milnor fiber of $\A$. 

\subsection{Milnor fibers as finite cyclic covers}
\label{subsec:mf cover} 
The {\em monodromy}\/ of the Milnor fibration 
is the self-diffeomorphism $h\colon F_{\bm}\to F_{\bm}$ 
given by $z \mapsto e^{2\pi \ii/N} z$.  Clearly, 
the complement $M$ is homotopy equivalent to the mapping 
torus of $h$.  The map $h$ generates a cyclic group of order 
$N=\sum_{H\in \A} m_H$  which acts freely on $F_{\bm}$. 
The quotient space, $F_{\bm}/\Z_N$, 
may be identified with the projective complement, $U=\P(M)$, in a way 
so that the projection map, $\sigma_{\bm}\colon F_{\bm} \surj F_{\bm}/\Z_N$, 
coincides with the restriction of the projectivization map, $\pi\colon M\surj U=\P(M)$, 
to the subspace $F_{\bm}$. 

The induced homomorphism on fundamental groups, 
$(f_{\bm})_{\sharp} \colon \pi_1(M)\to \pi_1(\C^*)$ may 
be identified with the map $\mu_\bm\colon \pi_1(M)\to \Z$, 
$\gamma_H \mapsto m_H$, which descends to an epimorphism, 
\begin{equation}
\label{eq:chi}
\begin{tikzcd}[column sep=22pt]
\chi_{\bm}\colon \pi_1(U)\ar[r, two heads] &\Z_N, 
\end{tikzcd}
\quad \overline{\gamma}_H \mapsto m_H \bmod N.
\end{equation}
As shown in \cite{CS95, Su-conm11, Su-toul}, the regular, 
$N$-fold cyclic cover $\sigma_{\bm}\colon F_{\bm} \to U$ is 
classified by this epimorphism. In particular, the 
usual Milnor fiber $F=F(\A)$ is classified by the ``diagonal" 
homomorphism, $\chi\colon \pi_1(U)\surj \Z_n$, given by 
$\chi(\overline{\gamma}_H )=1$, for all $H\in \A$.

\subsection{Trivial algebraic monodromy}
\label{subsec:trivial-mono}
Much effort has been put into computing the homology groups of the 
Milnor fiber $F_{\bm}$ and finding the eigenvalues of the algebraic 
monodromy $h_*$ acting on $H_*(F_{\bm};\C)$; see for instance 
\cite{Artal, CDO03, CS95, DeS-plms, PS-plms17, Su-toul,Yo20} 
and the references therein. Despite some progress, the problem 
of computing even $b_1(F(\A))$ remains open for arbitrary arrangements $\A$.

In recent work \cite{Su-mono}, we studied in depth those hyperplane arrangements 
for which the monodromy of the Milnor fibration acts trivially on the first 
homology of the Milnor fiber (either with $\Z$ or with $\Q$ coefficients). 
The description of $F_{\bm}$ as a finite cyclic cover of $U=U(\A)$ makes \
it apparent that the map $h_*\colon H_1(F_{\bm};\Q)\to H_1(F_{\bm};\Q)$
is the identity if and only if $b_1(F_{\bm})=b_1(U)$. 

The main result of this section is the next theorem, which gives a sufficient 
condition for the monodromy in degree $1$ to be trivial over $\Q$. 
The result follows from Theorem \ref{thm:decomp-cjl} and 
\cite[Proposition~4.5]{Su-mono}; for completeness, we provide a 
self-contained proof.

\begin{theorem}
\label{thm:decomp-mono}
Let $\A$ be an arrangement of rank $3$ or higher. Suppose 
$\A$ is $\Q$-decom\-posable and $B(\A)\otimes \Q$ is separated. Then, 
for any choice of multiplicities $\bm$ on $\A$, the algebraic monodromy 
of the Milnor fibration, $h_*\colon H_1(F_{\bm};\Q)\to H_1(F_{\bm};\Q)$, is trivial.
\end{theorem}

\begin{proof}
Fix an ordering $H_1,\dots, H_n$ of the hyperplanes in $\A$. 
Then $H^1(M;\C^*)$ may be identified with $(\C^*)^n$, with coordinates $t=(t_1,\dots,t_n)$ and 
$H^1(U;\C^*)$ may be identified with $(\C^*)^{n-1}$, with coordinates $(t_1,\dots,t_{n-1})$. 
The morphism $\pi^*\colon H^1(U;\C^*) \inj H^1(M;\C^*)$ 
may then be viewed as the monomial map $(\C^{*})^{n-1} \inj (\C^{*})^n$ which sends 
$(t_1,\dots, t_{n-1})$ to $(t_1,\dots, t_{n-1}, t_1^{-1}\cdots t_{n-1}^{-1})$. 

Given a choice of multiplicities $\bm$ on $\A$, consider the regular 
$\Z_N$-cover $F_{\bm} \to U$, where $N=\sum_{H\in \A} m_H$. 
By formula \eqref{eq:b1-cover}, we have that 
\begin{equation}
\label{eq:b1-fm}
b_1(F_{\bm})=\sum_{s\ge 1}\abs{\im(\chi_{\bm}^*) \cap \V_s(U)},
\end{equation}
where $\chi_{\bm}\colon \pi_1(U)\surj \Z_N$ is the homomorphism 
from \eqref{eq:chi} and $\chi_{\bm}^* \colon H^1(\Z_N;\C^*)\inj H^1(U;\C^*)$ 
is the induced morphism. The 
cyclic subgroup $\im(\chi_{\bm}^*)\subset  (\C^*)^{n-1}$ 
is generated  by the character $\rho=(\zeta^{m_1}, \dots, \zeta^{m_{n-1}})$, 
where $\zeta=e^{2 \pi \ii/N}$ and $m_i=m_{H_i}$. 
It follows that $\pi^*(\im(\chi_{\bm}^*))$ is the cyclic subgroup of 
$(\C^*)^{n}$ generated by $\tilde{\rho}=(\zeta^{m_1}, \dots, 
\zeta^{m_{n-1}}, \zeta^{m_n})$, and thus is contained in the subtorus 
$T_{\bm}=\{ (z^{m_1},\dots, z^{m_n}) \mid z\in \C^*\}\subset (\C^*)^{n}$.
 
Now let $C$ be an irreducible component of $\V_1(M)$. 
By assumption, $\A$ is $\Q$-decom\-posable and $B(\A)\otimes \Q$ is separated;
therefore, Theorem \ref{thm:decomp-cjl}, part \eqref{cj2} ensures that 
$C=T_X$, for some $2$-flat $X\in L_{2}(\A)$. Moreover, since $\A$ is assumed 
to be of rank at least $3$, we have that $\A_X$ is properly contained in $\A$, and 
thus $C$ lies in a proper coordinate subtorus of $H^1(M;\C^*)=(\C^*)^{n}$;  
hence, $C$ intersects intersects $T_{\bm}$ only at the identity. 
Consequently, $\pi^*(\im(\chi_{\bm}^*)) \cap \VV_1(M)=\{\bo\}$, and therefore  
$\im(\chi^*_{\bm}) \cap \V_1(U)=\{\bo\}$. It follows that 
$b_1(F_{\bm})=n-1$, which is equivalent to the monodromy action 
on $H_1(F_{\bm};\Q)$ being trivial.
\end{proof}

As an application of this theorem, we give a quick proof of a result of 
Venturelli, who showed in \cite[Theorem~3]{Ve} that the usual Milnor fibration 
of a certain class of arrangements has trivial algebraic monodromy 
(our proof works as well for all the fibers $F_{\bm}$). 
Let $\A$ be a central arrangement in $\C^3$, and let $\bar{\A}=\P(\A)$ 
be the projectivized arrangement of lines in $\CP^2$. 

\begin{proposition}[\cite{Ve}]
\label{prop:vent}
Suppose $\bar{\A}$ has two multiple points, $P_1$ and $P_2$, 
such that every line in $\bar{\A}$ passes through either $P_1$ or $P_2$. 
Then $h_*\colon H_1(F_{\bm};\Q)\to H_1(F_{\bm};\Q)$ is the identity. 
\end{proposition}

\begin{proof}
Let $m_i\ge 3$ be the multiplicity of $P_i$. 
The hypothesis can be rephrased as saying that $L_{\le 2}(\A) \cong L_{\le 2}(\A(m_1,m_2))$. 
By Proposition \ref{prop:pend-b-dec}, the Alexander invariant $B(\A)$ is decomposable and 
separated. The claim now follows from Theorem \ref{thm:decomp-mono}.
\end{proof}

\subsection{Discussion}
\label{subsec:mono-discuss}
It is natural to ask whether one can work over the integers in 
Theorem \ref{thm:decomp-mono}, and also whether one can dispense 
with the separation hypothesis on the Alexander invariant (either over 
$\Z$ or over $\Q$). More precisely, we have the following question. 

\begin{question}
\label{quest:mono-dec}
Let $(\A,\bm)$ be a multi-arrangement, and let $h\colon F_{\bm} \to F_{\bm}$ 
be the monodromy of the corresponding Milnor fibration.
\begin{enumerate}[itemsep=2pt, topsep=-1pt]
\item \label{mon1} 
If $\A$ is decomposable, is the monodromy action on 
$H_1(F_{\bm};\Z)$ trivial?
\item \label{mon2} 
If $\A$ is decomposable over $\Q$, is the monodromy action on 
$H_1(F_{\bm};\Q)$ trivial?
\end{enumerate}
\end{question}

If Question \ref{quest:B-separated}, part \eqref{q2} were to have a positive 
answer, then, by Theorem \ref{thm:decomp-mono}, the answer to 
part \eqref{mon2} of the above question would be yes. In general, the 
group $H_1(F_{\bm};\Z)$ may have torsion (see \cite{CDS, DeS-plms}), 
even for the usual Milnor fiber $F=F(\A)$ when $m_H=1$ for all $H\in \A$ 
(see \cite{Yo20}); thus, the monodromy $h$ may act trivially on $H_1(F_{\bm};\Q)$  
but not on $H_1(F_{\bm};\Z)$. Nevertheless, we do not know whether 
this can happen within the class of ($\Q$-)decomposable arrangements, 
though we suspect that it cannot. 

We conclude with an example worth pondering.

\begin{example}
\label{ex:pappus}
Let $\A$ be the non-Pappus arrangement of $9$ hyperplanes in $\C^3$ realizing the 
$(9_3)_2$ configuration described in Section \ref{subsec:dec-other}. This arrangement is 
decomposable, yet we do not know whether the Alexander invariant $B(\A)$ is separated, 
even over $\Q$. Nevertheless, as first shown in \cite{Artal} and later by different methods 
in \cite{CS95}, the monodromy of $\A$ acts trivially on $H_1(F;\Q)$, that is, $b_1(F)=8$. 
In fact, as noted in \cite[Example~10.10]{Su02}, all the components of the characteristic 
variety $\V_1(M(\A))$ are of the form $T_X$, for some $X\in \widetilde{L}_2(\A)$; hence, 
by the argument in the proof of Theorem \ref{thm:decomp-mono}, this also gives 
$b_1(F)=8$. 

On the other hand, the closely related Pappus arrangement, $\A'$  
(a realization the $(9_3)_1$ configuration), has intersection lattice with the same M\"{o}bius 
function as $\A$, yet $L(\A')\not\cong L(\A)$. Moreover, $\A'$ is not $\Q$-decomposable, and,  
as shown in \cite{Artal, CS95}, $b_1(F')=10$. 
\end{example}


\newcommand{\arxiv}[1]
{\texttt{\href{http://arxiv.org/abs/#1}{arXiv:#1}}}

\newcommand{\arx}[1]
{\texttt{\href{http://arxiv.org/abs/#1}{arXiv:}}
\texttt{\href{http://arxiv.org/abs/#1}{#1}}}

\newcommand{\arxx}[1]
{\texttt{\href{https://arxiv.org/abs/math/#1}{arXiv:#1}}}

\newcommand{\doi}[1]
{\texttt{\href{http://dx.doi.org/#1}{doi:\nolinkurl{#1}}}}

\renewcommand{\MR}[1]
{\href{http://www.ams.org/mathscinet-getitem?mr=#1}{MR#1}}


\begin{thebibliography}{00}

\bibitem{Ar}  D.~Arapura, 
{\em Geometry of cohomology support loci for local systems {\rm I}}, 
J. Algebraic Geom. \textbf{6} (1997), no.~3, 563--597.  
\MR{1487227}

\bibitem{Artal}
E.~Artal Bartolo,
\href{https://doi.org/10.2307/2160412}%
{\em Combinatorics and topology of line arrangements in the complex projective 
plane}, Proc. Amer. Math. Soc. \textbf{121} (1994), no.~2, 385--390.
\MR{1189536}

\bibitem{ACCM-2007} 
E.~Artal Bartolo, J.~Carmona Ruber, 
J.I.~Cogolludo-Agust\'{\i}n,  M.~Marco-Buzun\'ariz, 
\href{https://doi.org/10.2969/aspm/04310001}%
{\em Invariants of combinatorial line arrangements and 
Rybnikov's example},  in: {\em Singularity Theory and its Applications}, 
1--34, Adv. Studies in Pure Math., vol.~43,  Math. Soc. Japan, 
Tokyo, 2007.
\MR{2313406}

\bibitem{AGV20} E.~Artal Bartolo, B.~Guerville-Ball\'e, J.~Viu-Sos, 
\href{http://doi.org/10.1080/10586458.2018.1428131}%
{\em Fundamental groups of real arrangements and torsion in 
the lower central series quotients}, 
Exp. Math. \textbf{29} (2020), no.~1, 28--35. 
\MR{4067904}

\bibitem{AM} M.F.~Atiyah, I.G.~Macdonald, 
{\em Introduction to commutative algebra},
Addison-Wesley, Reading, MA, 1969.
\MR{0242802}

\bibitem{Br}  E.~Brieskorn, 
\href{https://doi.org/10.1007/BFb0069274}%
{\em Sur les groupes de tresses (d'apr\`es V. I. Arnol'd)},  
S\'eminaire Bourbaki, 24\`eme ann\'ee (1971/1972), Exp. No. 401, 
pp. 21--44, Lecture Notes in Math., vol.~317, Springer, Berlin, 1973. 
\MR{0422674}

\bibitem{Chen51} K.-T.~Chen, 
\href{https://dx.doi.org/10.2307/1969316}%
{\emph{Integration in free groups}}, Ann. of Math. (2) \textbf{54} 
(1951), no.~1, 147--162.
\MR{0042414} 

\bibitem{Chen73} K.-T.~Chen, 
\href{http://dx.doi.org/10.2307/1970846}%
{\em Iterated integrals of differential forms and loop space 
homology}, Ann. of Math. (2) \textbf{97} (1973), 217--246. 
\MR{0380859}

\bibitem{CDP} A.D.R.~Choudary, A.~Dimca, \c{S}.~Papadima, 
\href{https://doi.org/10.2140/agt.2005.5.691}%
{\em Some analogs of Zariski's Theorem on nodal line arrangements}, 
Algebr. Geom. Topol. \textbf{5} (2005), 691--711.
\MR{2153112}

\bibitem{CDS} D.C.~Cohen, G.~Denham, A.I.~Suciu,
\href{https://doi.org/10.2140/agt.2003.3.511}%
{\em Torsion in Milnor fiber homology}, 
Algebr. Geom. Topology \textbf{3} (2003), 511--535. 
\MR{1997327}

\bibitem{CDO03}  D.C.~Cohen, A.~Dimca, P.~Orlik,
\href{https://doi.org/10.5802/aif.1994}%
{\em Nonresonance conditions for arrangements}, 
Ann. Inst. Fourier (Grenoble) \textbf{53} (2003), 
no.~6, 1883--1896. 
\MR{2038782}

\bibitem{CF21}  D.C.~Cohen, M.~Falk,   
\href{https://doi.org/10.1016/j.aam.2020.102130}%
{\em On graphic arrangement groups},  Adv. in Appl. Math. 
\textbf{126} (2021), Art. ID 102130, 22 pp.
\MR{4224066}

\bibitem{CFR20}  D.C.~Cohen, M.~Falk, R.~Randell, 
\href{https://doi.org/10.1007/s40879-020-00412-1}%
{\em Discriminantal bundles, arrangement groups, and 
subdirect products of free groups}, Eur. J. Math. 
\textbf{6} (2020), nr.~3, 751--789.
\MR{4151718}

\bibitem{CS95} D.C.~Cohen, A.I.~Suciu, 
\href{https://doi.org/10.1112/jlms/51.1.105}%
{\em On {M}ilnor fibrations of arrangements},  
J. London Math. Soc. (2) \textbf{51} (1995), no.~1, 105--119.
\MR{1310725} 

\bibitem{CS-conm95} D.C.~Cohen, A.I.~Suciu, 
\href{http://dx.doi.org/10.1090/conm/181/02029}%
{\emph{The {C}hen groups of the pure braid group}}, 
in: {\em The \v {C}ech centennial} ({B}oston, {MA}, 1993), 45--64, 
Contemp. Math., vol. 181, Amer. Math. Soc., Providence, RI, 1995.
\MR{1320987}

\bibitem{CS-cmh97}  D.C.~Cohen, A.I.~Suciu,
\href{https://dx.doi.org/10.1007/s000140050017}%
{\em The braid monodromy of plane algebraic curves and 
hyperplane arrangements}, Comment. Math. Helv.
\textbf{72} (1997), no.~2, 285--315. 
\MR{1470093}

\bibitem{CS-tams99} D.C.~Cohen, A.I.~Suciu, 
\href{http://dx.doi.org/10.1090/S0002-9947-99-02206-0}%
{\em Alexander invariants of complex hyperplane arrangements}, 
Trans. Amer. Math. Soc. \textbf{351} (1999), no.~10, 4043--4067. 
\MR{1475679}

\bibitem{CS99} D.C.~Cohen, A.I.~Suciu, 
\href{https://doi.org/10.1017/S0305004199003576}%
{\em Characteristic varieties of arrangements},
Math. Proc. Cambridge Phil. Soc. \textbf{127} (1999), 
no.~1, 33--53. 
\MR{1692519} 

\bibitem{DSS} G.~Denham, G.G.~Smith, A.~Steiner, 
{\em HyperplaneArrangements: manipulating finite sets of hyperplanes. Version~2.0}, 
A \emph{Macaulay2} package available at 
\url{https://github.com/Macaulay2/M2/tree/master/M2/Macaulay2/packages}.

\bibitem{DeS-plms} G.~Denham, A.I.~Suciu,
\href{https://doi.org/10.1112/plms/pdt058}%
{\em Multinets, parallel connections, and {M}ilnor fibrations  
of arrangements}, Proc. London Math. Soc. 
\textbf{108} (2014), no.~6, 1435--1470.
\MR{3218315}

\bibitem{DSY16} G.~Denham, A.I.~Suciu, and S.~Yuzvinsky, 
\href{https://dx.doi.org/10.1007/s00029-015-0196-8}%
{\em Combinatorial covers and vanishing of cohomology}, 
Selecta Math. (N.S.) \textbf{22} (2016), no.~2, 561--594. 
\MR{3477330}

\bibitem{DP-ann} A.~Dimca, S.~Papadima, 
\href{https://doi.org/10.4007/annals.2013.177.2.1}%
{\em Arithmetic group symmetry and finiteness properties of 
{T}orelli groups}, Ann. Math. \textbf{177} (2013), no.~2, 395--423. 
\MR{3010803}

\bibitem{DPS-serre} A.~Dimca, S.~Papadima, A.I.~Suciu,  
{\em Formality, {A}lexander invariants, and a question of {S}erre}, 
unpublished manuscript, \arxiv{math.AT/0512480v3}.

\bibitem{DPS-duke} A.~Dimca, \c{S}.~Papadima, A.~Suciu,
\href{https://dx.doi.org/10.1215/00127094-2009-030}%
{\em Topology and geometry of cohomology jump loci}, 
Duke Math. Journal \textbf{148} (2009), no.~3, 405--457.
\MR{2527322} 

\bibitem{Fa88} M.~Falk, 
\href{https://doi.org/10.2307/2000924}%
{\em The minimal model of the complement of an arrangement of hyperplanes},  
Trans. Amer. Math. Soc. \textbf{309} (1988), no.~2, 543--556. 
\MR{0929668}

\bibitem{Fa89} M.~Falk,
\href{https://doi.org/10.1090/conm/090/1000594}%
{\em The cohomology and fundamental group of a hyperplane
complement}, Singularities (Iowa City, IA, 1986), 55--72, 
Contemp. Math., vol.~90, Amer. Math. Soc., Providence, RI, 1989.
\MR{1000594}

\bibitem{Fa97} M.~Falk,
\href{https://doi.org/10.1007/BF02558471}%
{\em Arrangements and cohomology},
Ann. Combin. \textbf{1} (1997), no.~2, 135--157.  
\MR{1629681} 

\bibitem{FY} M.~Falk, S.~Yuzvinsky,
\href{https://doi.org/10.1112/S0010437X07002722}%
{\em Multinets, resonance varieties, and pencils of plane curves},
Compositio Math. \textbf{143} (2007), no.~4, 1069--1088.
\MR{2339840} 

\bibitem{Fox} R.H.~Fox,
\href{https://doi.org/10.2307/1969736}%
{\em Free differential calculus. \textup{I}. Derivation in the
free group ring}, Ann. of Math. \textbf{57} (1953), 547--560.
\MR{0053938}

\bibitem{Hi97} E.~Hironaka,
\href{https://doi.org/10.5802/aif.1573}%
{\em Alexander stratifications of character varieties}, Annales 
de l'Institut Fourier (Grenoble) \textbf{47} (1997), no.~2, 555--583.
\MR{1450425}

\bibitem{JY93} T.~Jiang, S.S.-T.~Yau, 
\href{https://doi.org/10.1090/S0273-0979-1993-00409-9}%
{\em Topological invariance of intersection lattices of arrangements 
in $\CP^2$}, Bull. Amer. Math. Soc. \textbf{29} (1993), no.~1, 88--93. 
\MR{1197426} 

\bibitem{JY94} T.~Jiang, S.S.-T.~Yau, 
\href{http://www.numdam.org/item?id=CM_1994__92_2_133_0}%
{\em Diffeomorphic types of the complements of arrangements 
of hyperplanes}, Compositio Math. \textbf{92} (1994), no.~2, 133--155.
\MR{1283226}

\bibitem{Kohno-83} T.~Kohno,
\href{https://doi.org/10.1017/S0027763000020547}%
{\em On the holonomy Lie algebra and the nilpotent completion
of the fundamental group of the complement of hypersurfaces},
Nagoya Math. J. \textbf{92} (1983), 21--37.  
\MR{0726138}

\bibitem{Li92} A.~Libgober,
\href{https://doi.org/10.1016/0166-8641(92)90137-O}%
{\em On the homology of finite abelian coverings},
Topology Appl. \textbf{43} (1992), no.~2, 157--166.
\MR{1152316} 

\bibitem{LY00} A.~Libgober, S.~Yuzvinsky,
\href{https://doi.org/10.1023/A:1001826010964}%
{\em Cohomology of {O}rlik--{S}olomon algebras and local 
systems},  Compositio Math. \textbf{21} (2000), no.~3, 337--361.  
\MR{1761630} 

\bibitem{LS09}  P.~Lima-Filho, H.~Schenck, 
\href{https://doi.org/10.1093/imrn/rnn163}%
{\em Holonomy Lie algebras and the LCS formula for 
graphic arrangements}, Intern. Math. Research Notices 
\textbf{2009} (2009), no.~8, 1421--1432.
\MR{2496769}

\bibitem{Lofwall-16}  C.~L\"{o}fwall, 
\href{https://doi.org/10.1080/00927872.2015.1100303}%
{\em Decomposition theorems for a generalization of the holonomy {L}ie 
algebra of an arrangement}, Comm. Algebra \textbf{44} (2016), no.~11, 4654--4663. 
\MR{3512533}

\bibitem{Lofwall-20} C.~L\"{o}fwall,  
{\em The holonomy Lie algebra of a matroid}, \arxiv{2012.12044}.

\bibitem{LL} C.~L\"{o}fwall,  S,~Lundqvist, 
\href{https://doi.org/10.2140/jsag.2021.11.9}%
{\em Software for doing computations in graded Lie algebras},
J. Softw. Algebra Geom. \textbf{11} (2021), 9--14.
\MR{4285760}

\bibitem{Markl-Papadima} M.~Markl, \c{S}.~Papadima,
\href{https://doi.org/10.5802/aif.1315}%
{\emph{Homotopy {L}ie algebras and fundamental groups via 
deformation theory}}, Ann. Inst.  Fourier (Grenoble) \textbf{42} 
(1992), no.~4, 905--935. 
\MR{1196099}

\bibitem{Ms-80}  W.S.~Massey,
\href{https://dx.doi.org/10.1215/S0012-7094-80-04724-9}%
{\em Completion of link modules}, Duke Math. J.
\textbf{47} (1980), no.~2, 399--420.
\MR{0575904}

\bibitem{MS00}  D.~Matei, A.I.~Suciu,
\href{http://dx.doi.org/10.2969/ASPM/02710185}%
{\em Cohomology rings and nilpotent quotients of real and
complex arrangements}, in: Arrange\-ments--Tokyo 1998, 
185--215, Adv. Stud. Pure Math., vol.~27, 
Math. Soc. Japan, Tokyo, 2000.
\MR{1796900}

\bibitem{MS02}  D.~Matei, A.I.~Suciu,
\href{http://dx.doi.org/10.1155/S107379280210907X}%
{\em Hall invariants, homology of subgroups, and characteristic 
varieties}, Internat. Math. Res. Notices \textbf{2002} (2002), 
no.~9, 465--503.
\MR{1884468} 

\bibitem{Mi} J.~Milnor, 
\href{https://press.princeton.edu/books/paperback/9780691080659/singular-points-of-complex-hypersurfaces-am-61-volume-61}%
{\em Singular points of complex hypersurfaces},
Annals of Math. Studies, vol.~61, Princeton Univ. 
Press, Princeton, NJ, 1968.
\MR{0239612} 

\bibitem{OS} P.~Orlik, L.~Solomon,
\href{https://doi.org/10.1007/BF01392549}%
{\em Combinatorics and topology of complements of
hyperplanes}, Invent. Math. \textbf{56} (1980), no.~2, 167--189. 
\MR{0558866} 

\bibitem{Orlik-Terao}  P.~Orlik and H.~Terao,
\href{https://doi.org/10.1007/978-3-662-02772-1}%
{\em Arrangements of hyperplanes}, Grundlehren Math. Wiss., 
vol.~300, Springer-Verlag, Berlin, 1992.
\MR{1217488}

\bibitem{PS-imrn04} S.~Papadima, A.I.~Suciu, 
\href{http://dx.doi.org/10.1155/S1073792804132017}%
{\em Chen {L}ie algebras}, Int. Math. Res. Not. 
\textbf{2004} (2004), no.~21, 1057--1086. 
\MR{2037049}

\bibitem{PS-cmh06} S.~Papadima and A.I. Suciu,
\href{https://dx.doi.org/10.4171/CMH/77}%
{\em When does the associated graded Lie algebra of an arrangement
group decompose?}, Comment. Math. Helv. \textbf{81} (2006), 859--875.
\MR{2271225}

\bibitem{PS-plms17} S.~Papadima, A.I.~Suciu,
\href{https://dx.doi.org/10.1112/plms.12027}%
{\em The Milnor fibration of a hyperplane arrangement: from 
modular resonance to algebraic monodromy}, Proc. London 
Math. Soc.  \textbf{114} (2017), no.~6, 961--1004.
\MR{3661343}

\bibitem{PrS20} R.~D.~Porter, A.~I.~Suciu, 
\href{https://doi.org/10.1007/s40879-019-00392-x}%
{\em Homology, lower central series, and hyperplane arrangements},
Eur. J. Math. \textbf{6} (2020), no.~3, 1039--1072.
\MR{4151728}

\bibitem{Ra97}  R.~Randell,
\href{https://doi.org/10.1016/S0166-8641(96)00123-X}
{\em Homotopy and group cohomology of arrangements},
Topology Appl. \textbf{78} (1997), 201--213.
\MR{1454600}

\bibitem{Rybnikov-1} G.~Rybnikov,
\href{https://doi.org/10.1007/s10688-011-0015-8}%
{\em On the fundamental group of the complement of a complex
hyperplane arrangement},  Funct. Anal. Appl.  
\textbf{45} (2011), no.~2, 137--148. 
\MR{2848779}

\bibitem{Rybnikov-2} G.~Rybnikov, 
\emph{On the fundamental group and triple Massey's product},
preprint 1998, \arxx{9805061}.

\bibitem{Sa95} M.~Sakuma,
\href{https://doi.org/10.4153/CJM-1995-010-2}%
{\em Homology of abelian coverings of links and spatial graphs},
Canad. J. Math. \textbf{47} (1995), no.~1, 201--224.
\MR{1319696}

\bibitem{SS-tams02} H.~Schenck, A.I.~Suciu, 
\href{http://dx.doi.org/10.1090/S0002-9947-02-03021-0}%
{\em Lower central series and free resolutions of hyperplane 
arrangements}, Trans. Amer. Math. Soc. \textbf{354} (2002), 
no.~9, 3409--3433. 
\MR{1911506}

\bibitem{Su01} A.I.~Suciu,
\href{https://dx.doi.org/10.1090/conm/276/04510}%
{\em Fundamental groups of line arrangements: 
Enumerative aspects}, in: Advances in algebraic geometry 
motivated by physics (Lowell, MA, 2000), pp. 43--79, Contemp. 
Math., vol 276, Amer. Math. Soc., Providence, RI, 2001.
\MR{1837109} 

\bibitem{Su02} A.I.~Suciu,
\href{https://dx.doi.org/10.1016/S0166-8641(01)00052-9}%
{\em Translated tori in the characteristic varieties of
complex hyperplane arrangements}, Topology Appl. 
\textbf{118} (2002), no.~1-2, 209--223.  
\MR{1877726} 

\bibitem{Su-conm11} A.I.~Suciu,
\href{https://dx.doi.org/10.1090/conm/538/10600}%
{\em Fundamental groups, Alexander invariants, and 
cohomology jumping loci},  in: Topology of 
algebraic varieties and singularities, 179--223, Contemp. 
Math., vol. 538, Amer. Math. Soc., Providence, RI, 2011. 
\MR{2777821} 

\bibitem{Su-imrn} A.I.~Suciu,
\href{https://dx.doi.org/10.1093/imrn/rns246}%
{\em Characteristic varieties and Betti numbers of free
abelian covers}, Int. Math. Res. Notices \textbf{2014} 
(2014), no. 4, 1063--1124.
\MR{3168402}

\bibitem{Su-toul} A.I.~Suciu,
\href{https://dx.doi.org/10.5802/afst.1412}%
{\em Hyperplane arrangements and {M}ilnor fibrations}, 
Ann. Fac. Sci. Toulouse Math. \textbf{23} (2014), no.~2, 417--481.  
\MR{3205599}

\bibitem{Su-abexact} A.I.~Suciu,
\href{https://doi.org/10.2422/2036-2145.202112_005}%
{\em Alexander invariants and cohomology jump loci in group extensions}, 
Ann. Sc. Norm. Super. Pisa Cl. Sci. \textbf{25} (2024), no.~2, 1085--1154.
\MR{4778471}

\bibitem{Su-mono} A.I.~Suciu,
\href{https://doi.org/10.59277/RRMPA.2024.235.293}%
{\em Milnor fibrations of arrangements with trivial algebraic monodromy}, 
Rev. Roumaine Math. Pures Appl. \textbf{69} (2024), no.~2, 235--293.
\MR{4778442}

\bibitem{SW-mz}  A.I.~Suciu, H.~Wang, 
\href{https://dx.doi.org/10.1007/s00209-016-1811-x}%
{\em Pure virtual braids, resonance, and formality}, 
Math. Z.  \textbf{286} (2017), no.~3--4, 1495--1524.
\MR{3671586}

\bibitem{SW-jpaa}  A.I.~Suciu, H.~Wang, 
\href{https://doi.org/10.1016/j.jpaa.2018.11.006}
{\em Cup products, lower central series, and holonomy Lie algebras}, 
J. Pure Appl. Algebra \textbf{223} (2019), no.~8, 3359--3385.
\MR{3926216}

\bibitem{SW-forum} A.I.~Suciu, H.~Wang, 
\href{https://doi.org/10.1515/forum-2018-0098}%
{\em Formality properties of finitely generated groups and Lie 
algebras}, Forum Math. \textbf{31} (2019), no.~4, 867--905. 
\MR{3975666}

\bibitem{Sullivan} D.~Sullivan, 
\href{https://dx.doi.org/10.1007/BF02684341}%
{\em Infinitesimal computations in topology}, Inst. Hautes \'Etudes 
Sci. Publ. Math. (1977), no.~47, 269--331. 
\MR{0646078}

\bibitem{Ve} F.~Venturelli, 
{\em Gysin morphisms for non-transversal hyperplane sections with 
an application to line arrangements}, \arxiv{1912.11681}.

\bibitem{Yo20}  M.~Yoshinaga, 
\href{https://doi.org/10.1007/s40879-019-00387-8}%
{\em Double coverings of arrangement complements and $2$-torsion 
in Milnor fiber homology},  Eur. J. Math \textbf{6} (2020), nr.~3, 1097--1109.
\MR{4151730}

\end{thebibliography}
\end{document}